\documentclass[12pt,oneside,canadian,english]{amsart}
\usepackage[T1]{fontenc}
\usepackage[latin9]{inputenc}
\usepackage[a4paper]{geometry}
\usepackage{textcomp}
\usepackage{amsthm}
\usepackage{amstext}
\usepackage{amssymb}
\usepackage{graphicx}
\usepackage{setspace}
\usepackage{esint}
\makeatletter

\newcommand{\lyxmathsym}[1]{\ifmmode\begingroup\def\b@ld{bold}
  \text{\ifx\math@version\b@ld\bfseries\fi#1}\endgroup\else#1\fi}

\numberwithin{equation}{section}
\newenvironment{lyxlist}[1]
{\begin{list}{}
{\settowidth{\labelwidth}{#1}
 \setlength{\leftmargin}{\labelwidth}
 \addtolength{\leftmargin}{\labelsep}
 }}
{\end{list}}
\theoremstyle{plain}
\newtheorem{thm}{\protect\theoremname}[section]
  \theoremstyle{plain}
  \newtheorem{prop}[thm]{\protect\propositionname}
  \theoremstyle{plain}
  \newtheorem{lem}[thm]{\protect\lemmaname}
  \theoremstyle{plain}
  \newtheorem{cor}[thm]{\protect\corollaryname}
  \theoremstyle{remark}
  \newtheorem{rem}[thm]{\protect\remarkname}

\numberwithin{equation}{section}
\theoremstyle{plain}
\newtheorem*{thma}{Theorem A}
\newtheorem*{thmb}{Theorem B}

\makeatother

\usepackage{babel}
  \addto\captionscanadian{\renewcommand{\corollaryname}{Corollary}}
  \addto\captionscanadian{\renewcommand{\lemmaname}{Lemma}}
  \addto\captionscanadian{\renewcommand{\propositionname}{Proposition}}
  \addto\captionscanadian{\renewcommand{\remarkname}{Remark}}
  \addto\captionscanadian{\renewcommand{\theoremname}{Theorem}}
  \addto\captionsenglish{\renewcommand{\corollaryname}{Corollary}}
  \addto\captionsenglish{\renewcommand{\lemmaname}{Lemma}}
  \addto\captionsenglish{\renewcommand{\propositionname}{Proposition}}
  \addto\captionsenglish{\renewcommand{\remarkname}{Remark}}
  \addto\captionsenglish{\renewcommand{\theoremname}{Theorem}}
  \providecommand{\corollaryname}{Corollary}
  \providecommand{\lemmaname}{Lemma}
  \providecommand{\propositionname}{Proposition}
  \providecommand{\remarkname}{Remark}
\providecommand{\theoremname}{Theorem}

\begin{document}

\title{The unstable set of a periodic orbit for delayed positive feedback}

\maketitle
\begin{center}
{\large{Tibor Krisztin}}%
\footnote{e-mail: krisztin@math.u-szeged.hu%
},{\large{ Gabriella Vas}}%
\footnote{e-mail: vasg@math.u-szeged.hu%
} 
\par\end{center}

\begin{center}
MTA-SZTE Analysis and Stochastic Research Group, Bolyai Institute,
University of Szeged, Hungary
\par\end{center}

\noindent \textbf{~}
\begin{abstract}
\noindent In the paper {[}Large-amplitude periodic solutions for differential
equations with delayed monotone positive feedback, JDDE 23 (2011),
no. 4, 727\textendash{}790{]}, we have constructed large-amplitude
periodic orbits for an equation with delayed monotone positive feedback.
We have shown that the unstable sets of the large-amplitude periodic
orbits constitute the global attractor besides spindle-like structures.
In this paper we focus on a large-amplitude periodic orbit $\mathcal{O}_{p}$
with two Floquet multipliers outside the unit circle, and we intend
to characterize the geometric structure of its unstable set $\mathcal{W}^{u}\left(\mathcal{O}_{p}\right)$.
We prove that $\mathcal{W}^{u}\left(\mathcal{O}_{p}\right)$ is a
three-dimensional $C^{1}$-submanifold of the phase space and admits
a smooth global graph representation. Within $\mathcal{W}^{u}\left(\mathcal{O}_{p}\right)$,
there exist heteroclinic connections from $\mathcal{O}_{p}$ to three
different periodic orbits. These connecting sets are two-dimensional
$C^{1}$-submanifolds of $\mathcal{W}^{u}\left(\mathcal{O}_{p}\right)$
and homeomorphic to the two-dimensional open annulus. They form $C^{1}$-smooth
separatrices in the sense that they divide the points of $\mathcal{W}^{u}\left(\mathcal{O}_{p}\right)$
into three subsets according to their $\omega$-limit sets.
\end{abstract}
\noindent ~


~

\noindent \textbf{Suggested running head:} The unstable set of a periodic
orbit

\noindent ~


\section{Introduction}

Consider the delay differential equation
\begin{equation}
\dot{x}\left(t\right)=-\mu x\left(t\right)+f\left(x\left(t-1\right)\right),\label{eq:eq_general}
\end{equation}
 where $\mu$ is a positive constant and $f:\mathbb{R}\rightarrow\mathbb{R}$
is a smooth monotone nonlinearity.

The natural phase space for Eq.\,\eqref{eq:eq_general} is $C=C\left(\left[-1,0\right],\mathbb{R}\right)$
equipped with the supremum norm. For any $\varphi\in C$, there is
a unique solution $x^{\varphi}:\left[-1,\infty\right)\rightarrow\mathbb{R}$
of \eqref{eq:eq_general}. For each $t\geq0$, $x_{t}^{\varphi}\in C$
is defined by $x_{t}^{\varphi}\left(s\right)=x^{\varphi}\left(t+s\right)$,
$-1\leq s\leq0$. Then the map 
\[
\Phi:\left[-1,\infty\right)\times C\ni\left(t,\varphi\right)\mapsto x_{t}^{\varphi}\in C
\]
 is a continuous semiflow. 

In \cite{Krisztin-Vas}, the authors of this paper have studied Eq.\,\eqref{eq:eq_general}
under the subsequent hypothesis:
\begin{lyxlist}{00.00.0000}
\item [{{(H1)}}] $\mu>0$, $f\in C^{1}\left(\mathbb{R},\mathbb{R}\right)$
with $f'\left(\xi\right)>0$ for all $\xi\in\mathbb{R}$, and 
\[
\xi_{-2}<\xi_{-1}<\xi_{0}=0<\xi_{1}<\xi_{2}
\]
 are five consecutive zeros of $\mathbb{R}\ni\xi\mapsto-\mu\xi+f\left(\xi\right)\in\mathbb{R}$
with $f'\left(\xi_{j}\right)<\mu<f'\left(\xi_{k}\right)$ for $j\in\left\{ -2,0,2\right\} $
and $k\in\left\{ -1,1\right\} $ (see Fig.\,1). 
\end{lyxlist}
\begin{center}
\begin{figure}[h]
\begin{centering}
\includegraphics[scale=0.5]{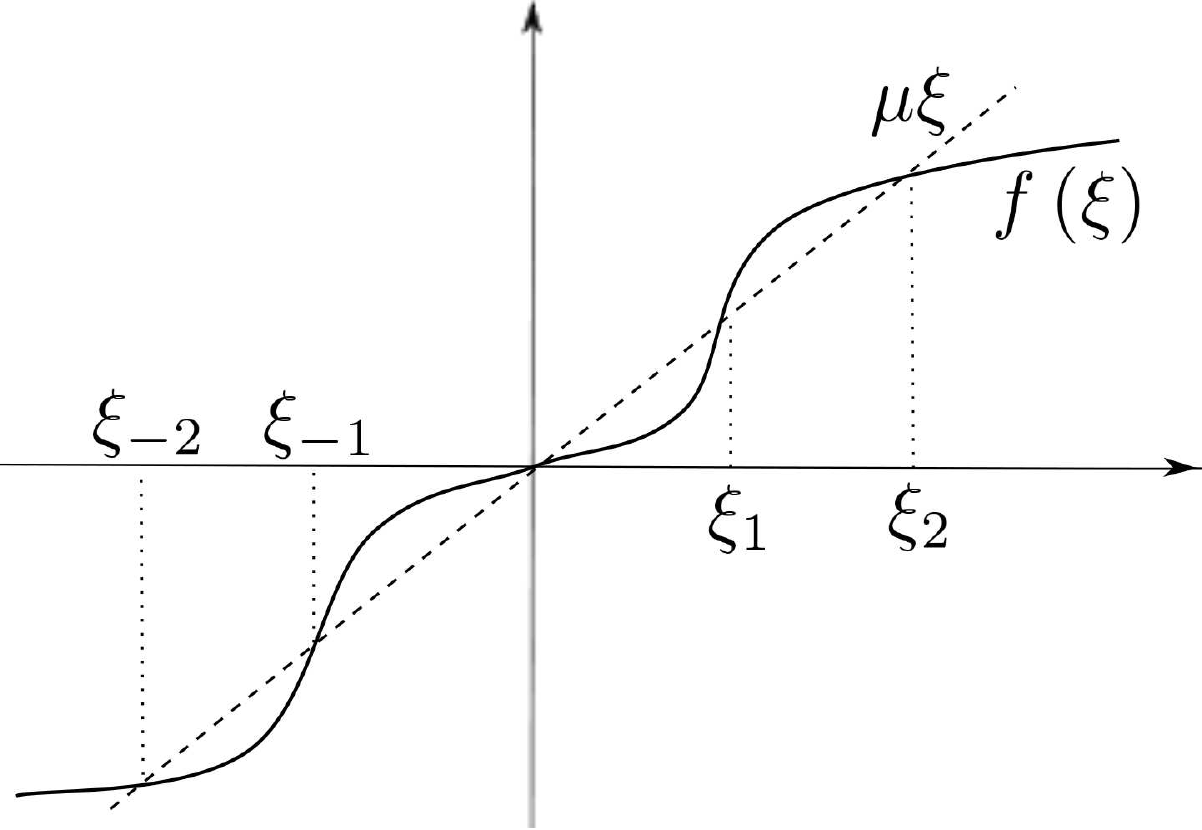} 
\par\end{centering}

\caption{A feedback function satisfying condition (H1). }

\end{figure}

\par\end{center}

Under hypothesis (H1), $\hat{\xi}_{j}\in C$, defined by $\hat{\xi}_{j}\left(s\right)=\xi_{j}$,
$-1\leq s\leq0$, is an equilibrium point of $\Phi$ for all $j\in\left\{ -2,-1,0,1,2\right\} $,
furthermore $\hat{\xi}_{-2},$ $\hat{\xi}_{0}$ and $\hat{\xi}_{2}$
are stable, and $\hat{\xi}_{-1}$ and $\hat{\xi}_{1}$ are unstable.
By the monotone property of $f$, the subsets 
\[
C_{-2,2}=\left\{ \varphi\in C:\,\xi_{-2}\leq\varphi\left(s\right)\leq\xi_{2}\mbox{ for all }s\in\left[-1,0\right]\right\} ,
\]
\[
C_{-2,0}=\left\{ \varphi\in C:\,\xi_{-2}\leq\varphi\left(s\right)\leq0\mbox{ for all }s\in\left[-1,0\right]\right\} ,
\]
\[
C_{0,2}=\left\{ \varphi\in C:\,0\leq\varphi\left(s\right)\leq\xi_{2}\mbox{ for all }s\in\left[-1,0\right]\right\} 
\]
 of the phase space $C$ are positively invariant under the semiflow
$\Phi$ (see Proposition \ref{pro:monotone_dynamical_system} in Section
\ref{sec:Prelimimaries}).

Let $\mathcal{A}$, $\mathcal{A}_{-2,0}$ and $\mathcal{A}_{0,2}$
denote the global attractors of the restrictions $\Phi|_{\left[0,\infty\right)\times C_{-2,2}}$,
$\Phi|_{\left[0,\infty\right)\times C_{-2,0}}$ and $\Phi|_{\left[0,\infty\right)\times C_{0,2}}$,
respectively. If (H1) holds and $\xi_{-2},\xi_{-1},0,\xi_{1},\xi_{2}$
are the only zeros of $-\mu\xi+f\left(\xi\right)$, then $\mathcal{A}$
is the global attractor of $\Phi$. The structures of $\mathcal{A}_{-2,0}$
and $\mathcal{A}_{0,2}$ are (at least partially) well understood,
see e.g.~\cite{Krisztin-1,Krisztin-2,Krisztin-3,Krisztin-Walther,Krisztin-Walther-Wu,Krisztin-Wu}.
$\mathcal{A}_{-2,0}$ and $\mathcal{A}_{0,2}$ admit Morse decompositions
\cite{Polner}. Further technical conditions regarding $f$ ensure
that $\mathcal{A}_{-2,0}$ and $\mathcal{A}_{0,2}$ have spindle-like
structures \cite{Krisztin-1,Krisztin-Walther,Krisztin-Walther-Wu,Krisztin-Wu}:
$\mathcal{A}_{0,2}$ is the closure of the unstable set of $\hat{\xi}_{1}$
containing the equilibrium points $\hat{\xi}_{0}$, $\hat{\xi}_{1}$,
$\hat{\xi}_{2}$, periodic orbits in $C_{0,2}$ and heteroclinic orbits
among them. In other cases $\mathcal{A}_{0,2}$ is larger than the
the closure of the unstable set of $\hat{\xi}_{1}$. The structure
of $\mathcal{A}_{-2,0}$ is similar. See Fig. 2 for a simple situation.

\begin{center}
\begin{figure}[h]
\begin{centering}
\includegraphics[width=5cm,height=6cm]{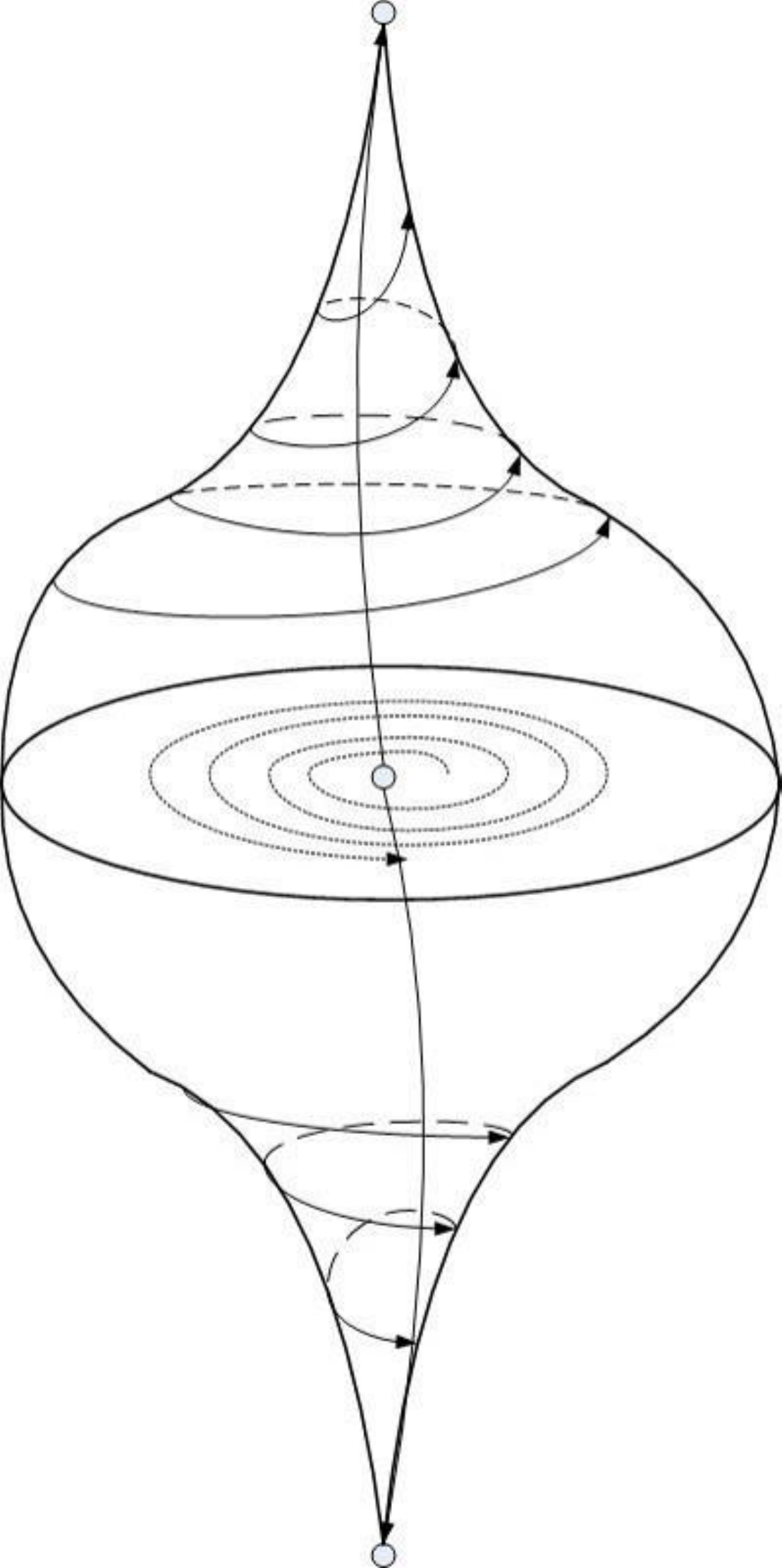} 
\par\end{centering}

\caption{A spindle-like structure }
\end{figure}

\par\end{center}

The monograph \cite{Krisztin-Walther-Wu} of Krisztin, Walther and
Wu has addressed the question whether the equality $\mathcal{A}=\mathcal{A}_{-2,0}\cup\mathcal{A}_{0,2}$
holds under hypothesis (H1). The authors of this paper have constructed
an example in \cite{Krisztin-Vas} so that (H1) holds, and Eq.\,\eqref{eq:eq_general}
admits periodic orbits in $\mathcal{A}\setminus\left(\mathcal{A}_{-2,0}\cup\mathcal{A}_{0,2}\right)$,
that is, besides the spindle-like structures. The periodic solutions
defining these periodic orbits oscillate slowly around $0$ and have
large amplitudes in the following sense. 

A periodic solution $r:\mathbb{R}\rightarrow\mathbb{R}$ of Eq.\,\eqref{eq:eq_general}
is called a large amplitude periodic solution if $r(\mathbb{R})\supset(\xi_{-1},\xi_{1})$.
A solution $r:\mathbb{R}\rightarrow\mathbb{R}$ is slowly oscillatory
if for each $t$, the restriction $r|_{[t-1,t]}$ has one or two sign
changes. Note that here slow oscillation is different from the usual
one used for equations with negative feedback condition \cite{Diekmann et al.,Walther-2}.
A large-amplitude slowly oscillatory periodic solution $r:\mathbb{R}\rightarrow\mathbb{R}$
is abbreviated as an LSOP solution. We say that an LSOP solution $r:\mathbb{R}\rightarrow\mathbb{R}$
is normalized if $r(-1)=0$, and for some $\eta>0$, $r(s)>0$ for
all $s\in(-1,-1+\eta)$. 

The first main result of \cite{Krisztin-Vas} is as follows.

\begin{thma}

There exist $\mu$ and $f$ satisfying (H1) such that Eq.\,\eqref{eq:eq_general}
has exactly two normalized LSOP solutions $p:\mathbb{R}\rightarrow\mathbb{R}$
and $q:\mathbb{R}\rightarrow\mathbb{R}$. For the ranges of $p$ and
$q$, $(\xi_{-1},\xi_{1})\subset p(\mathbb{R})\subset q(\mathbb{R})\subset\left(\xi_{-2},\xi_{2}\right)$
 holds. The corresponding periodic orbits 
\[
\mathcal{O}_{p}=\left\{ p_{t}:\, t\in\mathbb{R}\right\} \ \mathit{and}\ \mathcal{O}_{q}=\left\{ q_{t}:\, t\in\mathbb{R}\right\} 
\]
 are hyperbolic and unstable. $\mathcal{O}_{p}$ admits two different
Floquet multipliers outside the unit circle, which are real and simple.
$\mathcal{O}_{q}$ has one real simple Floquet multiplier outside
the unit circle.

\end{thma}

Note that although Theorem 1.1 in \cite{Krisztin-Vas} does not mention
that the Floquet multipliers found outside the unit circle are simple
and real, these properties are verified in Section 4 of the same paper.

In the proof of the theorem, $\mu=1$ and $f$ is close to the step
function 
\[
f^{K,0}\left(x\right)=\begin{cases}
-K & \mbox{if }x<-1,\\
0 & \mbox{if }\left|x\right|\leq1,\\
K & \mbox{if }x>1,
\end{cases}
\]
 where $K>0$ is chosen large enough.

In their paper \cite{Fiedler-Rocha-Wolfrum}, Fiedler, Rocha and Wolfrum
considered a special class of one-dimensional parabolic partial differential
equations and obtained a catalogue listing the possible structures
of the global attractor. In particular, the result of Theorem A motivated
Fiedler, Rocha and Wolfrum to find an analogous configuration for
their equation. It is an interesting question whether all the structures
found by them have counterparts in the theory of Eq.\,\eqref{eq:eq_general}.

Let $\mathcal{W}^{u}\left(\mathcal{O}_{p}\right)$ and $\mathcal{W}^{u}\left(\mathcal{O}_{q}\right)$
denote the unstable sets of $\mathcal{O}_{p}$ and $\mathcal{O}_{q}$,
respectively.

A solution $r:\mathbb{R}\rightarrow\mathbb{R}$ is called slowly oscillatory
around $\xi_{k}$, $k\in\left\{ -1,1\right\} $, if $\mathbb{R}\ni t\mapsto r(t)-\xi_{k}\in\mathbb{R}$
admits one or two sign changes on each interval of length $1$. As
it is described by Proposition 2.7 in \cite{Krisztin-Vas}, $f$ and
$\mu$ in Theorem A are set so that there exist at least one periodic
solution oscillating slowly around $\xi_{1}$ with range in $(0,\xi_{2})$,
furthermore there is a solution $x^{1}:\mathbb{R}\rightarrow\mathbb{R}$
among such periodic solutions that has maximal range $x^{1}(\mathbb{R})$
in the sense that $x^{1}(\mathbb{R})\supset x(\mathbb{R})$ for all
periodic solutions $x$ oscillating slowly around $\xi_{1}$ with
range in $(0,\xi_{2})$. Similarly, there exists a maximal periodic
solution $x^{-1}$ oscillating slowly around $\xi_{-1}$ with range
in $(\xi_{-2},0)$. Set 
\[
\mathcal{O}_{1}=\left\{ x_{t}^{1}:t\in\mathbb{R}\right\} \mbox{ and }\mathcal{O}_{-1}=\left\{ x_{t}^{-1}:t\in\mathbb{R}\right\} .
\]

Let $\omega\left(\varphi\right)$ denote the $\omega$-limit set of
any $\varphi\in C$. Introduce the connecting sets
\begin{align*}
C_{j}^{p}= & \left\{ \varphi\in\mathcal{W}^{u}\left(\mathcal{O}_{p}\right):\,\omega\left(\varphi\right)=\hat{\xi}_{j}\right\} ,\qquad j\in\left\{ -2,0,2\right\} ,
\end{align*}
\begin{align*}
C_{k}^{p}= & \left\{ \varphi\in\mathcal{W}^{u}\left(\mathcal{O}_{p}\right):\,\omega\left(\varphi\right)=\mathcal{O}_{k}\right\} ,\qquad k\in\left\{ -1,1\right\} ,
\end{align*}
and
\begin{align*}
C_{q}^{p}= & \left\{ \varphi\in\mathcal{W}^{u}\left(\mathcal{O}_{p}\right):\,\omega\left(\varphi\right)=\mathcal{O}_{q}\right\} .
\end{align*}
Sets $C_{j}^{q}$, $j\in\left\{ -2,2\right\} $, are defined analogously.

The next theorem has also been given in \cite{Krisztin-Vas} and describes
the dynamics in $\mathcal{A}\setminus(\mathcal{A}_{-2,0}\cup\mathcal{A}_{0,2})$
. 

\begin{thmb}

One may set $\mu$ and $f$ satisfying (H1) such that the statement
of Theorem A holds, and for the global attractor $\mathcal{A}$ we
have the equality 
\[
\mathcal{A}=\mathcal{A}_{-2,0}\cup\mathcal{A}_{0,2}\cup\mathcal{W}^{u}\left(\mathcal{O}_{p}\right)\cup\mathcal{W}^{u}\left(\mathcal{O}_{q}\right).
\]
 Moreover, the dynamics on $\mathcal{W}^{u}\left(\mathcal{O}_{p}\right)$
and $\mathcal{W}^{u}\left(\mathcal{O}_{q}\right)$ is as follows.
The connecting sets $C_{j}^{p}$, $C_{q}^{p}$, $C_{k}^{p}$, $j\in\left\{ -2,0,2\right\} $,
$k\in\left\{ -1,1\right\} $, are nonempty, and
\[
\mathcal{W}^{u}\left(\mathcal{O}_{p}\right)=\mathcal{O}_{p}\cup C_{-2}^{p}\cup C_{-1}^{p}\cup C_{0}^{p}\cup C_{1}^{p}\cup C_{2}^{p}\cup C_{q}^{p}.
\]
The connecting sets $C_{-2}^{q}$ and $C_{2}^{q}$ are nonempty, and
\[
\mathcal{W}^{u}\left(\mathcal{O}_{q}\right)=\mathcal{O}_{q}\cup C_{-2}^{q}\cup C_{2}^{q}.
\]
\end{thmb}

The system of heteroclinic connections is represented in Fig.\,3. 

\begin{center}
\begin{figure}[h]
\begin{centering}
\includegraphics[scale=0.6]{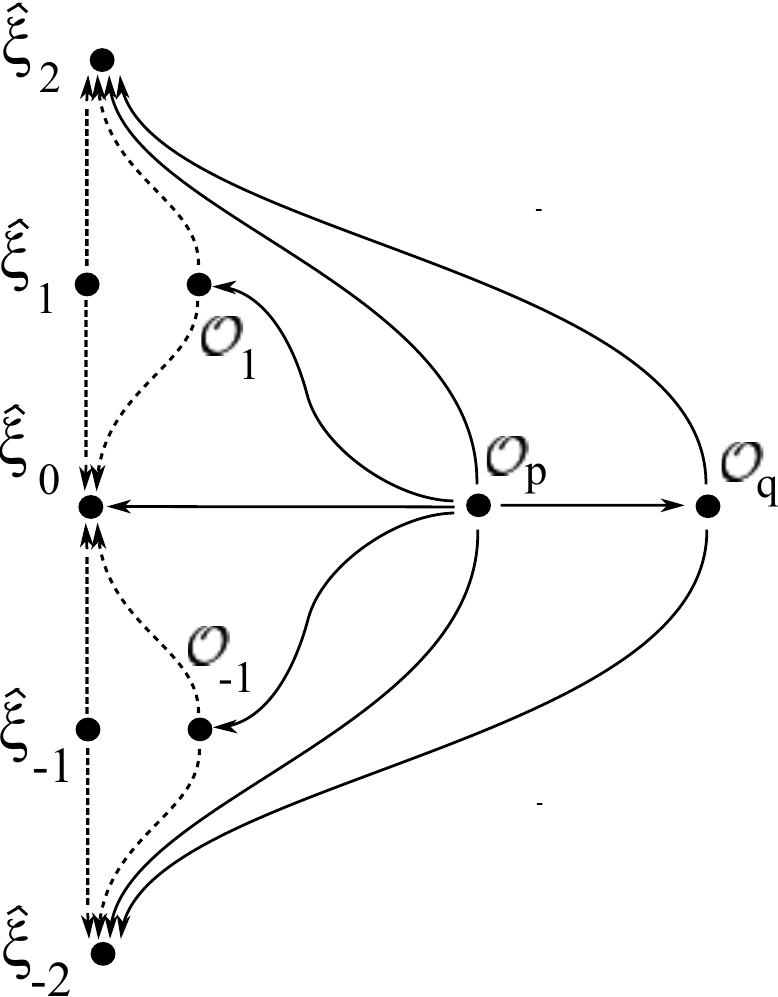} 
\par\end{centering}

\caption{Connecting orbits: the dashed arrows represent heteroclinic connections
in $\mathcal{A}_{-2,0}$ and in $\mathcal{A}_{0,2}$, while the solid
ones represent connecting orbits given by Theorem B.}
\end{figure}

\par\end{center}

Hereinafter we fix $\mu=1$ and set $f$ in Eq.\,\eqref{eq:eq_general}
so that Theorems A and B hold. The purpose of this paper is to characterize
the geometrical properties of $\mathcal{W}^{u}\left(\mathcal{O}_{p}\right)$
and the connecting sets within $\mathcal{W}^{u}\left(\mathcal{O}_{p}\right)$. 

We say that a subset $W$ of $C$ admits global graph representation,
if there exists a splitting $C=G\oplus E$ with closed subspaces $G$
and $E$ of $C$, a subset $U$ of $G$ and a map $w:U\rightarrow E$
such that 
\[
W=\left\{ \chi+w\left(\chi\right):\,\chi\in U\right\} .
\]
$W$ is said to have a smooth global graph representation if in the
above definition $U$ is open in $G$ and $w$ is $C^{1}$-smooth
on $U$. Note that in this case $W$ is a $C^{1}$-submanifold of
$C$ in the usual sense with dimension $\dim G$, see e.g. the definition
of Lang in \cite{Lang}. $W$ is said to admit a smooth global graph
representation with boundary if $G$ is $n$ dimensional with some
integer $n\geq1$, $U$ is the closure of an open set $U^{0}$, $w$
is $C^{1}$-smooth on $U^{0}$, the boundary $\mbox{bd}U$ of $U$
in $G$ is an $(n-1)$-dimensional $C^{1}$-submanifold of $G$, and
all points of $\mbox{bd}U$ have an open neighborhood in $G$ on which
$w$ can be extended to a $C^{1}$-smooth function. In this case $W$
is an $n$-dimensional $C^{1}$-submanifold of $C$ with boundary
in the usual sense \cite{Lang}.

The first result of this paper is the following.
\begin{thm}
\label{main_theorem_3} $\mathcal{W}^{u}\left(\mathcal{O}_{p}\right)$,
$C_{-2}^{p},\, C_{0}^{p}$ and $C_{2}^{p}$ are three-dimensional
$C^{1}$-submanifolds of $C$ admitting smooth global graph representations. 
\end{thm}
The next objects of our study are the connecting sets $C_{q}^{p}$,
$C_{-1}^{p}$, $C_{1}^{p}$ containing the heteroclinic orbits from
$\mathcal{O}_{p}$ to $\mathcal{O}_{q}$, $\mathcal{O}_{-1}$, $\mathcal{O}_{1}$,
respectively. We actually get a detailed picture of the structure
of $\mathcal{W}^{u}\left(\mathcal{O}_{p}\right)$ by characterizing
the unions
\[
S_{-1}=C_{-1}^{p}\cup\mathcal{O}_{p}\cup C_{q}^{p}\qquad\mbox{and}\qquad S_{1}=C_{1}^{p}\cup\mathcal{O}_{p}\cup C_{q}^{p}.
\]

A solution $x:\mathbb{R}\rightarrow\mathbb{R}$ is said to oscillate
around $\xi_{i}$, $i\in\left\{ -2,-1,0,1,2\right\} $, if the set
$x^{-1}\left(\xi_{i}\right)\subset\mathbb{R}$ is not bounded from
above. It is a direct consequence of Theorem B that for $k\in\left\{ -1,1\right\} $,
\begin{equation}
S_{k}=\left\{ \varphi\in\mathcal{W}^{u}\left(\mathcal{O}_{p}\right):\, x^{\varphi}\mbox{ oscillates around }\xi_{k}\right\} .\label{S_k}
\end{equation}

We say that a subset $W$ of $\mathcal{W}^{u}\left(\mathcal{O}_{p}\right)$
is above $S_{k}$, $k\in\left\{ -1,1\right\} $, if to each $\varphi\in W$
there corresponds an element $\psi$ of $S_{k}$ with $\psi\ll\varphi$
(that is, $\psi\left(s\right)<\varphi\left(s\right)$ for all $s\in\left[-1,0\right]$).
Similarly, a subset $W$ of $\mathcal{W}^{u}\left(\mathcal{O}_{p}\right)$
is below $S_{k}$, $k\in\left\{ -1,1\right\} $, if for all $\varphi\in W$
there exists $\psi\in S_{k}$ with $\varphi\ll\psi$. $W$ is between
$S_{-1}$ and $S_{1}$ if it is below $S_{1}$ and above $S_{-1}$.

Our main result offers geometrical and topological descriptions of
$C_{q}^{p}$, $C_{-1}^{p}$, $C_{1}^{p}$, $S_{-1}$ and $S_{1}$,
and their closures in $C$. It shows that $S_{-1}$ and $S_{1}$ separate
the points of $\mathcal{W}^{u}\left(\mathcal{O}_{p}\right)$ into
three groups according to their $\omega$-limit sets. Thereby, $S_{-1}$
and $S_{1}$ play a key role in the dynamics of the equation.
\begin{thm}
\label{main_theorem_4} ~

\noindent (i) The sets $C_{q}^{p}$, $C_{-1}^{p}$, $C_{1}^{p}$,
$S_{-1}$ and $S_{1}$ are two-dimensional $C^{1}$-submanifolds of
$\mathcal{W}^{u}\left(\mathcal{O}_{p}\right)$ with smooth global
graph representations. They are homeomorphic to the open annulus 
\[
A^{\left(1,2\right)}=\left\{ u\in\mathbb{R}^{2}:\,1<\left|u\right|<2\right\} .
\]

\noindent (ii) The equalitie\textup{s
\[
\overline{C_{q}^{p}}=\mathcal{O}_{p}\cup C_{q}^{p}\cup\mathcal{O}_{q},\qquad\overline{C_{k}^{p}}=\mathcal{O}_{p}\cup C_{k}^{p}\cup\mathcal{O}_{k}
\]
}and 
\[
\overline{S_{k}}=\mathcal{O}_{k}\cup S_{k}\cup\mathcal{O}_{q}=\mathcal{O}_{k}\cup C_{k}^{p}\cup\mathcal{O}_{p}\cup C_{q}^{p}\cup\mathcal{O}_{q}
\]
 hold for both $k\in\left\{ -1,1\right\} $. The sets $\overline{C_{q}^{p}}$,
$\overline{C_{-1}^{p}}$, $\overline{C_{1}^{p}}$,\textup{ $\overline{S_{-1}}$}
and \textup{$\overline{S_{1}}$} admit smooth global graph representations
with boundary, and thereby they are two-dimensional $C^{1}$-submanifolds
of $C$ with boundary. In addition, they are homeomorphic to the closed
annulus 
\[
A^{\left[1,2\right]}=\left\{ u\in\mathbb{R}^{2}:\,1\leq\left|u\right|\leq2\right\} .
\]
(iii) $S_{_{-1}}$ and $S_{1}$ are separatrices in the sense that
$C_{2}^{p}$ is above $S_{1}$, $C_{0}^{p}$ is between $S_{-1}$
and $S_{1}$, furthermore $C_{-2}^{p}$ is below $S_{-1}$.
\end{thm}
Fig.~4 visualizes the structure of the closure $\overline{\mathcal{W}^{u}\left(\mathcal{O}_{p}\right)}$
of $\mathcal{W}^{u}\left(\mathcal{O}_{p}\right)$ in $C$. To get
an overview of the above results regarding $\mathcal{W}^{u}\left(\mathcal{O}_{p}\right)$,
see the inner part of Fig.~4, drawn in black. We emphasize a particular
consequence of Theorem \ref{main_theorem_4}: the tangent spaces of
$S_{-1}$ and $S_{1}$ coincide along $\mathcal{O}_{p}$, see Fig.~5.

\begin{figure}[h]
\begin{centering}
\includegraphics[scale=0.85]{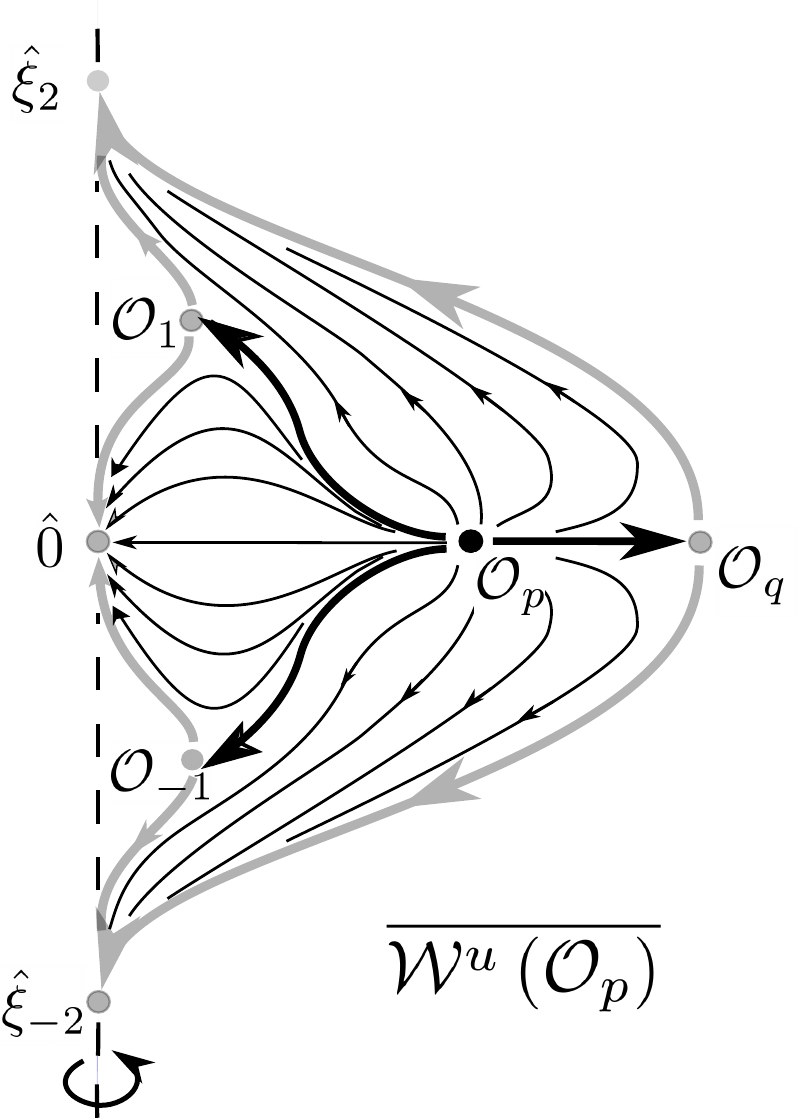} 
\par\end{centering}

\caption{$\overline{\mathcal{W}^{u}\left(\mathcal{O}_{p}\right)}$ can be visualized
as a ``tulip'' rotated around the vertical axis: the dots correspond
to equilibria and periodic orbits, the thick arrows symbolize two-dimensional
heteroclinic connecting sets, and the three groups of thin arrows
represent three-dimensional connecting sets. The elements of $\mathcal{W}^{u}\left(\mathcal{O}_{p}\right)$
are drawn in black. Grey is used for the boundary of $\mathcal{W}^{u}\left(\mathcal{O}_{p}\right)$. }
\end{figure}

\begin{figure}[h]
\begin{centering}
\includegraphics[scale=1.3]{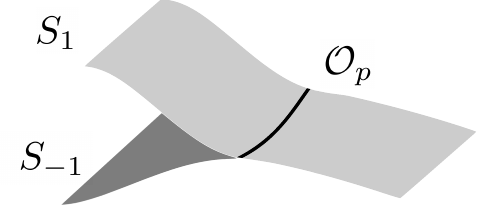} 
\par\end{centering}

\caption{The tangent spaces of $S_{-1}$ and $S_{1}$ coincide along $\mathcal{O}_{p}$.}
\end{figure}

Let $\mathcal{W}^{u}\left(\mathcal{O}_{1}\right)$ and $\mathcal{W}^{u}\left(\mathcal{O}_{-1}\right)$
denote the unstable sets of $\mathcal{O}_{1}$ and $\mathcal{O}_{-1}$,
respectively, defined as the forward extension of a one-dimensional
local unstable manifold of a return map (corresponding to the only
Floquet multiplier outside the unit circle which is real and simple),
see \eqref{unstable set is forwad extension of unstable manifold-1}.
We expect $\mathcal{W}^{u}\left(\mathcal{O}_{q}\right)$, $\mathcal{W}^{u}\left(\mathcal{O}_{-1}\right)$
and $\mathcal{W}^{u}\left(\mathcal{O}_{1}\right)$ to be two-dimensional
$C^{1}$-submanifolds of $C$. We conjecture that for the closure
$\overline{\mathcal{W}^{u}\left(\mathcal{O}_{p}\right)}$ of $\mathcal{W}^{u}\left(\mathcal{O}_{p}\right)$
in $C$, the equality 
\[
\overline{\mathcal{W}^{u}\left(\mathcal{O}_{p}\right)}=\mathcal{W}^{u}\left(\mathcal{O}_{p}\right)\cup\mathcal{W}^{u}\left(\mathcal{O}_{q}\right)\cup\mathcal{W}^{u}\left(\mathcal{O}_{1}\right)\cup\mathcal{W}^{u}\left(\mathcal{O}_{-1}\right)\cup\left\{ \hat{\xi}_{-2},\hat{0},\hat{\xi}_{2}\right\} 
\]
holds, as it represented in Fig. 4. Moreover, all points of $\mathcal{W}^{u}\left(\mathcal{O}_{q}\right)\cup\mathcal{W}^{u}\left(\mathcal{O}_{1}\right)\cup\mathcal{W}^{u}\left(\mathcal{O}_{-1}\right)$
have an open neighborhood on which the $C^{1}$-map in the graph representation
of $\mathcal{W}^{u}\left(\mathcal{O}_{p}\right)$ can be smoothly
extended.

It also remains an open question whether $\mathcal{A}\setminus(\mathcal{A}_{-2,0}\cup\mathcal{A}_{0,2})$
is homeomorphic to the three-dimensional body 
\[
\mathcal{B}_{3}\left(\left(0,0,0\right),2\right)\backslash\left\{ \mathcal{B}_{3}\left(\left(0,0,1\right),1\right)\cup\mathcal{B}_{3}\left(\left(0,0,-1\right),1\right)\right\} \subset\mathbb{R}^{3},
\]
 where $\mathcal{B}_{3}\left(\left(a_{1},a_{2},a_{3}\right),r\right)$
denotes the three-dimensional closed ball with center $\left(a_{1},a_{2},a_{3}\right)$
and radius $r$.

The proofs of Theorems \ref{main_theorem_3}--\ref{main_theorem_4}
apply general results on delay differential equations, the Floquet
theory (Appendix VII of \cite{Krisztin-Walther-Wu}, \cite{Mallet-Paret}),
results on local invariant manifolds for maps in Banach spaces (Appendices
I-II of \cite{Krisztin-Walther-Wu}), correspondences between different
return maps (Appendices I and V of \cite{Krisztin-Walther-Wu}), a
result from transversality theory \cite{Abraham-Robbin} and also
a discrete Lyapunov functional of Mallet-Paret and Sell counting the
sign changes of the elements of $C$ (Appendix VI of \cite{Krisztin-Walther-Wu},
\cite{Mallet-Paret_Sell}). 

This paper is organized as follows. Section \ref{sec:Prelimimaries}
offers a general overview of the theoretical background and introduces
the discrete Lyapunov functional. As the Floquet theory and certain
results on local invariant manifolds of return maps play essential
role in this work, Section \ref{sec:Floquet-multipliers and invariant manifolds}
is devoted to the discussion of these concepts. Sections 4 and 5 contain
the proofs of Theorems \ref{main_theorem_3} and \ref{main_theorem_4},
respectively. 

The proof of Theorem \ref{main_theorem_3} in Section 4 takes advantage
of the fact that the unstable set of a hyperbolic periodic orbit is
the forward continuation of a local unstable manifold of a Poincar\'e
map by the semiflow. In consequence, by using the smoothness of the
local unstable manifold and the injectivity of the derivative of the
solution operator, we prove that all points $\varphi$ of $\mathcal{W}^{u}\left(\mathcal{O}_{p}\right)$
belong to a subset $W_{\varphi}$ of $\mathcal{W}^{u}\left(\mathcal{O}_{p}\right)$
that is a three-dimensional $C^{1}$-submanifold of $C$. This means
that $\mathcal{W}^{u}\left(\mathcal{O}_{p}\right)$ is an immersed
submanifold of $C$. In general, an immersed submanifold is not necessarily
an embedded submanifold of the phase space. In order to prove that
$\mathcal{W}^{u}\left(\mathcal{O}_{p}\right)$ is embedded in $C$,
we have to show that for any $\varphi$ in $\mathcal{W}^{u}\left(\mathcal{O}_{p}\right)$,
there is no sequence in $\mathcal{W}^{u}\left(\mathcal{O}_{p}\right)\backslash W_{\varphi}$
converging to $\varphi$. We define a projection $\pi_{3}$ from $C$
into $\mathbb{R}^{3}$. Using well-known properties of the discrete
Lyapunov functional, we show that $\pi_{3}$ is injective on $\mathcal{W}^{u}\left(\mathcal{O}_{p}\right)$
and on the tangent spaces of $W_{\varphi}$. This implies that $\pi_{3}W_{\varphi}$
is open in $\mathbb{R}^{3}$. If a sequence $\left(\varphi^{n}\right)_{n=0}^{\infty}$
in $\mathcal{W}^{u}\left(\mathcal{O}_{p}\right)\backslash W_{\varphi}$
converges to $\varphi$ as $n\rightarrow\infty$, then $\pi_{3}\varphi^{n}\rightarrow\pi_{3}\varphi$
as $n\rightarrow\infty$, and $\pi_{3}\varphi^{n}\in\pi_{3}W_{\varphi}$
for all $n$ large enough. The injectivity of $\pi_{3}$ on $\mathcal{W}^{u}\left(\mathcal{O}_{p}\right)$
then implies that $\varphi^{n}\in W_{\varphi}$, which is a contradiction.
So $\mathcal{W}^{u}\left(\mathcal{O}_{p}\right)$ is a three-dimensional
embedded $C^{1}$-submanifold of the phase space. The description
of $\mathcal{W}^{u}\left(\mathcal{O}_{p}\right)$ is rounded up by
giving a graph representation for $\mathcal{W}^{u}\left(\mathcal{O}_{p}\right)$
in order to present the simplicity of it structure. The smoothness
of the sets $C_{-2}^{p},\, C_{0}^{p}$ and $C_{2}^{p}$ then follows
at once because they are open subsets of $\mathcal{W}^{u}\left(\mathcal{O}_{p}\right)$.
We also obtain as an important consequence that the semiflow defined
by the solution operator extends to a $C^{1}$-flow on $\mathcal{W}^{u}\left(\mathcal{O}_{p}\right)$
with injective derivatives.

The proof of Theorem \ref{main_theorem_4} in Section 5 is built from
several steps, and it is organized into five subsections.

In Subsection 5.1 we list preliminary results regarding the closure
$\overline{S_{k}}$ of $S_{k}$ in $C$, $k\in\left\{ -1,1\right\} $.
We introduce in particular a projection $\pi_{2}$ from $C$ into
$\mathbb{R}^{2}$, and -- using the special properties of the discrete
Lyapunov functional -- we show that $\pi_{2}$ is injective on $\overline{S_{k}}$.
The injectivity of $\pi_{2}|_{\overline{S_{k}}}$ is already sufficient
to give a two-dimensional graph representation for any subset $W$
of $\overline{S_{k}}$ (without smoothness properties): there is a
linear isomorphism $J_{2}:\mathbb{R}^{2}\rightarrow C$ such that
$P_{2}=J_{2}\circ\pi_{2}:C\rightarrow C$ is a projection onto a two-dimensional
subspace $G_{2}$ of $C$, and there exists a map $w_{k}$ defined
on the image set $P_{2}\overline{S_{k}}$ with range in $P_{2}^{-1}\left(0\right)$
such that for any subset $W\subseteq S_{k}$, 
\[
W=\left\{ \chi+w_{k}\left(\chi\right):\,\chi\in P_{2}W\right\} .
\]
The smoothness of $w_{k}$ and the properties of its domain $P_{2}\overline{S_{k}}\subset G_{2}$
are investigated later. Subsection 5.1 is closed with showing that
$\pi_{2}|_{\overline{S_{k}}}$ is a homeomorphism onto its image,
furthermore $\pi_{2}$ is injective on the tangent spaces of $\overline{S_{k}}$. 

It is clear that $\left(\mathcal{O}_{k}\cup S_{k}\cup\mathcal{O}_{q}\right)\subset\overline{S_{k}}$
for both $k\in\left\{ -1,1\right\} $. The converse inclusion is proved
in Subsection 5.2 based on the previously obtained result that $\overline{S_{k}}$
is mapped injectively into $\mathbb{R}^{2}$. Then it follows easily
that $\overline{C_{k}^{p}}$, $k\in\left\{ -1,1\right\} ,$ and $\overline{C_{q}^{p}}$
are not larger than the unions $\mathcal{O}_{p}\cup C_{k}^{p}\cup\mathcal{O}_{k}$
and $\mathcal{O}_{p}\cup C_{q}^{p}\cup\mathcal{O}_{q}$, respectively. 

It is a more challenging task to show that $C_{q}^{p}$ and $C_{k}^{p}$,
$k\in\left\{ -1,1\right\} ,$ are $C^{1}$-submanifolds of $\mathcal{W}^{u}\left(\mathcal{O}_{p}\right)$
(as stated by Theorem \ref{main_theorem_4}.(i)). The proof of this
assertion is contained in Subsection 5.3. It is partly based on transversality
\cite{Abraham-Robbin}; we verify that $\mathcal{W}^{u}\left(\mathcal{O}_{p}\right)$
intersects transversally a local center-stable manifold of a Poincar\'e
return map at a point of $\mathcal{O}_{k}$ and a local stable manifold
of a Poincar\'e return map at a point of $\mathcal{O}_{q}$, and
thereby the intersections -- subsets of $C_{q}^{p}$ and $C_{k}^{p}$
-- are one-dimensional submanifolds of $\mathcal{W}^{u}\left(\mathcal{O}_{p}\right)$.
The main difficulty in this task is that the hyperbolicity of $\mathcal{O}_{k}$
is not known. Krisztin, Walther and Wu have proved transversality
in a similar situation \cite{Krisztin-Walther-Wu}. Then we apply
techniques that already appeared in Section 4. The injectivity of
the derivative of the flow induced by the solution operator on $\mathcal{W}^{u}\left(\mathcal{O}_{p}\right)$
guarantees that each point $\varphi$ in $C_{q}^{p}$ or $C_{k}^{p}$
belongs to a ``small'' subset of $C_{q}^{p}$ or $C_{k}^{p}$, respectively,
that is a two-dimensional $C^{1}$-submanifold of $\mathcal{W}^{u}\left(\mathcal{O}_{p}\right)$.
Therefore, $C_{q}^{p}$ and $C_{k}^{p}$ are immersed $C^{1}$-submanifolds
of $\mathcal{W}^{u}\left(\mathcal{O}_{p}\right)$. In order to prove
that $C_{q}^{p}$ and $C_{k}^{p}$ are embedded in $\mathcal{W}^{u}\left(\mathcal{O}_{p}\right)$,
we repeat an argument from the proof of Theorem \ref{main_theorem_3}
with $\pi_{2}$ in the role of $\pi_{3}$. Based on the property that
$C_{q}^{p}$ and $C_{k}^{p}$ are $C^{1}$-submanifolds of $\mathcal{W}^{u}\left(\mathcal{O}_{p}\right)$,
we prove at the end of Subsection 5.3 that $w_{k}$ is continuously
differentiable on the open sets $P_{2}C_{q}^{p}$ and $P_{2}C_{k}^{p}$,
i.e., the representations 
\[
C_{q}^{p}=\left\{ \chi+w_{k}\left(\chi\right):\,\chi\in P_{2}C_{q}^{p}\right\} \quad\mbox{and}\quad C_{k}^{p}=\left\{ \chi+w_{k}\left(\chi\right):\,\chi\in P_{2}C_{k}^{p}\right\} .
\]
are smooth. 

Next we verify in Subsection 5.4 that the images of $C_{q}^{p}$,
$C_{k}^{p}$ and $S_{k}$, $k\in\left\{ -1,1\right\} ,$ under $\pi_{2}$
are topologically equivalent to the open annulus, and the images of
their closures are topologically equivalent to the closed annulus.

As
\[
S_{k}=\left\{ \chi+w_{k}\left(\chi\right):\,\chi\in P_{2}S_{k}\right\} \quad\mbox{and}\quad P_{2}S_{k}=P_{2}C_{k}^{p}\cup P_{2}\mathcal{O}_{p}\cup P_{2}C_{q}^{p},
\]
we have a smooth representation for $S_{k}$ if we show that $P_{2}S_{k}$
is open in $G_{2}$ and $w_{k}$ is smooth at the points of $P_{2}\mathcal{O}_{p}$.
This is done in Subsection 5.5. It follows immediately that $S_{k}$
is a $C^{1}$-submanifold of $\mathcal{W}^{u}\left(\mathcal{O}_{p}\right)$.
Simultaneously, we verify that all points of $P_{2}\mathcal{O}_{k}\cup P_{2}\mathcal{O}_{q}$
have open neighborhoods on which $w_{k}$ can be extended to $C^{1}$-functions.
As $P_{2}\mathcal{O}_{k}\cup P_{2}\mathcal{O}_{q}$ is the boundary
of \textit{$P_{2}\overline{S_{k}}$}, this step guarantees that $\overline{S_{k}}$
has a smooth representation with boundary, and thereby $\overline{S_{k}}$
is a $C^{1}$-submanifold of $C$ with boundary. The same reasonings
yield the analogous results for $\overline{C_{q}^{p}}$ and $\overline{C_{k}^{p}}$.
Summing up, the proofs of Theorem \ref{main_theorem_4}.(i) and (ii)
are completed in Subsection 5.5.

It remains to show that $S_{-1}$ and $S_{1}$ are indeed separatrices
in the sense described by Theorem \ref{main_theorem_4}.(iii). It
is easy to see that the assertion restricted to a local unstable manifold
of $\mathcal{O}_{p}$ holds. Then we use the monotonicity of the semiflow
to extend the statement for $\mathcal{W}^{u}\left(\mathcal{O}_{p}\right).$

Several techniques applied here have already appeared in the monograph
\cite{Krisztin-Walther-Wu} of Krisztin, \foreignlanguage{canadian}{Walther}
and Wu. The novelty of this paper compared to \cite{Krisztin-Walther-Wu}
is that here we describe the unstable set of a periodic orbit, while
\cite{Krisztin-Walther-Wu} considers the unstable set of an equlibrium
point.

\textbf{Acknowledgments. }Both authors were supported by the Hungarian
Scientific Research Fund, Grant No. K109782. The research of Gabriella
Vas was supported by the European Union and the State of Hungary,
co-financed by the European Social Fund in the framework of TÁMOP-4.2.4.A/
2-11/1-2012-0001 \textquoteleft{}National Excellence Program\textquoteright{}.
The research of Tibor Krisztin was also supported by the European
Union and co-funded by the European Social Fund. Project title: \textquotedblleft{}Telemedicine-focused
research activities on the field of Matematics, Informatics and Medical
sciences\textquotedblright{} Project number: TÁMOP-4.2.2.A-11/1/KONV-2012-0073.

\section{Preliminaries\label{sec:Prelimimaries}}

We fix $\mu=1$ and set $f$ in Eq.\,\eqref{eq:eq_general} so that
Theorems A and B hold. In this section we give a summary of the theoretical
background. In particular, we discuss the differentiability of the
semiflow, the basic properties of the global attractor, the discrete
Lyapunov functional of Mallet-Paret and Sell, and we list some technical
results. The discussion of the Floquet theory and the Poincar\'e
return maps is left to the next section.

\subsection*{Phase space, solution, segment.}

The natural phase space\emph{ }for Eq.\,\eqref{eq:eq_general} is
the Banach space $C=C\left(\left[-1,0\right],\mathbb{R}\right)$ of
continuous real functions defined on $\left[-1,0\right]$ equipped
with the supremum norm 
\[
\left\Vert \varphi\right\Vert =\sup_{-1\leq s\leq0}\left|\varphi\left(s\right)\right|.
\]

If $J$ is an interval, $u:J\rightarrow\mathbb{R}$ is continuous
and $\left[t-1,t\right]\subseteq J$, then the segment $u_{t}\in C$
is defined by $u_{t}\left(s\right)=u\left(t+s\right)$, $-1\leq s\leq0$.

Let $C^{1}$ denote the subspace of $C$ containing the continuously
differentiable functions. Then $C^{1}$ is also a Banach space with
the norm $\left\Vert \varphi\right\Vert _{C^{1}}=\left\Vert \varphi\right\Vert +\left\Vert \varphi'\right\Vert .$ 

For all $\xi\in\mathbb{R}$, $\hat{\xi}\in C$ is defined by $\hat{\xi}\left(s\right)=\xi$
for all $s\in\left[-1,0\right]$.

A solution of Eq.\,\eqref{eq:eq_general} is either a continuous
function on $\left[t_{0}-1,\infty\right)$, $t_{0}\in\mathbb{R}$,
which is differentiable for $t>t_{0}$ and satisfies equation Eq.\,\eqref{eq:eq_general}
on $\left(t_{0},\infty\right)$, or a continuously differentiable
function on $\mathbb{R}$ satisfying the equation for all $t\in\mathbb{R}$.
To all $\varphi\in C$, there corresponds a unique solution $x^{\varphi}:\left[-1,\infty\right)\rightarrow\mathbb{R}$
of Eq.~\eqref{eq:eq_general} with $x_{0}^{\varphi}=\varphi$. On
$\left(0,\infty\right)$, $x^{\varphi}$ is given by the variation-of-constants
formula for ordinary differential equations repeated on successive
intervals of length $1$: 
\begin{equation}
x^{\varphi}\left(t\right)=e^{n-t}x^{\varphi}\left(n\right)+\int_{n}^{t}e^{s-t}f\left(x^{\varphi}\left(s-1\right)\right)\mbox{d}s\quad\mbox{for all }n\in\mathbb{N},\, n\leq t\leq n+1.\label{eq:variation-of-constants formula}
\end{equation}

\subsection*{Semiflow.}

The solutions of Eq.\,\eqref{eq:eq_general} define the continuous
semiflow 
\[
\Phi:\mathbb{R}^{+}\times C\ni\left(t,\varphi\right)\mapsto x_{t}^{\varphi}\in C.
\]
All maps $\Phi\left(t,\cdot\right):C\rightarrow C$, $t\geq1$, are
compact \cite{Hale-1}. As $f'>0$ on $\mathbb{R}$, all maps $\Phi\left(t,\cdot\right):C\rightarrow C$,
$t\geq0$, are injective \cite{Krisztin-Walther-Wu}. It follows that
for every $\varphi\in C$ there is at most one solution $x:\mathbb{R}\rightarrow\mathbb{R}$
of Eq.\,\eqref{eq:eq_general} with $x_{0}=\varphi.$ Whenever such
solution exists, we denote it also by $x^{\varphi}$. 

For fixed $\varphi\in C$, the map $\left(1,\infty\right)\ni t\mapsto\Phi(t,\varphi)\in C$
is continuously differentiable with $D_{1}\Phi\left(t,\varphi\right)1=\dot{x_{t}}^{\varphi}$
for all $t>1$. For all $t\geq0$ fixed, $C\ni\varphi\mapsto\Phi(t,\varphi)\in C$
is continuously differentiable, and $D_{2}\Phi(t,\varphi)\eta=v_{t}^{\eta}$,
where $v^{\eta}:\left[-1,\infty\right)\rightarrow\mathbb{R}$ is the
solution of the linear variational equation 
\begin{align}
\dot{v}(t) & =-v(t)+f'\left(x^{\varphi}\left(t-1\right)\right)v\left(t-1\right)\label{var eq}
\end{align}
with $v_{0}^{\eta}=\eta$. So the restriction of $\Phi$ to the open
set $\left(1,\infty\right)\times C$ is continuously differentiable. 
\begin{prop}
\label{prop:uniqueness of solutions}Suppose that $\eta\in C$, $b:\mathbb{R}\rightarrow\mathbb{R}$
is positive, and the problem 
\[
\begin{cases}
\dot{v}(t) & =-v(t)+b(t)v\left(t-1\right)\\
v_{0} & =\eta
\end{cases}
\]
has a solution $v^{\eta}$ either on $\left[t_{0}-1,\infty\right)$
with $t_{0}\leq0$ or on $\mathbb{R}$ (i.e., there is a continuous
function $v^{\eta}:\left[t_{0}-1,\infty\right)\rightarrow\mathbb{R}$
with $v_{0}^{\eta}=\eta$ that is differentiable and satisfies the
equation for $t>t_{0}$, or there exists a differentiable function
$v^{\eta}:\mathbb{R}\rightarrow\mathbb{R}$ with $v_{0}^{\eta}=\eta$
satisfying the equation for all real $t$, respectively). Then $v^{\eta}$
is unique.\end{prop}
\begin{proof}
As the solution on $\left[0,\infty\right)$ is determined by a variation-of-constants
formula analogous to \eqref{eq:variation-of-constants formula}, the
uniqueness in forward time is clear. For $t<0$, the uniqueness follows
from $v\left(t-1\right)=\left(\dot{v}\left(t\right)+v\left(t\right)\right)/b\left(t\right)$. 
\end{proof}
In particular, the solution operator $D_{2}\Phi(t,\varphi)$ corresponding
to the variational equation \eqref{var eq} is injective for all $\varphi\in C$
and $t\geq0$. 

A function \emph{$\hat{\xi}\in C$} is an equilibrium point\emph{
}(or stationary point) of $\Phi$ if and only if $\hat{\xi}\left(s\right)=\xi$
for all $-1\leq s\leq0$ with $\xi\in\mathbb{R}$ satisfying $-\xi+f\left(\xi\right)=0$.
Then $x^{\hat{\xi}}\left(t\right)=\xi$ for all $t\in\mathbb{R}$.
As it is described in Chapter 2 of \cite{Krisztin-Walther-Wu}, condition
$f'\left(\xi\right)<1$ implies that $\hat{\xi}$ is stable and locally
attractive. If $f'\left(\xi\right)>1$, then $\hat{\xi}$ is unstable.
So hypothesis (H1) with $\mu=1$ implies that $\hat{\xi}_{-2},$ $\hat{\xi}_{0}$
and $\hat{\xi}_{2}$ are stable, and $\hat{\xi}_{-1}$ and $\hat{\xi}_{1}$
are unstable.

\subsection*{Limit sets.}

If $\varphi\in C$ and $x^{\varphi}:\left[-1,\infty\right)\rightarrow\mathbb{R}$
is a bounded solution of Eq.\,\eqref{eq:eq_general}, then the\emph{
$\omega$-}limit set 
\begin{align*}
\omega\left(\varphi\right)= & \left\{ \psi\in C:\,\mbox{there exists a sequence }\left(t_{n}\right)_{0}^{\infty}\mbox{ in }\left[0,\infty\right)\right.\\
 & \left.\mbox{ with }t_{n}\rightarrow\infty\mbox{ and }\Phi\left(t_{n},\varphi\right)\rightarrow\psi\mbox{ as }n\rightarrow\infty\right\} 
\end{align*}
is nonempty, compact, connected and invariant. For a solution $x:\mathbb{R}\rightarrow\mathbb{R}$
such that $x|_{\left(-\infty,0\right]}$ is bounded, the \emph{$\alpha$}-limit
set
\begin{align*}
\alpha\left(x\right)= & \left\{ \psi\in C:\,\mbox{there exists a sequence }\left(t_{n}\right)_{0}^{\infty}\mbox{ in }\mathbb{R}\right.\\
 & \left.\mbox{ with }t_{n}\rightarrow-\infty\mbox{ and }x_{t_{n}}\rightarrow\psi\mbox{ as }n\rightarrow\infty\right\} 
\end{align*}
is also nonempty, compact, connected and invariant. 

According to the Poincar\'e--Bendixson theorem of Mallet-Paret and
Sell \cite{Mallet-Paret_Sell2}, for all 
\[
\varphi\in C_{-2,2}=\left\{ \varphi\in C:\,\xi_{-2}\leq\varphi\left(s\right)\leq\xi_{2}\mbox{ for all }s\in\left[-1,0\right]\right\} ,
\]
the set $\omega\left(\varphi\right)$ is either a single nonconstant
periodic orbit, or for each $\psi\in\omega\left(\varphi\right)$,
\[
\alpha\left(x^{\psi}\right)\cup\omega\left(\psi\right)\subseteq\left\{ \hat{\xi}_{-2},\hat{\xi}_{-1},\hat{\xi}_{0},\hat{\xi}_{1},\hat{\xi}_{2}\right\} .
\]
An analogous result holds for $\alpha\left(x\right)$ in case $x$
is defined on $\mathbb{R}$ and $\left\{ x_{t}:\ t\leq0\right\} \subset C_{-2,2}$. 

By Theorem 4.1 in Chapter 5 of \cite{Smith}, there is an open and
dense set of initial functions in $C_{-2,2}$ so that the corresponding
solutions converge to equilibria.

Note that there is no homoclinic orbit to $\hat{\xi}_{j}$, $j\in\left\{ -2,0,2\right\} $,
as these equilibria are stable. It follows from Proposition 3.1 in
\cite{Krisztin-3} that there exists no homoclinic orbits to the unstable
equilibria $\hat{\xi}_{-1}$ and $\hat{\xi}_{1}$.

\subsection*{The global attractor.}

The\emph{ }global attractor $\mathcal{A}$ of the restriction $\Phi|_{\left[0,\infty\right)\times C_{-2,2}}$
is a nonempty, compact set in $C$, that is invariant in the sense
that $\Phi\left(t,\mathcal{A}\right)=\mathcal{A}$ for all $t\geq0$,
and that attracts bounded sets in the sense that for every bounded
set $B\subset C_{-2,2}$ and for every open set $U\supset\mathcal{A}$,
there exists $t\geq0$ with $\Phi\left(\left[t,\infty\right)\times B\right)\subset U$.
Global attractors are uniquely determined \cite{Hale-1}. It can be
shown that 
\begin{align*}
\mathcal{A}= & \left\{ \varphi\in C_{-2,2}:\mbox{ there is a bounded solution }x:\mathbb{R}\rightarrow\mathbb{R}\right.\\
 & \left.\mbox{ of Eq.\,}\eqref{eq:eq_general}\mbox{ so that }\varphi=x_{0}\right\} ,
\end{align*}
see \cite{Krisztin-Walther,Mallet-Paret,Polner}. 

The compactness of $\mathcal{A}$, its invariance property and the
injectivity of the maps $\Phi\left(t,\cdot\right):C\rightarrow C$,
$t\geq0$, combined permit to verify that the map 
\[
\left[0,\infty\right)\times\mathcal{A}\ni\left(t,\varphi\right)\mapsto\Phi\left(t,\varphi\right)\in\mathcal{A}
\]
extends to a continuous flow $\Phi_{\mathcal{A}}:\mathbb{R}\times\mathcal{A}\rightarrow\mathcal{A}$;
for every $\varphi\in\mathcal{A}$ and for all $t\in\mathbb{R}$ we
have $\Phi_{\mathcal{A}}\left(t,\varphi\right)=x_{t}^{\varphi}$ with
the uniquely determined solution $x^{\varphi}:\mathbb{R}\rightarrow\mathbb{R}$
of Eq.\,\eqref{eq:eq_general} satisfying $x_{0}^{\varphi}=\varphi$.

Note that we have $\mathcal{A}=\Phi\left(1,\mathcal{A}\right)\subset C^{1}$;
$\mathcal{A}$ is a closed subset of $C^{1}$. Using the flow $\Phi_{\mathcal{A}}$
and the continuity of the map
\[
C\ni\varphi\mapsto\Phi\left(1,\varphi\right)\in C^{1},
\]
one obtains that $C$ and $C^{1}$ define the same topology on $\mathcal{A}$.

\subsection*{A discrete Lyapunov functional\emph{.}}

Following Mallet-Paret and Sell in \cite{Mallet-Paret_Sell}, we use
a discrete Lyapunov functional $V:C\setminus\left\{ \hat{0}\right\} \rightarrow2\mathbb{N}\cup\left\{ \infty\right\} $.
For $\varphi\in C\setminus\left\{ \hat{0}\right\} ,$ set $sc\left(\varphi\right)=0$
if $\varphi\geq\hat{0}$ or $\varphi\leq\hat{0}$ (i.e., $\varphi\left(s\right)\geq0$
for all $s\in\left[-1,0\right]$ or $\varphi\left(s\right)\leq0$
for all $s\in\left[-1,0\right]$, respectively), otherwise define
\[
sc\left(\varphi\right)=\textrm{sup}\Bigl\{ k\in\mathbb{N}\setminus\left\{ 0\right\} :\mbox{ there exist a strictly increasing sequence}
\]
\[
\left.\left(s_{i}\right)_{0}^{k}\subseteq\left[-1,0\right]\textrm{ with }\varphi\left(s_{i-1}\right)\varphi\left(s_{i}\right)<0\textrm{ for }i\in\left\{ 1,2,..,k\right\} \right\} .
\]
 Then set
\[
V\left(\varphi\right)=\left\{ \begin{array}{ll}
sc\left(\varphi\right), & \textrm{if }sc\left(\varphi\right)\textrm{ is even or }\infty,\\
sc\left(\varphi\right)+1, & \textrm{if }sc\left(\varphi\right)\textrm{ is odd}.
\end{array}\right.
\]

Also define 
\begin{eqnarray*}
R & = & \left\{ \varphi\in C^{1}:\,\varphi\left(0\right)\neq0\mbox{ or }\dot{\varphi}\left(0\right)\varphi\left(-1\right)>0,\right.\\
 &  & \,\varphi\left(-1\right)\neq0\mbox{ or }\dot{\varphi}\left(-1\right)\varphi\left(0\right)<0,\left.\mbox{all zeros of }\varphi\mbox{ are simple}\right\} .
\end{eqnarray*}

$V$ has the following lower semi-continuity and continuity property
(for a proof, see \cite{Krisztin-Walther-Wu,Mallet-Paret_Sell}). 
\begin{lem}
\label{lem: continuity_of_V} For each $\varphi\in C\setminus\left\{ \hat{0}\right\} $
and $\left(\varphi_{n}\right)_{0}^{\infty}\subset C\setminus\left\{ \hat{0}\right\} $
with $\varphi_{n}\rightarrow\varphi$ as $n\rightarrow\infty$, $V\left(\varphi\right)\leq\liminf_{n\rightarrow\infty}V\left(\varphi_{n}\right)$.
For each $\varphi\in R$ and $\left(\varphi_{n}\right)_{0}^{\infty}\subset C^{1}\setminus\left\{ \hat{0}\right\} $
with $\left\Vert \varphi_{n}-\varphi\right\Vert _{C^{1}}\rightarrow0$
as $n\rightarrow\infty$, $V\left(\varphi\right)=\lim_{n\rightarrow\infty}V\left(\varphi_{n}\right)<\infty$. 
\end{lem}
The next result explains why $V$ is called a Lyapunov functional
(for a proof, see \cite{Krisztin-Walther-Wu,Mallet-Paret_Sell} again).
For an interval $J\subset\mathbb{R}$, we use the notation 
\[
J+\left[-1,0\right]=\left\{ t\in\mathbb{R}:\, t=t_{1}+t_{2}\mbox{ with }t_{1}\in J,\, t_{2}\in\left[-1,0\right]\right\} .
\]

\begin{lem}
\noindent \label{lem:4_properties_of_V} Assume that $\mu\geq0$,
$J\subset\mathbb{R}$ is an interval, $a:J\rightarrow\mathbb{R}$
is positive and continuous, $z:J+\left[-1,0\right]\rightarrow\mathbb{R}$
is continuous, $z\left(t\right)\neq0$ for some $t\in J+\left[-1,0\right]$,
and $z$ is differentiable on $J$. Suppose that
\begin{equation}
\dot{z}\left(t\right)=-\mu z\left(t\right)+a\left(t\right)z\left(t-1\right)\label{Lyapunov-eq}
\end{equation}
 holds for all $t>\inf J$ in $J$. Then the following statements
hold.

\noindent (i) If $t_{1},t_{2}\in J$ with $t_{1}<t_{2}$ , then $V\left(z_{t_{1}}\right)\geq V\left(z_{t_{2}}\right)$.

\noindent (ii) If $t,t-2\in J$, $z\left(t-1\right)=z\left(t\right)=0$,
then either $V\left(z_{t}\right)=\infty$ or $V\left(z_{t-2}\right)>V\left(z_{t}\right)$.

\noindent (iii) If $t\in J$, $t-3\in J$, and $V\left(z_{t-3}\right)=V\left(z_{t}\right)<\infty$,
then $z_{t}\in R$.
\end{lem}
If $f$ is a $C^{1}$-smooth function with $f'>0$ on $\mathbb{R}$,
$x,\hat{x}:J+\left[-1,0\right]\rightarrow\mathbb{R}$ are solutions
of Eq.\,\eqref{eq:eq_general} and $c\in\mathbb{R}\setminus\left\{ 0\right\} $,
then Lemma \ref{lem:4_properties_of_V} can be applied for $z=\left(x-\hat{x}\right)/c$
with the positive continuous function
\[
a:J\ni t\mapsto\int_{0}^{1}f'\left(sx\left(t-1\right)+\left(1-s\right)\hat{x}\left(t-1\right)\right)\mbox{d}s\in\left[0,\infty\right).
\]

\subsection*{Further notations and preliminary results.}

A solution $x$ is oscillatory around an equilibrium $\hat{\xi}$
if $x^{-1}\left(\xi\right)$ is not bounded from above, and it is
slowly oscillatory around $\hat{\xi}$ if $t\rightarrow x\left(t\right)-\xi$
has one or two sign changes on each interval of length $1$. 

$B\left(\varphi,r\right)$, $\varphi\in C$ , $r>0$, denotes the
open ball in $C$ with center $\varphi$ and radius $r$.

We use the notation $S_{\mathbb{C}}^{1}$ for the set $\left\{ z\in\mathbb{C}:\,\left|z\right|=1\right\} $.

For a simple closed curve $c:\left[a,b\right]\rightarrow\mathbb{R}^{2}$,
$\mbox{int}\left(c\left[a,b\right]\right)$ and $\mbox{ext}\left(c\left[a,b\right]\right)$
denote the interior and exterior, i.e., the bounded and unbounded
components of $\mathbb{R}^{2}\setminus c\left(\left[a,b\right]\right)$,
respectively. We use the same notations for closed curves $c:\left[a,b\right]\rightarrow G_{2}$,
where $G_{2}$ is any two-dimensional real Banach space. 

We say $\varphi\leq\psi$ for $\varphi,\psi\in C$ if $\varphi\left(s\right)\leq\psi\left(s\right)$
for all $s\in[-1,0]$. Relation $\varphi<\psi$ holds if $\varphi\leq\psi$
and $\varphi\neq\psi$. In addition, $\varphi\ll\psi$ if $\varphi\left(s\right)<\psi\left(s\right)$
for all $s\in[-1,0]$. Relations ``$\geq$'', ``$>$'' and ``$\gg$''
are defined analogously.

The semiflow $\Phi$ is monotone in the following sense. 
\begin{prop}
\label{pro:monotone_dynamical_system} If $\varphi,\psi\in C$ with
$\varphi\leq\psi$ $\left(\varphi\geq\psi\right)$, then $x_{t}^{\varphi}\leq x_{t}^{\psi}$
$\left(x_{t}^{\varphi}\geq x_{t}^{\psi}\right)$ for all $t\geq0$.
If $\varphi<\psi$ $\left(\varphi>\psi\right)$, then $x_{t}^{\varphi}\ll x_{t}^{\psi}$
$\left(x_{t}^{\varphi}\gg x_{t}^{\psi}\right)$ for all $t\geq2$.
If $\varphi\ll\psi$ $\left(\varphi\gg\psi\right)$, then $x_{t}^{\varphi}\ll x_{t}^{\psi}$
$\left(x_{t}^{\varphi}\gg x_{t}^{\psi}\right)$ for all $t\geq0$. 
\end{prop}
The assertion follows easily from the variation-of-constant formula.
For a proof we refer to \cite{Smith}. Note that Proposition \ref{pro:monotone_dynamical_system}
guarantees the positive invariance of $C_{-2,0}$, $C_{0,2}$ and
$C_{-2,2}$. 

The periodic solutions have nice monotone properties (see Theorem
7.1 in \cite{Mallet-Paret_Sell2}) as follows.
\begin{prop}
\label{pro: monotone_type}Suppose $r:\mathbb{R}\rightarrow\mathbb{R}$
is a periodic solution of Eq.~\eqref{eq:eq_general} with minimal
period $\omega>0$. Then $r$ is of monotone type in the following
sense: if $t_{0}<t_{1}<t_{0}+\omega$ are fixed so that $r\left(t_{0}\right)=\min_{t\mathbb{\in R}}r(t)$
and $r\left(t_{1}\right)=\max_{t\mathbb{\in R}}r(t)$, then $\dot{r}\left(t\right)>0$
for $t\in\left(t_{0},t_{1}\right)$ and $\dot{r}\left(t\right)<0$
for $t\in\left(t_{1},t_{0}+\omega\right)$.
\end{prop}
We also need the next technical results. The first one is the direct
consequence of Lemmas VI.4, VI.5 and VI.6 in \cite{Krisztin-Walther-Wu}.
\begin{lem}
\noindent \label{lem:technical result}Let $\mu\geq0$, $\alpha_{0}>0$
and $\alpha_{1}\geq\alpha_{0}$. Let sequences of continuous real
functions $a^{n}$ on $\mathbb{R}$ and continuously differentiable
real functions $z^{n}$ on $\mathbb{R}$, $n\geq0$, be given such
that for all $n\geq0$, $\alpha_{0}\leq a^{n}\left(t\right)\leq\alpha_{1}$
for all $t\in\mathbb{R}$, $z^{n}\left(t\right)\neq0$ for some $t\in\mathbb{R}$,
$V\left(z_{t}^{n}\right)\leq2$ for all $t\in\mathbb{R}$, and $z^{n}$
satisfies 
\[
\dot{z}^{n}\left(t\right)=-\mu z^{n}\left(t\right)+a^{n}\left(t\right)z^{n}\left(t-1\right)
\]
 on $\mathbb{R}$. Let a further continuous real function $a$ on
$\mathbb{R}$ be given so that $a^{n}\rightarrow a$ as $n\rightarrow\infty$
uniformly on compact subsets of $\mathbb{R}$. Then a continuously
differentiable function $z:\mathbb{R}\rightarrow\mathbb{R}$ and a
subsequence $\left(z^{n_{k}}\right)_{k=0}^{\infty}$ of $\left(z^{n}\right)_{n=0}^{\infty}$
can be given such that $z^{n_{k}}\rightarrow z$ and $\dot{z}^{n_{k}}\rightarrow\dot{z}$
as $k\rightarrow\infty$ uniformly on compact subsets of $\mathbb{R}$,
moreover
\[
\dot{z}\left(t\right)=-\mu z\left(t\right)+a\left(t\right)z\left(t-1\right)
\]
for all $t\in\mathbb{R}$.
\end{lem}
The subsequent result shows that Lyapunov functionals can be used
effectively to show that solutions of linear equations cannot decay
too fast at $\infty$. For a proof, see Lemma VI.3 in \cite{Krisztin-Walther-Wu}.
\begin{lem}
\label{lem:technical result 2}Let $\mu\geq0$, $\alpha_{0}>0$ and
$\alpha_{1}\geq\alpha_{0}$. Assume that $t_{0}\in\mathbb{R}$, $a:\left[t_{0}-5,t_{0}\right]\rightarrow\mathbb{R}$
is continuous with $\alpha_{0}\leq a\left(t\right)\leq\alpha_{1}$
for all $t\in\left[t_{0}-5,t_{0}\right]$, $z:\left[t_{0}-6,t_{0}\right]\rightarrow\mathbb{R}$
is continuous, differentiable for $t_{0}-5<t\leq t_{0}$ and satisfies
\eqref{Lyapunov-eq} for $t_{0}-5<t\leq t_{0}$. In addition, assume
that $z_{t_{0}-5}\neq0$ and $V\left(z_{t_{0}-5}\right)\leq2$. Then
there exists $K=K\left(\mu,\alpha_{0},\alpha_{1}\right)>0$ such that
\[
\left\Vert z_{t_{0}-1}\right\Vert \leq K\left\Vert z_{t_{0}}\right\Vert .
\]

\end{lem}
The last result of this section is Lemma I.8 in \cite{Krisztin-Walther-Wu}.
It will be used to abbreviate proofs of smoothness of submanifolds.
\begin{prop}
\label{prop:smoothness of submanifolds}Let $g$ be a $C^{1}$-map
from an $m$-dimensional $C^{1}$-manifold $M$ into a $C^{1}$-manifold
$N$ modeled over a Banach space. If for some $p\in M$, the derivative
$Dg\left(p\right)$ of $g$ at $p$ is injective, then $p$ has an
open neighborhood $U$ in $M$ so that $g\left(U\right)$ is an $m$-dimensional
$C^{1}$-submanifold of $N$. 
\end{prop}

\section{Floquet multipliers and a Poincar\'e return map\label{sec:Floquet-multipliers and invariant manifolds}}

In this section we give a brief introduction to the Floquet theory
regarding periodic solutions which are slowly oscillatory around an
equilibrium. Then we define a Poincar\'e map and collect the most
important properties of its local invariant manifolds. At last we
apply these results to $p$, $q$, $x^{1}$ and $x^{-1}$. The section
is closed by showing that the unstable space of the monodromy operator
corresponding to the periodic orbit $\mathcal{O}_{k}$ is one-dimensional
for both $k\in\left\{ -1,1\right\} $.

\subsection{Floquet multipliers}

Suppose $r:\mathbb{R}\rightarrow\mathbb{R}$ is a periodic solution
of Eq.~\eqref{eq:eq_general} with minimal period $\omega>0$. If
$r$ is slowly oscillatory around an equilibrium (as $p$, $q$, $x^{1}$
or $x^{-1}$ are), then Proposition \ref{pro: monotone_type} implies
that $\omega\in\left(1,2\right)$. Assume that this is the case.

Consider the period map $Q=\Phi\left(\omega,\cdot\right)$ with fixed
point $r_{0}$ and its derivative $M=D_{2}\Phi\left(\omega,r_{0}\right)$
at $r_{0}$. Then $M\varphi=u_{\omega}^{\varphi}$ for all $\varphi\in C$,
where $u^{\varphi}:\left[-1,\infty\right)\rightarrow\mathbb{R}$ is
the solution of the linear variational equation 
\begin{align}
\dot{u}(t) & =-u(t)+f'\left(r\left(t-1\right)\right)u\left(t-1\right)\label{eq:lin.eq.}
\end{align}
with $u_{0}^{\varphi}=\varphi$. $M$ is called the monodromy operator. 

$M$ is a compact operator, $0$ belongs to its spectrum $\sigma=\sigma\left(M\right)$,
and its eigenvalues of finite multiplicity -- the so called Floquet
multipliers -- form $\sigma\left(M\right)\backslash\left\{ 0\right\} $.
The importance of $M$ lies in the fact that we obtain information
about the stability properties of the orbit $\mathcal{O}_{r}=\left\{ r_{t}:\, t\in\mathbb{R}\right\} $
from $\sigma\left(M\right)$.

As $\dot{r}$ is a nonzero solution of the variational equation \eqref{eq:lin.eq.},
$1$ is a Floquet multiplier with eigenfunction $\dot{r}_{0}$. The
periodic orbit $\mathcal{O}_{r}$ is said to be hyperbolic if the
generalized eigenspace of $M$ corresponding to the eigenvalue $1$
is one-dimensional, furthermore there are no Floquet multipliers on
the unit circle besides $1$. 

The paper \cite{Mallet-Paret_Sell} of Mallet-Paret and Sell and Appendix
VII of the monograph \cite{Krisztin-Walther-Wu} of Krisztin, Walther
and Wu confirm the subsequent properties. $\mathcal{O}_{r}$ has a
real Floquet multiplier $\lambda_{1}>1$ with a strictly positive
eigenvector $v_{1}$. The realified generalized eigenspace $C_{<\lambda_{1}}$
associated with the spectral set $\left\{ z\in\sigma:\,\left|z\right|<\lambda_{1}\right\} $
satisfies 
\begin{equation}
C_{<\lambda_{1}}\cap V^{-1}\left(0\right)=\emptyset.\label{C_<lambda_1}
\end{equation}
 Let $C_{\leq\rho}$, $\rho>0$, denote the realified generalized
eigenspace of $M$  associated with the spectral set $\left\{ z\in\sigma:\,\left|z\right|\leq\rho\right\} $.
The set 
\[
\left\{ \rho\in\left(0,\infty\right):\,\sigma\left(M\right)\cap\rho S_{\mathbb{C}}^{1}\neq\emptyset,\, C_{\leq\rho}\cap V^{-1}\left(\left\{ 0,2\right\} \right)=\emptyset\right\} 
\]
is nonempty and has a maximum $r_{M}$. Then 
\begin{equation}
C_{\leq r_{M}}\cap V^{-1}\left(\left\{ 0,2\right\} \right)=\emptyset,\quad C_{r_{M}<}\setminus\left\{ \hat{0}\right\} \subset V^{-1}\left(\left\{ 0,2\right\} \right)\mbox{ and }\mbox{dim}C_{r_{M}<}\leq3,\label{C_r_M<}
\end{equation}
 where $C_{r_{M}<}$ is the realified generalized eigenspace of $M$
associated with the nonempty spectral set $\left\{ z\in\sigma:\,\left|z\right|>r_{M}\right\} $.
It will easily follow from the results of this paper that $\mbox{dim}C_{r_{M}<}=3$
for the periodic solutions $p,q,x^{-1}$ and $x^{1}$, see Remark
\ref{remark}. Recently Mallet-Paret and Nussbaum have shown that
the equality $\mbox{dim}C_{r_{M}<}=3$ holds in general \cite{mallet-paret_nussbaum}.

Let $C_{s}$, $C_{c}$ and $C_{u}$ be the closed subspaces of $C$
chosen so that $C=C_{s}\oplus C_{c}\oplus C_{u}$, $C_{s}$, $C_{c}$
and $C_{u}$ are invariant under $M$, and the spectra $\sigma_{s}\left(M\right)$,
$\sigma_{c}\left(M\right)$ and $\sigma_{u}\left(M\right)$ of the
induced maps $C_{s}\ni x\mapsto Mx\in C_{s}$, $C_{c}\ni x\mapsto Mx\in C_{c}$,
and $C_{u}\ni x\mapsto Mx\in C_{u}$ are contained in $\left\{ \mu\in\mathbb{C}:\,\left|\mu\right|<1\right\} $,
$\left\{ \mu\in\mathbb{C}:\,\left|\mu\right|=1\right\} $ and $\left\{ \mu\in\mathbb{C}:\,\left|\mu\right|>1\right\} $,
respectively. 

As $\mathcal{O}_{r}$ has a real Floquet multiplier $\lambda_{1}>1$,
$C_{u}$ is nontrivial. 

$C_{c}$ is also nontrivial because $\dot{r}_{0}\in C_{c}$. It is
easy to see that the monotone property of $r$ described in Proposition
\ref{pro: monotone_type} and $\omega\in\left(1,2\right)$ imply the
existence of $t\in\mathbb{R}$ with $V\left(\dot{r}_{t}\right)=2$.
As $\mathbb{R}\ni t\rightarrow\dot{r}_{t}\in C$ is periodic, and
$\mathbb{R}\ni t\rightarrow V\left(\dot{r}_{t}\right)$ is monotone
decreasing by Lemma \ref{lem:4_properties_of_V}, it follows that
$V\left(\dot{r}_{t}\right)=2$ for all real $t$. In particular, $V\left(\dot{r}_{0}\right)=2$.
Hence \eqref{C_r_M<} gives that $r_{M}<1$, moreover \eqref{C_<lambda_1}
and \eqref{C_r_M<} together give that $C_{c}\setminus\left\{ \hat{0}\right\} \subset V^{-1}\left(2\right)$.
The nontriviality of $C_{u}$ and $\mbox{dim}C_{r_{M}<}\leq3$ in
addition imply that $C_{c}$ is at most two-dimensional in our case:
\[
C_{c}=\left\{ \begin{array}{ll}
\mathbb{R}\dot{r}_{0}, & \mbox{if }\mathcal{O}_{r}\mbox{ is hyperbolic,}\\
\mathbb{R}\dot{r}_{0}\oplus\mathbb{R}\xi, & \mbox{otherwise,}
\end{array}\right.
\]
where $\xi\in C_{c}\setminus\mathbb{R}\dot{r}_{0}$ provided that
$\mathcal{O}_{r}$ is nonhyperbolic.

\subsection{A Poincar\'e return map}

As above, let $r:\mathbb{R}\rightarrow\mathbb{R}$ be any periodic
solution of Eq.~\eqref{eq:eq_general} which oscillates slowly around
an equilibrium, and let $\omega\in\left(1,2\right)$ denote its minimal
period. 

Fix a $\xi\in C_{c}\setminus\mathbb{R}\dot{r}_{0}$ in case $O_{r}$
is nonhyperbolic and define 
\[
Y=\left\{ \begin{array}{ll}
C_{s}\oplus C_{u}, & \mbox{if }\mathcal{O}_{r}\mbox{ is hyperbolic,}\\
C_{s}\oplus\mathbb{R}\xi\oplus C_{u}, & \mbox{\mbox{if }\ensuremath{\mathcal{O}_{r}}\mbox{ is nonhyperbolic}.}
\end{array}\right.
\]
Then $Y\subset C$ is a hyperplane with codimension $1$. Choose $e^{*}$
to be a continuous linear functional with null space $\left(e^{*}\right)^{-1}\left(0\right)=Y$.
The Hahn--Banach theorem guarantees the existence of $e^{*}$. As
$D_{1}\Phi\left(\omega,r_{0}\right)1=\dot{r}_{0}\notin Y$, and thus
$e^{*}\left(D_{1}\Phi\left(\omega,r_{0}\right)1\right)\neq0$, the
implicit function theorem can be applied to the map $ $
\[
\left(t,\varphi\right)\mapsto e^{*}\left(\Phi\left(t,\varphi\right)-r_{0}\right)
\]
 in a neighborhood of $\left(\omega,r_{0}\right)$. It yields a convex
bounded open neighborhood $N$ of $r_{0}$ in $C$, $\varepsilon\in\left(0,\omega\right)$
and a $C^{1}$-map $\gamma:N\rightarrow\left(\omega-\varepsilon,\omega+\varepsilon\right)$
with $\gamma\left(r_{0}\right)=\omega$ so that for each $\left(t,\varphi\right)\in\left(\omega-\varepsilon,\omega+\varepsilon\right)\times N$,
the segment $x_{t}^{\varphi}$ belongs to $r_{0}+Y$ if and only if
$t=\gamma(\varphi)$ (see \cite{Diekmann et al.}, Appendix I in \cite{Krisztin-Walther-Wu},
\cite{Lani-Wayda}). In addition, by continuity we may assume that
$D_{1}\Phi\left(\gamma(\varphi),\varphi\right)1\notin Y$ for all
$\varphi\in N$. The Poincar\'e return map $P_{Y}$ is defined by
\[
P_{Y}:N\cap\left(r_{0}+Y\right)\ni\varphi\mapsto\Phi\left(\gamma(\varphi),\varphi\right)\in r_{0}+Y.
\]
 Then $P_{Y}$ is continuously differentiable with fixed point $r_{0}$. 

It is convenient to have a formula not only for the derivative $DP_{Y}\left(\varphi\right)$
of $P_{Y}$ at $\varphi\in N\cap\left(r_{0}+Y\right)$, but also for
the derivatives of the iterates of $P_{Y}$. For all $\varphi$ in
the domain of $P_{Y}^{j}$, $j\geq1$, set 
\[
\gamma_{j}\left(\varphi\right)=\Sigma_{k=0}^{j-1}\gamma\left(P_{Y}^{k}\left(\varphi\right)\right).
\]
Then 
\[
DP_{Y}^{j}\left(\varphi\right)\eta=D_{1}\Phi\left(\gamma_{j}(\varphi),\varphi\right)\gamma_{j}'\left(\varphi\right)\eta+D_{2}\Phi\left(\gamma_{j}(\varphi),\varphi\right)\eta
\]
for all $\eta\in Y$. Differentiation of the equation $e^{*}\left(\Phi\left(\gamma_{j}(\varphi),\varphi\right)-r_{0}\right)=0$
yields that 
\[
\gamma_{j}'\left(\varphi\right)\eta=-\frac{e^{*}\left(D_{2}\Phi\left(\gamma_{j}(\varphi),\varphi\right)\eta\right)}{e^{*}\left(D_{1}\Phi\left(\gamma_{j}(\varphi),\varphi\right)1\right)},
\]
and therefore
\begin{equation}
DP_{Y}^{j}\left(\varphi\right)\eta=D_{2}\Phi\left(\gamma_{j}(\varphi),\varphi\right)\eta-\frac{e^{*}\left(D_{2}\Phi\left(\gamma_{j}(\varphi),\varphi\right)\eta\right)}{e^{*}\left(D_{1}\Phi\left(\gamma_{j}(\varphi),\varphi\right)1\right)}D_{1}\Phi\left(\gamma_{j}(\varphi),\varphi\right)1\label{derivativse of iteratives of P_Y}
\end{equation}
for all $\eta\in Y$.

Let $\sigma\left(P_{Y}\right)$ and $\sigma\left(M\right)$ denote
the spectra of $DP_{Y}\left(r_{0}\right):Y\rightarrow Y$ and the
monodromy operator, respectively. We obtain the following result from
Theorem XIV.4.5 in \cite{Diekmann et al.}. 
\begin{lem}
~

\noindent (i) $\sigma\left(P_{Y}\right)\setminus\left\{ 0,1\right\} =\sigma\left(M\right)\setminus\left\{ 0,1\right\} $,
and for every $\lambda\in\sigma\left(M\right)\setminus\left\{ 0,1\right\} $,
the projection along $\mathbb{R}\dot{r}_{0}$ onto $Y$ defines an
isomorphism from the realified generalized eigenspace of $\lambda$
and $M$ onto the realified generalized eigenspace of $\lambda$ and
$DP_{Y}\left(r_{0}\right)$.

\noindent (ii) If the generalized eigenspace $G\left(1,M\right)$
associated with $1$ and $M$ is one-dimensional, then $1\notin\sigma\left(P_{Y}\right)$.

\noindent (iii) If $\dim G\left(1,M\right)>1$, then $1\in\sigma\left(P_{Y}\right)$,
and the realified generalized eigenspaces $G_{\mathbb{R}}\left(1,M\right)$
and $G_{\mathbb{R}}\left(1,P_{Y}\right)$ associated with $1$ and
$M$ and with $1$ and $DP_{Y}\left(r_{0}\right)$, respectively,
satisfy 
\[
G_{\mathbb{R}}\left(1,P_{Y}\right)=Y\cap G_{\mathbb{R}}\left(1,M\right)\quad\mbox{and}\quad G_{\mathbb{R}}\left(1,M\right)=\mathbb{R}\dot{r}_{0}\oplus G_{\mathbb{R}}\left(1,P_{Y}\right).
\]

\end{lem}
In our case, the special choice of $Y$ implies the following corollary. 
\begin{cor}
\label{cor:spectral corespondences}~

\noindent (i) $C_{s}$ and $C_{u}$ are invariant under $DP_{Y}\left(r_{0}\right)$,
and the spectra $\sigma_{s}\left(P_{Y}\right)$ and $\sigma_{u}\left(P_{Y}\right)$
of the induced maps $C_{s}\ni x\mapsto DP_{Y}\left(r_{0}\right)x\in C_{s}$
and $C_{u}\ni x\mapsto DP_{Y}\left(r_{0}\right)x\in C_{u}$ are contained
in $\left\{ \mu\in\mathbb{C}:\,\left|\mu\right|<1\right\} $ and $\left\{ \mu\in\mathbb{C}:\,\left|\mu\right|>1\right\} $,
respectively.

\noindent (ii) If $M$ has an eigenfunction $v$ corresponding to
a simple eigenvalue $\lambda\in\sigma\left(M\right)\setminus\left\{ 0,1\right\} $,
then $v$ is an eigenfunction of $DP_{Y}\left(r_{0}\right)$ corresponding
to the same eigenvalue. 

\noindent (iii) If $\mathcal{O}_{r}$ is nonhyperbolic, then $\xi$
is an eigenfunction of $DP_{Y}\left(r_{0}\right)$, and it corresponds
to an eigenvalue with absolute value $1$.
\end{cor}
In particular, if $\lambda_{1}$ is a simple Floquet multiplier, then
the strictly positive eigenfunction $v_{1}$ of $M$ corresponding
to $\lambda_{1}$ is also an eigenfunction of $DP_{Y}\left(r_{0}\right)$
corresponding to $\lambda_{1}$.

In case $\mathcal{O}_{r}$ is hyperbolic, then according to Theorem
I.3 in Appendix I of \cite{Krisztin-Walther-Wu}, there exist convex
open neighborhoods $N_{s}$, $N_{u}$ of $\hat{0}$ in $C_{s}$, $C_{u}$,
respectively, and a $C^{1}$-map $w_{u}:N_{u}\rightarrow C_{s}$ with
range in $N_{s}$ so that $w_{u}\left(\hat{0}\right)=\hat{0}$, $Dw_{u}\left(\hat{0}\right)=0$,
and the submanifold 
\[
\mathcal{W}_{loc}^{u}\left(P_{Y},r_{0}\right)=\left\{ r_{0}+\chi+w_{u}\left(\chi\right):\,\chi\in N_{u}\right\} 
\]
of $r_{0}+Y$ is equal to the set 
\begin{eqnarray*}
\left\{ \varphi\in r_{0}+N_{s}+N_{u}:\,\mbox{there is a trajectory }\left(\varphi_{n}\right)_{-\infty}^{0}\mbox{ of }P_{Y}\mbox{ with }\varphi_{0}=\varphi\mbox{ such that }\right.\\
\left.\varphi_{n}\in r_{0}+N_{s}+N_{u}\mbox{ for all }n\leq0\mbox{ and }\varphi_{n}\rightarrow r_{0}\mbox{ as }n\rightarrow-\infty\right\} .
\end{eqnarray*}
$\mathcal{W}_{loc}^{u}\left(P_{Y},r_{0}\right)$ is called a local
unstable manifold of $P_{Y}$ at $r_{0}$.

The unstable set of the orbit $\mathcal{O}_{r}$ is defined as the
forward extension of $\mathcal{W}_{loc}^{u}\left(P_{Y},r_{0}\right)$
in time: 
\begin{equation}
\mathcal{W}^{u}\left(\mathcal{O}_{r}\right)=\Phi\left(\left[0,\infty\right)\times\mathcal{W}_{loc}^{u}\left(P_{Y},r_{0}\right)\right).\label{unstable set is forwad extension of unstable manifold-1}
\end{equation}
 If $\mathcal{O}_{r}$ is hyperbolic, then 
\[
\mathcal{W}^{u}\left(\mathcal{O}_{r}\right)=\left\{ x_{0}:\ x:\mathbb{R}\rightarrow\mathbb{R}\mbox{ is a solution of \eqref{eq:eq_general}, }\alpha\left(x\right)\mbox{ exists and }\alpha\left(x\right)=\mathcal{O}_{r}\right\} .
\]

If $\mathcal{O}_{r}$ is hyperbolic, then by Theorem I.2 in \cite{Krisztin-Walther-Wu},
there are convex open neighborhoods $N_{s}$, $N_{u}$ of $\hat{0}$
in $C_{s}$, $C_{u}$, respectively, and a $C^{1}$-map $w_{s}:N_{s}\rightarrow C_{u}$
with range in $N_{u}$ such that $w_{s}\left(\hat{0}\right)=\hat{0}$,
$Dw_{s}\left(\hat{0}\right)=0$, and 
\[
\mathcal{W}_{loc}^{s}\left(P_{Y},r_{0}\right)=\left\{ r_{0}+\chi+w_{s}\left(\chi\right):\,\chi\in N_{s}\right\} 
\]
 is equal to 
\begin{eqnarray*}
\left\{ \varphi\in r_{0}+N_{s}+N_{u}:\,\mbox{there is a trajectory }\left(\varphi_{n}\right)_{0}^{\infty}\mbox{ of }P_{Y}\mbox{ in }\right.\\
\left.r_{0}+N_{s}+N_{u}\mbox{ with }\varphi_{0}=\varphi\mbox{ and }\varphi_{n}\rightarrow r_{0}\mbox{ as }n\rightarrow\infty\right\} .
\end{eqnarray*}
 $\mathcal{W}_{loc}^{s}\left(P_{Y},r_{0}\right)$ is a local stable
manifold of $P_{Y}$ at $r_{0}$. It is a $C^{1}$-submanifold of
$r_{0}+Y$ with codimension $\mbox{dim}C_{u}$, and it is a $C^{1}$-submanifold
of $C$ with codimension $\mbox{dim}C_{u}+1$.

In case $\mathcal{O}_{r}$ is nonhyperbolic, we need a local center-stable
manifold $\mathcal{W}_{loc}^{sc}\left(P_{Y},r_{0}\right)$ of $P_{Y}$
at $r_{0}$. According to Theorem II.1 in \cite{Krisztin-Walther-Wu},
there exist convex open neighborhoods $N_{sc}$ and $N_{u}$ of $\hat{0}$
in $C_{s}\oplus\mathbb{R}\xi$ and $C_{u}$, respectively, and a $C^{1}$-map
$w_{sc}:N_{sc}\rightarrow C_{u}$ such that $w_{sc}\left(\hat{0}\right)=\hat{0}$,
$Dw_{sc}\left(\hat{0}\right)=0$, $w_{sc}\left(N_{sc}\right)\subset N_{u}$
and the local center-stable manifold 
\[
\mathcal{W}_{loc}^{sc}\left(P_{Y},r_{0}\right)=\left\{ r_{0}+\chi+w_{sc}\left(\chi\right):\,\chi\in N_{sc}\right\} 
\]
satisfies 
\[
\bigcap_{n=0}^{\infty}P_{Y}^{-1}\left(r_{0}+N_{sc}+N_{u}\right)\subset\mathcal{W}_{loc}^{sc}\left(P_{Y},r_{0}\right).
\]
Note that $\mathcal{W}_{loc}^{sc}\left(P_{Y},r_{0}\right)$ is also
a $C^{1}$-submanifold of $r_{0}+Y$ with codimension $\mbox{dim}C_{u}$,
and it is a $C^{1}$-submanifold of $C$ with codimension $\mbox{dim}C_{u}+1$.
\begin{prop}
\label{qvarphi dot near r_0}One may choose the neighborhoods $N_{s}$
and $N_{sc}$ so small in the definitions of $\mathcal{W}_{loc}^{s}\left(P_{Y},r_{0}\right)$,
$\mathcal{W}_{loc}^{sc}\left(P_{Y},r_{0}\right)$, respectively, such
that for all $\varphi$ in $\mathcal{W}_{loc}^{s}\left(P_{Y},r_{0}\right)\cap\mathcal{A}$
and in $\mathcal{W}_{loc}^{sc}\left(P_{Y},r_{0}\right)\cap\mathcal{A}$,
$\dot{\varphi}\notin Y$ and $V\left(\dot{\varphi}\right)\geq2$.
Analogously, one may suppose that $\dot{\varphi}\notin Y$ for all
$\varphi\in\mathcal{W}_{loc}^{u}\left(P_{Y},r_{0}\right)\cap\mathcal{A}$.\end{prop}
\begin{proof}
Recall that the $C$-norm and the $C^{1}$-norm are equivalent on
the global attractor $\mathcal{A}$. Hence for all $\varphi\in\mathcal{A}$
with small $\left\Vert \varphi-r_{0}\right\Vert $, $\dot{\varphi}\notin Y$
follows from $\dot{r}_{0}\notin Y$, furthermore $V\left(\dot{\varphi}\right)\geq2$
follows from $V\left(\dot{r}_{0}\right)=2$ and the lower semicontinuity
of $V$.
\end{proof}
The next result is an immediate consequence of Proposition I.7 in
\cite{Krisztin-Walther-Wu} combined with characterizations of the
local stable and center-stable manifolds given by Theorems I.2 and
II.1 in \cite{Krisztin-Walther-Wu}.
\begin{prop}
\label{prop:trajectory of P_Y} Let $\mathcal{W}$ denote a local
stable manifold $\mathcal{W}_{loc}^{s}\left(P_{Y},r_{0}\right)$ if
$\mathcal{O}_{r}$ is hyperbolic, and let $\mathcal{W}$ be a local
center-stable manifold $\mathcal{W}_{loc}^{sc}\left(P_{Y},r_{0}\right)$
otherwise. Let $\varphi\in C$ be given such that $\Phi\left(t,\varphi\right)\rightarrow\mathcal{O}_{r}$
as $t\rightarrow\infty$. Then there exist $T\geq0$ and a trajectory
$\left(\varphi^{n}\right)_{n=0}^{\infty}$ of $P_{Y}$ in $\mathcal{W}$
such that $\varphi^{0}=\Phi\left(T,\varphi\right)$ and $\varphi^{n}\rightarrow r_{0}$
as $n\rightarrow\infty$.
\end{prop}

\subsection{Examples}

Consider the case when $r$ is the LSOP solution $p$ given by Theorem
A. Theorem A states that $\mathcal{O}_{p}$ is hyperbolic, and has
two real and simple Floquet multipliers outside the unit circle. Hence
$C_{c}=\mathbb{R}\dot{p}_{0}$ and 
\[
C_{u}=\left\{ c_{1}v_{1}+c_{2}v_{2}:\, c_{1},c_{2}\in\mathbb{R}\right\} ,
\]
where $v_{1}$ is a positive eigenfunction corresponding to $M$ and
the leading real eigenvalue $\lambda_{1}>1$, and $v_{2}$ is an eigenfunction
corresponding to $M$ and the eigenvalue $\lambda_{2}$ with $1<\lambda_{2}<\lambda_{1}$.
For the solution $u^{v_{2}}:\left[-1,\infty\right)\rightarrow\mathbb{R}$
of the linear variational equation \eqref{eq:lin.eq.} with initial
segment $v_{2}$, $V\left(u_{t}^{v_{2}}\right)=2$ for all $t\geq0$.
For both $i\in\left\{ 1,2\right\} $, $\lambda_{i}$ is an eigenvalue
of $DP_{Y}\left(p_{0}\right)$ with the eigenvector $v_{i}$. 

The local unstable manifold $\mathcal{W}_{loc}^{u}\left(P_{Y},p_{0}\right)$
of the Poincar\'e map $P_{Y}$ at $p_{0}$ is a two-dimensional $C^{1}$-submanifold
of $p_{0}+Y$.

We will use the subsequent technical result.
\begin{prop}
\label{pro: properties of local unstable manifold at r_0}One may
choose $N_{u}$ so small that the tangent space $T_{\varphi}\mathcal{W}_{loc}^{u}\left(P_{Y},p_{0}\right)$
has a strictly positive element for all $\varphi\in\mathcal{W}_{loc}^{u}\left(P_{Y},p_{0}\right)$.\end{prop}
\begin{proof}
By decreasing $N_{u}$ if necessary, we can achieve that $v_{1}+Dw_{u}\left(\chi\right)v_{1}\gg\hat{0}$
for all $\chi\in N_{u}$, where $v_{1}$ is a fixed positive eigenfunction
corresponding to the leading eigenvalue $\lambda_{1}$ of $DP_{Y}\left(p_{0}\right)$.
Let $\varphi\in\mathcal{W}_{loc}^{u}\left(P_{Y},p_{0}\right)$ be
arbitrary and choose $\chi^{\varphi}\in N_{u}$ with $\varphi=p_{0}+\chi^{\varphi}+w_{u}\left(\chi^{\varphi}\right)$.
Then for all $t$ in an open interval $I\subset\mathbb{R}$ containing
$0$, $\gamma\left(t\right)=p_{0}+\chi^{\varphi}+tv_{1}+w_{u}\left(\chi^{\varphi}+tv_{1}\right)$
is defined. Moreover, $\gamma:I\rightarrow\mathcal{W}_{loc}^{u}\left(P_{Y},p_{0}\right)$
is a $C^{1}$-curve with $\gamma\left(0\right)=\varphi$ and 
\[
T_{\varphi}\mathcal{W}_{loc}^{u}\left(P_{Y},p_{0}\right)\ni\gamma'\left(0\right)=v_{1}+Dw_{u}\left(\chi^{\varphi}\right)v_{1}\gg\hat{0}.
\]

\end{proof}
We plan to consider other periodic orbits oscillating slowly around
an equilibrium, but keep the same notations for simplicity ($\omega$
for the minimal period, $P_{Y}$ for the Poincar\'e map, $\lambda_{i}$,
$i\geq1$, for the Floquet multipliers, $v_{i}$, $i\geq1$, for eigenvectors,
and so on). It will be clear from the context which periodic orbit
we refer to. 

Theorem A gives a second LSOP solution $q:\mathbb{R}\rightarrow\mathbb{R}$.
$\mathcal{O}_{q}$ is hyperbolic, and it has exactly one simple Floquet
multiplier outside the unit circle, which is real and greater than
$1$. This leading eigenvalue will be also denoted by $\lambda_{1}$,
but it differs from the leading Floquet multiplier of $\mathcal{O}_{p}$.
To $\lambda_{1}$ there corresponds a positive eigenfunction $v_{1}$
(different from the previous $v_{1}$). Hence for $r=q$, $C_{c}=\mathbb{R}\dot{q}_{0}$
and $C_{u}=\mathbb{R}v_{1}$.  The local stable manifold $\mathcal{W}_{loc}^{s}\left(P_{Y},q_{0}\right)$
of $P_{Y}$ at $q_{0}$ is a $C^{1}$-submanifold of $q_{0}+Y$ with
codimension $1$, and a $C^{1}$-submanifold of $C$ with codimension
$2$. We have the tangent space $T_{q_{0}}\mathcal{W}_{loc}^{s}\left(P_{Y},q_{0}\right)=C_{s}$
at $q_{0}$ in $q_{0}+Y$. 

Recall that there exist periodic solutions $x^{1}:\mathbb{R}\rightarrow\mathbb{R}$
and $x^{-1}:\mathbb{R}\rightarrow\mathbb{R}$ of Eq.\,\eqref{eq:eq_general}
oscillating slowly around $\xi_{1}$ and $\xi_{-1}$ with ranges in
$\left(0,\xi_{2}\right)$ and $\left(\xi_{-2},0\right)$, respectively,
so that the ranges $x^{1}(\mathbb{R})$ and $x^{-1}(\mathbb{R})$
are maximal in the sense that $x^{1}(\mathbb{R})\supset x(\mathbb{R})$
for all periodic solutions $x$ oscillating slowly around $\xi_{1}$
with ranges in $(0,\xi_{2})$; and analogously for $x^{-1}$. We do
not know whether the corresponding periodic orbits, $\mathcal{O}_{1}$
and $\mathcal{O}_{-1}$, are hyperbolic or not.
\begin{prop}
\label{dimC_u=00003D1}For both periodic orbits $\mathcal{O}_{1}$
and $\mathcal{O}_{-1}$, $\dim C_{u}=1$.\end{prop}
\begin{proof}
We give a proof for $\mathcal{O}_{1}$. As $\mathcal{O}_{1}$ has
a Floquet multiplier $\lambda_{1}>1$, it is clear that $\dim C_{u}\geq1.$

Let $\mathcal{W}$ denote the local stable manifold $\mathcal{W}_{loc}^{s}\left(P_{Y},x_{0}^{1}\right)$
if $\mathcal{O}_{1}$ is hyperbolic, and let $\mathcal{W}$ be the
local center-stable manifold $\mathcal{W}_{loc}^{sc}\left(P_{Y},x_{0}^{1}\right)$
otherwise. Then $\mathcal{W}$ is a $C^{1}$-submanifold of $x_{0}^{1}+Y$
with $T_{x_{0}^{1}}\mathcal{W}=C_{s}$ if $\mathcal{O}_{1}$ is hyperbolic,
and with $T_{x_{0}^{1}}\mathcal{W}=C_{s}\oplus\mathbb{R}\xi$ if $\mathcal{O}_{1}$
is nonhyperbolic.

By Theorem B, there exists $\eta\in\mathcal{W}^{u}\left(\mathcal{O}_{p}\right)$
so that $x_{t}^{\eta}\rightarrow\mathcal{O}_{1}$ as $t\rightarrow\infty$.
Then Proposition \ref{prop:trajectory of P_Y} guarantees the existence
of a sequence $\left(t_{n}\right)_{n=0}^{\infty}$ in $\mathbb{R}$
with $t_{n}\rightarrow\infty$ as $n\rightarrow\infty$ such that
$x_{t_{n}}^{\eta}\in\mathcal{W}\setminus\left\{ x_{0}^{1}\right\} $
for all $n\geq0$ and $ $$x_{t_{n}}^{\eta}\rightarrow x_{0}^{1}$
as $n\rightarrow\infty$. 

We introduce the notation $y^{n}:\mathbb{R}\rightarrow\mathbb{R}$,
$n\geq0$, for the function obtained from $x^{\eta}$ by time shift
so that $y_{0}^{n}=x_{t_{n}}^{\eta}$. Then $y^{n}\left(t\right)\rightarrow x^{1}\left(t\right)$
as $n\rightarrow\infty$ for all $t\in\mathbb{R}$ by the continuity
of the flow $\Phi_{\mathcal{A}}$. Since $x^{\eta}$ is a bounded
solution of Eq.\,\eqref{eq:eq_general}, the solutions $y^{n}$ are
uniformly bounded on $\mathbb{R}$, and Eq.\,\eqref{eq:eq_general}
gives a uniform bound for their derivatives. By applying the Arzel\`{a}\textendash{}Ascoli
theorem successively on the intervals $\left[-j,j\right]$, $j\geq1$,
we obtain strictly increasing maps $\chi_{j}:\mathbb{N}\rightarrow\mathbb{N}$,
$1\leq j\in\mathbb{N}$, so that for every integer $j\geq1$, the
subsequence $\left(y^{\chi_{1}\circ\ldots\circ\chi_{j}\left(k\right)}\right)_{k=0}^{\infty}$
converges uniformly on $\left[-j,j\right]$. By diagonalization, set
$\chi\left(k\right)=\chi_{1}\circ\ldots\circ\chi_{k}\left(k\right)$
and consider the subsequence $\left(y^{n_{k}}\right)_{k=0}^{\infty}=\left(y^{\chi\left(k\right)}\right)_{k=0}^{\infty}$.
Then $y^{n_{k}}\rightarrow x^{1}$ as $k\rightarrow\infty$ uniformly
on all compact subsets of $\mathbb{R}$. 

Define
\[
z^{k}\left(t\right)=\frac{y^{n_{k}}\left(t\right)-x^{1}\left(t\right)}{\left\Vert x_{t_{n_{k}}}^{\eta}-x_{0}^{1}\right\Vert }\quad\mbox{for all }k\geq0\mbox{ and }t\in\mathbb{R}.
\]
Then $z^{k}$, $k\geq0$, satisfies the equation $\dot{z}^{k}\left(t\right)=-z^{k}\left(t\right)+a_{k}\left(t\right)z^{k}\left(t-1\right)$
on $\mathbb{R}$, where the coefficient function $a_{k}$ is defined
by 
\[
a_{k}:\mathbb{R}\ni t\mapsto\int_{0}^{1}f'\left(sy^{n_{k}}\left(t-1\right)+\left(1-s\right)x^{1}\left(t-1\right)\right)\mbox{d}s\in\mathbb{R}^{+},\quad k\geq0.
\]
Note that there are constants $\alpha_{1}\geq\alpha_{0}>0$ independent
of $k$ and $t$ such that $\alpha_{0}\leq a_{k}\left(t\right)\leq\alpha_{1}$
for all $ $$k\geq0$ and $t\in\mathbb{R}$, moreover, $a_{k}\rightarrow a$
as $k\rightarrow\infty$ uniformly on compact subsets of $\mathbb{R}$,
where 
\[
a:\mathbb{R}\ni t\mapsto f'\left(x^{1}\left(t-1\right)\right)\in\mathbb{R}^{+}.
\]
In addition, observe that for all $ $$k\geq0$ and $t\in\mathbb{R}$,
$z_{t}^{k}\neq\hat{0}$ because $y_{0}^{n_{k}}=x_{t_{n_{k}}}^{\eta}\neq x_{0}^{1}$
and the flow $\Phi_{\mathcal{A}}$ is injective. Hence $V\left(z_{t}^{k}\right)$
is defined and equals $2$ for all $ $$k\geq0$ and $t\in\mathbb{R}$
by Proposition 8.3 in \cite{Krisztin-Vas}. Lemma \ref{lem:technical result}
then implies the existence of a continuously differentiable function
$z:\mathbb{R}\rightarrow\mathbb{R}$ and a subsequence $\left(z^{k_{l}}\right)_{l=0}^{\infty}$
of $\left(z^{k}\right)_{k=0}^{\infty}$ such that $z^{k_{l}}\rightarrow z$
and $\dot{z}^{k_{l}}\rightarrow\dot{z}$ as $k\rightarrow\infty$
uniformly on compact subsets of $\mathbb{R}$, moreover 
\begin{equation}
\dot{z}\left(t\right)=-z\left(t\right)+a\left(t\right)z\left(t-1\right)\label{eq:*-1}
\end{equation}
 for all real $t$.

We claim that $z_{0}\neq\hat{0}$ and
\[
z_{0}\in T_{x_{0}^{1}}\mathcal{W}=\left\{ \begin{array}{ll}
C_{s}, & \mbox{if }\mathcal{O}_{r}\mbox{ is hyperbolic,}\\
C_{s}\oplus\mathbb{R}\xi, & \mbox{otherwise.}
\end{array}\right.
\]
Consider the map $w=w_{s}$ if $\mathcal{O}_{1}$ is hyperbolic, and
the map $w=w_{sc}$ otherwise. Choose $\chi^{l}\in T_{x_{0}^{1}}\mathcal{W}$,
$l\geq0$, with $\chi^{l}\rightarrow\hat{0}$ as $l\rightarrow\infty$
so that $x_{t_{n_{k_{l}}}}^{\eta}=x_{0}^{1}+\chi^{l}+w\left(\chi^{l}\right)$
for all $l\geq0$. Then 
\[
z_{0}=\lim_{l\rightarrow\infty}z_{0}^{k_{l}}=\lim_{l\rightarrow\infty}\frac{x_{t_{n_{k_{l}}}}^{\eta}-x_{0}^{1}}{\left\Vert x_{t_{n_{k_{l}}}}^{\eta}-x_{0}^{1}\right\Vert }=\lim_{l\rightarrow\infty}\frac{\chi^{l}+w\left(\chi^{l}\right)}{\left\Vert \chi^{l}+w\left(\chi^{l}\right)\right\Vert }.
\]
As $z_{0}$ is the limit of unit vectors, it is clearly nontrivial.
$Dw\left(\hat{0}\right)=0$ implies that $\lim_{l\rightarrow\infty}w\left(\chi^{l}\right)/\left\Vert \chi^{l}\right\Vert =\hat{0}$
and thus 
\[
\lim_{l\rightarrow\infty}\frac{w\left(\chi^{l}\right)}{\left\Vert \chi^{l}+w\left(\chi^{l}\right)\right\Vert }=\lim_{l\rightarrow\infty}\frac{\frac{w\left(\chi^{l}\right)}{\left\Vert \chi^{l}\right\Vert }}{\left\Vert \frac{\chi^{l}}{\left\Vert \chi^{l}\right\Vert }+\frac{w\left(\chi^{l}\right)}{\left\Vert \chi^{l}\right\Vert }\right\Vert }=\hat{0}
\]
and $ $
\[
\lim_{l\rightarrow\infty}\frac{\left\Vert \chi^{l}\right\Vert }{\left\Vert \chi^{l}+w\left(\chi^{l}\right)\right\Vert }=\lim_{l\rightarrow\infty}\frac{1}{\left\Vert \frac{\chi^{l}}{\left\Vert \chi^{l}\right\Vert }+\frac{w\left(\chi^{l}\right)}{\left\Vert \chi^{l}\right\Vert }\right\Vert }=1.
\]
We obtain that

\begin{align*} 
\underset{{\normalsize z_0}}{\underbrace{\frac{\chi^{l}+w\left(\chi^{l}\right)}{\left\Vert \chi^{l}+w\left(\chi^{l}\right)\right\Vert }}_{\downarrow}} & =\frac{\chi^{l}}{\left\Vert \chi^{l}\right\Vert }  \underset{1}{\underbrace{\frac{\left\Vert \chi^{l}\right\Vert }{\left\Vert \chi^{l}+w\left(\chi^{l}\right)\right\Vert }}_{\downarrow}}  +  \underset{0}{\underbrace{\frac{w\left(\chi^{l}\right)}{\left\Vert \chi^{l}+w\left(\chi^{l}\right)\right\Vert }}_{\downarrow}}
\end{align*}as $l\rightarrow\infty$. Then the limit $\lim_{l\rightarrow\infty}\chi^{l}/\left\Vert \chi^{l}\right\Vert $
necessarily exists too, and $ $ 
\[
z_{0}=\lim_{l\rightarrow\infty}\frac{\chi^{l}+w\left(\chi^{l}\right)}{\left\Vert \chi^{l}+w\left(\chi^{l}\right)\right\Vert }=\lim_{l\rightarrow\infty}\frac{\chi^{l}}{\left\Vert \chi^{l}\right\Vert }\in T_{x_{0}^{1}}\mathcal{W}\subset Y.
\]

Since $V\left(z_{0}^{k_{l}}\right)=2$ for all $l\geq0$, the lower-semicontinuity
of $V$ proved in Lemma \ref{lem: continuity_of_V} implies that $V\left(z_{0}\right)\leq\liminf_{l\rightarrow\infty}V\left(z_{0}^{k_{l}}\right)=2$.
Recall that $\dot{x}_{0}^{1}\in C_{c}$ also belongs to $V^{-1}\left(\left\{ 0,2\right\} \right)$,
moreover, $\dot{x}_{0}^{1}\notin Y$. Thus $\dot{x}_{0}^{1}$ and
$z_{0}$ are linearly independent elements of $\left(C_{s}\oplus C_{c}\right)\cap V^{-1}\left(\left\{ 0,2\right\} \right)$.
In consequence, result \eqref{C_r_M<} gives that $C_{u}$ is at most
one-dimensional. 

The proof is analogous for $\mathcal{O}_{-1}$.
\end{proof}
The previous result implies that if $\mathcal{O}_{k}$, $k\in\left\{ -1,1\right\} $,
is hyperbolic, then the local stable manifold $\mathcal{W}_{loc}^{s}\left(P_{Y},x_{0}^{k}\right)$
of $P_{Y}$ at $x_{0}^{k}$ is a $C^{1}$-submanifold of $x_{0}^{k}+Y$
with codimension $1$ and with tangent space $T_{x_{0}^{k}}\mathcal{W}_{loc}^{s}\left(P_{Y},x_{0}^{k}\right)=C_{s}$
at $x_{0}^{k}$. It is a $C^{1}$-submanifold of $C$ with codimension
$2$.

Similarly, if $\mathcal{O}_{k}$, $k\in\left\{ -1,1\right\} $, is
nonhyperbolic, then the local center-stable manifold $\mathcal{W}_{loc}^{sc}\left(P_{Y},x_{0}^{k}\right)$
of $P_{Y}$ at $x_{0}^{k}$ is a $C^{1}$-submanifold of $x_{0}^{k}+Y$
with codimension $1$ and with tangent space $T_{x_{0}^{k}}\mathcal{W}_{loc}^{sc}\left(P_{Y},x_{0}^{k}\right)=C_{s}\oplus\mathbb{R}\xi$
at $x_{0}^{k}$. It is also a $C^{1}$-submanifold of $C$ with codimension
$2$.
\begin{rem}
\label{remark} We see from the proof of Proposition \ref{dimC_u=00003D1}
that for $r=x^{k}$, $k\in\left\{ -1,1\right\} $, $C_{r_{M}<}$ admits
at least three linearly independent elements: $v_{1}\in C_{u}$, $\dot{x}_{0}^{k}\in C_{c}$
and $z_{0}\in C_{s}\oplus C_{c}$. As $C_{r_{M}}$ is at most three-dimensional
by \eqref{C_r_M<}, we conclude that $\dim C_{r_{M}<}=3$. A similar
reasoning confirms the same equality for $r=q$. It is obviuos that
the dimension of $C_{r_{M}<}$ is maximal also in the case $r=p$,
as $\mathcal{O}_{p}$ has two Floquet-multipliers outside the unit
circle. These observations are in accordance with the recent result
\cite{mallet-paret_nussbaum} of Mallet-Paret and Nussbaum stating
that $\dim C_{r_{M}<}=3$ in more general situations.
\end{rem}

\section{The Proof of Theorem \ref{main_theorem_3}}

Note that each $\varphi$ in the unstable set $\mathcal{W}^{u}\left(\mathcal{O}_{p}\right)$
arises in the form $\varphi=\Phi\left(t,\psi\right)$, where $\psi\in\mathcal{W}_{loc}^{u}\left(P_{Y},p_{0}\right)$
and $t>1$. Indeed, 
\[
\mathcal{W}^{u}\left(\mathcal{O}_{r}\right)=\Phi\left(\left[0,\infty\right)\times\mathcal{W}_{loc}^{u}\left(P_{Y},r_{0}\right)\right),\tag{3.5}
\]
and from each $\psi\in\mathcal{W}_{loc}^{u}\left(P_{Y},p_{0}\right)$
we can start a backward trajectory $\left(\psi^{n}\right)_{-\infty}^{0}$
of $P_{Y}$ in $\mathcal{W}_{loc}^{u}\left(P_{Y},p_{0}\right)$ converging
to $p_{0}$ as $n\rightarrow-\infty$. As the first part of the proof
of Theorem \ref{main_theorem_3}, we are going to show in Proposition
\ref{prop:certain subsets are submanifolds} that for all $t>1$ and
$\psi\in\mathcal{W}_{loc}^{u}\left(P_{Y},p_{0}\right)$, $\varphi=\Phi\left(t,\psi\right)$
belongs to a subset $W_{t,\psi,\varepsilon}$ of $\mathcal{W}^{u}\left(\mathcal{O}_{p}\right)$
that is a three-dimensional submanifold of $C$. This implies that
$\mathcal{W}^{u}\left(\mathcal{O}_{p}\right)$ is an immersed submanifold
of $C$. The proof of Proposition \ref{prop:certain subsets are submanifolds}
is based on \eqref{unstable set is forwad extension of unstable manifold-1},
the differentiabilty of $\Phi|_{\left(1,\infty\right)\times C}$ and
the injectivity of $D_{2}\Phi\left(t,\varphi\right)$ for $t\geq0$. 

However, it does not follow immediately that $ $$\mathcal{W}^{u}\left(\mathcal{O}_{p}\right)$
is an embedded $C^{1}$-submani$\-$fold of $C$. We also need to
show for any $\varphi\in\mathcal{W}^{u}\left(\mathcal{O}_{p}\right)$
the existence of a ball $B$ in $C$ centered at $\varphi$ such that
\begin{equation}
\mathcal{W}^{u}\left(\mathcal{O}_{p}\right)\cap B=W_{t,\psi,\varepsilon}\cap B.\label{equality with ball B}
\end{equation}
To do this, we will give a sequence of further auxiliary results right
after Proposition \ref{prop:certain subsets are submanifolds}. We
will introduce a projection $\pi_{3}$ from $C$ into $\mathbb{R}^{3}$,
and use the special properties of the Lyapunov fuctional $V$ to show
that $\pi_{3}$ is injective on $\mathcal{W}^{u}\left(\mathcal{O}_{p}\right)$
and on the tangent spaces of $W_{t,\psi,\varepsilon}$. These results
will easily imply \eqref{equality with ball B}.

Afterwards we offer a smooth global graph representation for $\mathcal{W}^{u}\left(\mathcal{O}_{p}\right)$
in order to indicate the simplicity of its structure. The smoothness
of the sets $C_{-2}^{p},\, C_{0}^{p}$ and $C_{2}^{p}$ then follows
at once because they are open subsets of $\mathcal{W}^{u}\left(\mathcal{O}_{p}\right)$.
At last we show that the semiflow induced by the solution operator
$\Phi$ extends to a $C^{1}$-flow on $\mathcal{W}^{u}\left(\mathcal{O}_{p}\right)$.
This property will be applied later in the proof of Theorem \ref{main_theorem_4}.
\begin{prop}
\label{prop:certain subsets are submanifolds} To each $\psi\in\mathcal{W}_{loc}^{u}\left(P_{Y},p_{0}\right)$
and $t>1$, there corresponds an $\varepsilon=\varepsilon\left(\psi,t\right)\in\left(0,t-1\right)$
so that the subset 
\[
W_{t,\psi,\varepsilon}=\Phi\left(\left(t-\varepsilon,t+\varepsilon\right)\times\left(\mathcal{W}_{loc}^{u}\left(P_{Y},p_{0}\right)\cap B\left(\psi,\varepsilon\right)\right)\right)
\]
 of \textup{$\mathcal{W}^{u}\left(\mathcal{O}_{p}\right)$} is$ $
a three-dimensional $C^{1}$-submanifold of $C$.\end{prop}
\begin{proof}
It is clear from \eqref{unstable set is forwad extension of unstable manifold-1}
that $W_{t,\psi,\varepsilon}$ defined as above is a subset of $\mathcal{W}^{u}\left(\mathcal{O}_{p}\right)$
for all $\varepsilon\in\left(0,t-1\right).$

Consider the three-dimensional $C^{1}$-submanifold $\left(1,\infty\right)\times\mathcal{W}_{loc}^{u}\left(P_{Y},p_{0}\right)$
of $\mathbb{R}\times C$ and the continuously differentiable map 
\[
\Sigma:\left(1,\infty\right)\times\mathcal{W}_{loc}^{u}\left(P_{Y},p_{0}\right)\ni\left(s,\varphi\right)\mapsto\Phi\left(s,\varphi\right)\in C.
\]
 It suffices to show by Proposition \ref{prop:smoothness of submanifolds}
that for all $\psi\in\mathcal{W}_{loc}^{u}\left(P_{Y},p_{0}\right)$
and $t>1$, the derivative $D\Sigma\left(t,\psi\right)$ is injective
on the tangent space $T_{\left(t,\psi\right)}\left(\left(1,\infty\right)\times\mathcal{W}_{loc}^{u}\left(P_{Y},p_{0}\right)\right)=\mathbb{R}\times T_{\psi}\mathcal{W}_{loc}^{u}\left(P_{Y},p_{0}\right)$.
This space is spanned by the tangent vectors of the following curves
at $0$: 
\[
\left(-1,1\right)\ni s\mapsto\left(t+s,\psi\right)\quad\mbox{and}\quad\left(-1,1\right)\ni s\mapsto\left(t,\gamma_{i}\left(s\right)\right),\, i\in\left\{ 1,2\right\} ,
\]
where 
\[
\gamma_{i}:\left(-1,1\right)\rightarrow\mathcal{W}_{loc}^{u}\left(P_{Y},p_{0}\right)\mbox{ is a }C^{1}\mbox{-curve,}
\]
\[
\gamma_{i}\left(0\right)=\psi\mbox{ and }\mbox{ }D\gamma_{i}\left(0\right)=\eta_{i}\mbox{ for both }i\in\left\{ 1,2\right\} ,
\]
with $\eta_{1}$ and $\eta_{2}$ forming a basis of the two-dimensional
tangent space $T_{\psi}\mathcal{W}_{loc}^{u}\left(P_{Y},p_{0}\right)$.
As $\eta_{1}\in Y$, $\eta_{2}\in Y$ and $\dot{\psi}\notin Y$ by
Proposition \ref{qvarphi dot near r_0}, the vectors $\eta_{1}$,
$\eta_{2}$ and $\dot{\psi}$ are linearly independent. Clearly, 
\[
\frac{\mbox{d}}{\mbox{d}s}\Sigma\left(t+s,\psi\right)|_{s=0}=\frac{\mbox{d}}{\mbox{d}s}\Phi\left(t+s,\psi\right)|_{s=0}=D_{1}\Phi\left(t,\psi\right)1=\dot{x}_{t}^{\psi}=D_{2}\Phi\left(t,\psi\right)\dot{\psi}
\]
and 
\[
\frac{\mbox{d}}{\mbox{d}s}\Sigma\left(t,\gamma_{i}\left(s\right)\right)|_{s=0}=\frac{\mbox{d}}{\mbox{d}s}\Phi\left(t,\gamma_{i}\left(s\right)\right)|_{s=0}=D_{2}\Phi\left(t,\psi\right)\eta_{i},\qquad i\in\left\{ 1,2\right\} .
\]
As $D_{2}\left(t,\psi\right):C\rightarrow C$ is injective (see Section
2) and $\eta_{1}$, $\eta_{2}$ and $\dot{\psi}$ are linearly independent,
we deduce that the range $D\Sigma\left(t,\psi\right)\left(\mathbb{R}\times T_{\psi}\mathcal{W}_{loc}^{u}\left(P_{Y},p_{0}\right)\right)$
is three-dimensional, and thus $D\Sigma\left(t,\psi\right)$ is injective.
\end{proof}
Next we characterize $\mathcal{W}^{u}\left(\mathcal{O}_{p}\right)$
and its tangent vectors in terms of oscillation frequencies.
\begin{prop}
\label{prop: V on unstable set}For all $\varphi\in\mathcal{W}^{u}\left(\mathcal{O}_{p}\right)$
and $\psi\in\mathcal{W}^{u}\left(\mathcal{O}_{p}\right)$ with $\varphi\neq\psi$,
$V\left(\psi-\varphi\right)\leq2$.\end{prop}
\begin{proof}
We distinguish three cases:

(i) both $\varphi\in\mathcal{O}_{p}$ and $\psi\in\mathcal{O}_{p}$;

(ii) $\varphi\in\mathcal{O}_{p}$ and $\psi\in\mathcal{W}^{u}\left(\mathcal{O}_{p}\right)\setminus\mathcal{O}_{p}$
(or vice verse);

(iii) both $\varphi\in\mathcal{W}^{u}\left(\mathcal{O}_{p}\right)\setminus\mathcal{O}_{p}$
and $\psi\in\mathcal{W}^{u}\left(\mathcal{O}_{p}\right)\setminus\mathcal{O}_{p}$
.

Let $\omega>1$ denote the minimal period of $p$. It is easy to deduce
from Proposition \ref{pro: monotone_type} that 
\begin{equation}
V\left(p_{\tau}-p_{\sigma}\right)=2\mbox{ for all }\tau\in\left[0,\omega\right)\mbox{ and }\sigma\in\left[0,\omega\right)\mbox{ with }\tau\neq\sigma.\label{*}
\end{equation}
 Hence the statement holds in case (i).

Case (ii). By definition, there exist $\sigma\in\left[0,\omega\right)$
and $\left(t_{n}\right)_{0}^{\infty}\subset\mathbb{R}$ so that $t_{n}\rightarrow-\infty$
and $x_{t_{n}}^{\psi}\rightarrow p_{\sigma}$ as $n\rightarrow\infty$.
As $x_{t_{n}}^{\varphi}\in\mathcal{O}_{p}$ for all $n\geq0,$ we
may also assume by compactness that $x_{t_{n}}^{\varphi}\rightarrow p_{\tau}$
as $n\rightarrow\infty$ for some $\tau\in\left[0,\omega\right)$.
As the $C$-norm and $C^{1}$-norm are equivalent on the global attractor,
$x_{t_{n}}^{\psi}\rightarrow p_{\sigma}$ and $x_{t_{n}}^{\varphi}\rightarrow p_{\tau}$
as $n\rightarrow\infty$ also in $C^{1}$-norm.

By Lemma \ref{lem:4_properties_of_V} (iii) and property \eqref{*},
$p_{\sigma}-p_{\tau}\in R$ for all $\tau\in\left[0,\omega\right)$
and $\sigma\in\left[0,\omega\right)$ with $\tau\neq\sigma$. Hence
if $\sigma\neq\tau$, then Lemma \ref{lem: continuity_of_V} implies
that 
\[
2=V\left(p_{\sigma}-p_{\tau}\right)=\lim_{n\rightarrow\infty}V\left(x_{t_{n}}^{\psi}-x_{t_{n}}^{\varphi}\right).
\]
By the monotonicity of $V$ we conclude that $ $ $V\left(x_{t}^{\psi}-x_{t}^{\varphi}\right)\leq2$
for all real $t$. If $\sigma=\tau$, then for all $\varepsilon>0$
small, $\sigma+\varepsilon\neq\tau$ and $x_{t_{n}+\varepsilon}^{\psi}\rightarrow p_{\sigma+\varepsilon}$
as $n\rightarrow\infty$ both in $C$-norm and $C^{1}$-norm. Therefore
by Lemma \ref{lem: continuity_of_V} and by our previous reasoning,
\[
V\left(x_{t}^{\psi}-x_{t}^{\varphi}\right)\leq\liminf_{\varepsilon\rightarrow0+}V\left(x_{t+\varepsilon}^{\psi}-x_{t}^{\varphi}\right)\leq2
\]
 for all $t\in\mathbb{R}$. In particular, $V\left(\psi-\varphi\right)\leq2$.

We omit the proof of case (iii), as it is analogous to the one given
for (ii).
\end{proof}
As it is stated in the next proposition, the tangent vectors of $\mathcal{W}^{u}\left(\mathcal{O}_{p}\right)$
have at most two sign changes. This result is a direct consequence
of Proposition \ref{prop: V on unstable set}.
\begin{prop}
\label{prop:V_on_tangent_vectors}Assume $\varphi\in\mathcal{W}^{u}\left(\mathcal{O}_{p}\right)$,
$\gamma:\left(-1,1\right)\rightarrow C$ is a $C^{1}$-curve with
$\gamma\left(0\right)=\varphi$, and $\left(s_{n}\right)_{0}^{\infty}$
is a sequence in $\left(-1,1\right)\backslash\left\{ 0\right\} $
so that $s_{n}\rightarrow0$ as $n\rightarrow\infty$ and $\gamma\left(s_{n}\right)\in\mathcal{W}^{u}\left(\mathcal{O}_{p}\right)$
for all $n\geq0$. Also assume that $\gamma'\left(0\right)\neq\hat{0}$.
Then $V\left(\gamma'\left(0\right)\right)\leq2$. \end{prop}
\begin{proof}
By Proposition \ref{prop: V on unstable set}, 
\[
V\left(\frac{\gamma\left(s_{n}\right)-\gamma\left(0\right)}{s_{n}}\right)\leq2\quad\mbox{for all }n\geq0.
\]
Since $\left(\gamma\left(s_{n}\right)-\gamma\left(0\right)\right)/s_{n}\rightarrow\gamma'\left(0\right)$
in $C$ as $n\rightarrow\infty$, the statement follows from the lower
semi-continuity property of $V$ presented by Lemma \ref{lem: continuity_of_V}. 
\end{proof}
In order to get more information on the unstable set $\mathcal{W}^{u}\left(\mathcal{O}_{p}\right)$,
we project it into the three-dimensional Euclidean space. Introduce
the linear map 
\[
\pi_{3}:C\ni\varphi\mapsto\left(\varphi\left(0\right),\varphi\left(-1\right),\mathcal{I}\left(\varphi\right)\right)\in\mathbb{R}^{3},
\]
where $\mathcal{I}\left(\varphi\right)=\intop_{-1}^{0}\varphi\left(s\right)\mbox{d}s$.
The next statement can be obtained also from Proposition \ref{prop: V on unstable set}.
\begin{prop}
\label{prop:Pi3 on unstable set}$\pi_{3}$ is injective on \textup{$\mathcal{W}^{u}\left(\mathcal{O}_{p}\right)$.}\end{prop}
\begin{proof}
Suppose that there exist $\varphi\in\mathcal{W}^{u}\left(\mathcal{O}_{p}\right)$
and $\psi\in\mathcal{W}^{u}\left(\mathcal{O}_{p}\right)$ so that
$\varphi\neq\psi$ and $\pi_{3}\varphi=\pi_{3}\psi$. Consider the
solutions $x^{\varphi}:\mathbb{R}\rightarrow\mathbb{R}$ and $x^{\psi}:\mathbb{R}\rightarrow\mathbb{R}$
of Eq.\,\eqref{eq:eq_general}. The segments $x_{t}^{\varphi}$ and
$x_{t}^{\psi}$ belong to $\mathcal{W}^{u}\left(\mathcal{O}_{p}\right)$,
and the injectivity of the semiflow $\Phi$ implies that $x_{t}^{\varphi}\neq x_{t}^{\psi}$
for all $t\in\mathbb{R}$. Hence $V\left(x_{t}^{\varphi}-x_{t}^{\psi}\right)\leq2$
for all $t\in\mathbb{R}$ by Proposition \ref{prop: V on unstable set}.
Since $\varphi\left(0\right)-\psi\left(0\right)=\varphi\left(-1\right)-\psi\left(-1\right)=0$,
Lemma \ref{lem:4_properties_of_V} (ii) gives that 
\[
V\left(\varphi-\psi\right)<V\left(x_{-2}^{\varphi}-x_{-2}^{\psi}\right)\leq2,
\]
that is $V\left(\varphi-\psi\right)=0$ and $\varphi\leq\psi$ or
$\psi\leq\varphi$. Using $\mathcal{I}\left(\varphi\right)=\mathcal{I}\left(\psi\right)$
we conclude that $\varphi=\psi$, which contradicts our initial assumption.
\end{proof}
We also need to know how $\pi_{3}$ acts on the tangent vectors of
$\mathcal{W}^{u}\left(\mathcal{O}_{p}\right).$
\begin{prop}
\label{prop:Pi3 on tangent vectors}If $\gamma:\left(-1,1\right)\rightarrow C$
is a $C^{1}$-curve with range in $\mathcal{W}^{u}\left(\mathcal{O}_{p}\right)$
and $\gamma'\left(0\right)\neq\hat{0}$, then $\pi_{3}\gamma'\left(0\right)\neq\left(0,0,0\right)$.\end{prop}
\begin{proof}
Let $\gamma:\left(-1,1\right)\rightarrow C$ be a $C^{1}$-curve with
range in $\mathcal{W}^{u}\left(\mathcal{O}_{p}\right)$ and with $\gamma'\left(0\right)\neq\hat{0}$.
Let $x:\mathbb{R}\rightarrow\mathbb{R}$ be the unique solution of
Eq.~\eqref{eq:eq_general} with $x_{0}=\gamma\left(0\right)\in\mathcal{W}^{u}\left(\mathcal{O}_{p}\right)$,
and set $a:\mathbb{R}\ni t\mapsto f'\left(x\left(t-1\right)\right)\in\mathbb{R}^{+}.$

1. We claim that the problem 
\[
\begin{cases}
\dot{y}\left(t\right)=-y\left(t\right)+a\left(t\right)y\left(t-1\right), & t\in\mathbb{R},\\
y_{0}=\gamma'\left(0\right)
\end{cases}
\]
has a unique solution $y:\mathbb{R}\rightarrow\mathbb{R}$.

Fix a sequence $\left(s_{n}\right)_{n=0}^{\infty}$ in $\left(-1,1\right)\setminus\left\{ 0\right\} $
with $s_{n}\rightarrow0$ as $n\rightarrow\infty$. As $\gamma'\left(0\right)\neq\hat{0}$,
we may assume that $\gamma\left(s_{n}\right)\neq\gamma\left(0\right)$
for all $n\geq0$. Consider the solutions $x^{n}=x^{\gamma\left(s_{n}\right)}:\mathbb{R}\rightarrow\mathbb{R}$.
Then $x_{t}^{n}\in\mathcal{W}^{u}\left(\mathcal{O}_{p}\right)$ for
all $n\geq0$ and $t\in\mathbb{R}$, furthermore $x^{n}\left(t\right)\rightarrow x\left(t\right)$
as $n\rightarrow\infty$ for all $t\in\mathbb{R}$ by the continuity
of the flow $\Phi_{\mathcal{A}}$. Since all their segments belong
to the bounded global attractor, the solutions $x^{n}$ are uniformly
bounded on $\mathbb{R}$, and Eq.\,\eqref{eq:eq_general} gives a
uniform bound for their derivatives. Therefore by applying the Arzel\`{a}\textendash{}Ascoli
theorem successively on the intervals $\left[-j,j\right]$, $j\geq1$,
and by using a diagonalization process, we obtain that $\left(x^{n}\right)_{n=0}^{\infty}$
has a subsequence $\left(x^{n_{k}}\right)_{k=0}^{\infty}$ such that
the convergence $x^{n_{k}}\rightarrow x$ is uniform on all compact
subsets of $\mathbb{R}$. Set 
\[
y^{k}\left(t\right)=\frac{x^{n_{k}}\left(t\right)-x\left(t\right)}{s_{n_{k}}}\quad\mbox{for all }k\geq0\mbox{ and }t\in\mathbb{R}.
\]
Then for all $ $$k\geq0$ and $t\in\mathbb{R}$, $y_{t}^{k}\neq\hat{0}$
by the injectivity of the flow $\Phi_{\mathcal{A}}$, and $V\left(y_{t}^{k}\right)\leq2$
by Proposition \ref{prop: V on unstable set}. In addition, $y^{k}$,
$k\geq0$, satisfies the equation $\dot{y}^{k}\left(t\right)=-y^{k}\left(t\right)+a_{k}\left(t\right)y^{k}\left(t-1\right)$
on $\mathbb{R}$, where 
\[
a_{k}:\mathbb{R}\ni t\mapsto\int_{0}^{1}f'\left(sx^{n_{k}}\left(t-1\right)+\left(1-s\right)x\left(t-1\right)\right)\mbox{d}s\in\mathbb{R}^{+},\quad k\geq0.
\]
It is clear that there are constants $\alpha_{1}\geq\alpha_{0}>0$
independent of $k$ and $t$ such that $\alpha_{0}\leq a_{k}\left(t\right)\leq\alpha_{1}$
for all $ $$k\geq0$ and $t\in\mathbb{R}$. Also note that $a_{k}\rightarrow a$
as $k\rightarrow\infty$ uniformly on compact subsets of $\mathbb{R}$.
Therefore by Lemma \ref{lem:technical result}, there exist a continuously
differentiable function $y:\mathbb{R}\rightarrow\mathbb{R}$ and a
subsequence $\left(y^{k_{l}}\right)_{l=0}^{\infty}$ of $\left(y^{k}\right)_{k=0}^{\infty}$
such that $y^{k_{l}}\rightarrow y$ and $\dot{y}^{k_{l}}\rightarrow\dot{y}$
as $k\rightarrow\infty$ uniformly on compact subsets of $\mathbb{R}$,
moreover 
\begin{equation}
\dot{y}\left(t\right)=-y\left(t\right)+a\left(t\right)y\left(t-1\right)\label{eq:*}
\end{equation}
 for all real $t$. It is clear from the construction that 
\[
y_{0}=\lim_{l\rightarrow\infty}\frac{x_{0}^{n_{k_{l}}}-x_{0}}{s_{n_{k_{l}}}}=\lim_{l\rightarrow\infty}\frac{\gamma\left(s_{n_{k_{l}}}\right)-\gamma\left(0\right)}{s_{n_{k_{l}}}}=\gamma'\left(0\right).
\]
The uniqueness of $y$ is guaranteed by Proposition \ref{prop:uniqueness of solutions}.

2. Next we claim that $\left(-1,1\right)\ni s\mapsto\Phi_{\mathcal{A}}\left(-2,\gamma\left(s\right)\right)$
is differentiable at $s=0$, and 
\[
\frac{\mbox{d}}{\mbox{d}s}\Phi_{\mathcal{A}}\left(-2,\gamma\left(s\right)\right)|_{s=0}=y_{-2}.
\]
If this is not true, then there exists a sequence $\left(s_{n}\right)_{n=0}^{\infty}$
in $\left(-1,1\right)\setminus\left\{ 0\right\} $ with $s_{n}\rightarrow0$
as $n\rightarrow\infty$ such that for all $n\geq0$, 
\[
\frac{\Phi_{\mathcal{A}}\left(-2,\gamma\left(s_{n}\right)\right)-\Phi_{\mathcal{A}}\left(-2,\gamma\left(0\right)\right)}{s_{n}}
\]
remains outside a fixed neighborhood of $y_{-2}$ in $C$. So to verify
the claim, it suffices to show that any sequence $\left(s_{n}\right)_{n=0}^{\infty}$
in $\left(-1,1\right)\setminus\left\{ 0\right\} $ with $s_{n}\rightarrow0$
as $n\rightarrow\infty$ admits a subsequence $\left(s_{n_{l}}\right)_{l=0}^{\infty}$
for which $ $ 
\[
\frac{\Phi_{\mathcal{A}}\left(-2,\gamma\left(s_{n_{l}}\right)\right)-\Phi_{\mathcal{A}}\left(-2,\gamma\left(0\right)\right)}{s_{n_{l}}}\rightarrow y_{-2}\quad\mbox{as }l\rightarrow\infty.
\]
Indeed, by repeating the reasoning in the first part of the proof
word by word, one can show that the sequence $\left(x^{n}\right)_{n=0}^{\infty}$
formed by the solutions $x^{n}=x^{\gamma\left(s_{n}\right)}:\mathbb{R}\rightarrow\mathbb{R}$,
$n\geq0$, has a subsequence $\left(x^{n_{l}}\right)_{l=0}^{\infty}$
such that $\left(x^{n_{l}}-x\right)/s_{n_{l}}\rightarrow y$ as $l\rightarrow\infty$
uniformly on compact subsets of $\mathbb{R}$. In particular, 
\[
y_{-2}=\lim_{l\rightarrow\infty}\frac{x_{-2}^{n_{_{l}}}-x_{-2}}{s_{n_{l}}}=\lim_{l\rightarrow\infty}\frac{\Phi_{\mathcal{A}}\left(-2,\gamma\left(s_{n_{l}}\right)\right)-\Phi_{\mathcal{A}}\left(-2,\gamma\left(0\right)\right)}{s_{n_{_{l}}}}.
\]

3. So $y_{-2}$ is a tangent vector of $\mathcal{W}^{u}\left(\mathcal{O}_{p}\right)$
at $x_{-2}$, and thus $V\left(y_{-2}\right)\leq2$ by Proposition
\ref{prop:V_on_tangent_vectors}. 

4. To prove the assertion indirectly, suppose that 
\[
\gamma'\left(0\right)\left(0\right)=\gamma'\left(0\right)\left(-1\right)=\mathcal{I}\left(\gamma'\left(0\right)\right)=0.
\]
 Then as $ $$y\left(0\right)=\gamma'\left(0\right)\left(0\right)=0$
and $y\left(-1\right)=\gamma'\left(0\right)\left(-1\right)=0$, $V\left(\gamma'\left(0\right)\right)<V\left(y_{-2}\right)\leq2$
by Lemma \ref{lem:4_properties_of_V} (ii). So $V\left(\gamma'\left(0\right)\right)=0$,
that is $\gamma'\left(0\right)\geq\hat{0}$ or $\gamma'\left(0\right)\leq\hat{0}$.
As we have also assumed that $\mathcal{I}\left(\gamma'\left(0\right)\right)=0$,
necessarily $\gamma'\left(0\right)=\hat{0}$ follows, a contradiction.
The proof is complete.
\end{proof}
Now we can verify Theorem \ref{main_theorem_3}.

\begin{proof}[Proof of Theorem \ref{main_theorem_3}]~\\
\emph{1.The proof of the assertion that $\mathcal{W}^{u}\left(\mathcal{O}_{p}\right)$
is a} \emph{three-dimensional $C^{1}$-submanifold of $C$}. All $\varphi\in\mathcal{W}^{u}\left(\mathcal{O}_{p}\right)$
can be written in form $\varphi=\Phi\left(t,\psi\right)$, where $t>1$
and $\psi\in\mathcal{W}_{loc}^{u}\left(P_{Y},p_{0}\right)$. This
property follows from relation \eqref{unstable set is forwad extension of unstable manifold-1}
and the fact that to each $\psi\in\mathcal{W}_{loc}^{u}\left(P_{Y},p_{0}\right)$,
there corresponds a trajectory $\left(\psi^{n}\right)_{-\infty}^{0}$
of $P_{Y}$ in $\mathcal{W}_{loc}^{u}\left(P_{Y},p_{0}\right)$ with
$\psi^{0}=\psi$ and $\psi^{n}\rightarrow p_{0}$ as $n\rightarrow-\infty$.
Hence Proposition \ref{prop:certain subsets are submanifolds} guarantees
the existence of $\varepsilon>0$ so that the subset 
\[
W_{t,\psi,\varepsilon}=\Phi\left(\left(t-\varepsilon,t+\varepsilon\right)\times\left(\mathcal{W}_{loc}^{u}\left(P_{Y},p_{0}\right)\cap B\left(\psi,\varepsilon\right)\right)\right)
\]
of $\mathcal{W}^{u}\left(\mathcal{O}_{p}\right)$ containing $\varphi$
is a three-dimensional $C^{1}$-submanifold of $C$. 

To show that $\mathcal{W}^{u}\left(\mathcal{O}_{p}\right)$ is a\emph{
}three-dimensional $C^{1}$-submanifold of $C$, it suffices to exclude
for all $t>1$ and $\psi\in\mathcal{W}_{loc}^{u}\left(P_{Y},p_{0}\right)$
the existence of a sequence $\left(\varphi^{n}\right)_{n=0}^{\infty}$
in $ $$\mathcal{W}^{u}\left(\mathcal{O}_{p}\right)$ so that $\varphi^{n}\notin W_{t,\psi,\varepsilon}$
for $n\geq0$ and $\varphi^{n}\rightarrow\varphi=\Phi\left(t,\psi\right)$
as $n\rightarrow\infty$. According to Proposition \ref{prop:Pi3 on tangent vectors},
$D\pi_{3}\left(\varphi\right)=\pi_{3}$ is injective on the three-dimensional
tangent space $T_{\varphi}W_{t,\psi,\varepsilon}$, i.e. it defines
an isomorphism from $T_{\varphi}W_{t,\psi,\varepsilon}$ onto $\mathbb{R}^{3}$.
Thus the inverse mapping theorem yields a constant $\delta>0$ such
that the restriction of $\pi_{3}$ to $W_{t,\psi,\varepsilon}\cap B\left(\varphi,\delta\right)$
is a diffeomorphism from $W_{t,\psi,\varepsilon}\cap B\left(\varphi,\delta\right)$
onto an open set $U$ in $\mathbb{R}^{3}$. If a sequence $\left(\varphi^{n}\right)_{n=0}^{\infty}$
in $ $$\mathcal{W}^{u}\left(\mathcal{O}_{p}\right)$ converges to
$\varphi$ as $n\rightarrow\infty$, then $\pi_{3}\varphi^{n}\rightarrow\pi_{3}\varphi$
as $n\rightarrow\infty$, and $\pi_{3}\varphi^{n}\in U$ for all sufficiently
large $n$. The injectivity of $\pi_{3}$ on $\mathcal{W}^{u}\left(\mathcal{O}_{p}\right)$
verified in Proposition \ref{prop:Pi3 on unstable set} then implies
that $\varphi^{n}\in W_{t,\psi,\varepsilon}$.

\emph{2. Graph representation for }$\mathcal{W}^{u}\left(\mathcal{O}_{p}\right)$\emph{.}
Choose $\varphi_{j}\in C$ such that $\pi_{3}\varphi_{j}=e_{j}$,
$j\in\left\{ 1,2,3\right\} $, where $e_{1}=\left(1,0,0\right)$,
$e_{2}=\left(0,1,0\right)$ and $e_{3}=\left(0,0,1\right)$. This
is possible as  $\pi_{3}:C\ni\varphi\mapsto\left(\varphi\left(0\right),\varphi\left(-1\right),\mathcal{I}\left(\varphi\right)\right)\in\mathbb{R}^{3}$
is injective on the $3$-dimensional tangent spaces of $\mathcal{W}^{u}\left(\mathcal{O}_{p}\right)$,
and hence it is surjective. Clearly $\varphi_{1}$, $\varphi_{2}$
and $\varphi_{3}$ are linearly independent. 

Let $J_{3}:\mathbb{R}^{3}\rightarrow C$ be the injective linear map
for which $J_{3}e_{j}=\varphi_{j}$, $j\in\left\{ 1,2,3\right\} $,
and let $P_{3}=J_{3}\circ\pi_{3}$. Then $P_{3}:C\rightarrow C$ is
continuous, linear and $P_{3}\varphi_{j}=\varphi_{j}$ for all $j\in\left\{ 1,2,3\right\} $.
In consequence, $P_{3}\circ P_{3}=P_{3}$, which means that $P_{3}$
is a projection. The space 
\[
G_{3}=P_{3}C=\left\{ c_{1}\varphi_{1}+c_{2}\varphi_{2}+c_{3}\varphi_{3}:\, c_{1},c_{2},c_{3}\in\mathbb{R}\right\} 
\]
 is $3$-dimensional, and with $E=P_{3}^{-1}\left(0\right)$, we have
$C=G_{3}\oplus E$. As the restriction of $P_{3}$ to $\mathcal{W}^{u}\left(\mathcal{O}_{p}\right)$
is injective, the inverse $P_{3}^{-1}$ of the map $\mathcal{W}^{u}\left(\mathcal{O}_{p}\right)\ni\varphi\mapsto P_{3}\varphi\in G_{3}$
exists. At last, introduce the map 
\[
w:P_{3}\mathcal{W}^{u}\left(\mathcal{O}_{p}\right)\ni\chi\mapsto\left(\mbox{id}-P_{3}\right)\circ P_{3}^{-1}\left(\chi\right)\in E.
\]
Then 
\[
\mathcal{W}^{u}\left(\mathcal{O}_{p}\right)=\left\{ \chi+w\left(\chi\right):\,\chi\in P_{3}\mathcal{W}^{u}\left(\mathcal{O}_{p}\right)\right\} .
\]

It remains to show that $U_{3}=P_{3}\mathcal{W}^{u}\left(\mathcal{O}_{p}\right)$
is open in $G_{3}$ and $w$ is $C^{1}$-smooth. Let $\chi\in P_{3}\mathcal{W}^{u}\left(\mathcal{O}_{p}\right)$
be arbitrary. Then $\chi=P_{3}\varphi$ with some $\varphi\in\mathcal{W}^{u}\left(\mathcal{O}_{p}\right)$.
As the restriction of $\pi_{3}$ to $T_{\varphi}\mathcal{W}\left(\mathcal{O}_{p}\right)$
is injective, $DP_{3}\left(\varphi\right)=P_{3}$ defines an isomorphism
from $T_{\varphi}\mathcal{W}\left(\mathcal{O}_{p}\right)$ to $G_{3}$.
Consequently the inverse mapping theorem implies that an $\varepsilon>0$
can be given such that $P_{3}$ maps $\mathcal{W}\left(\mathcal{O}_{p}\right)\cap B\left(\varphi,\varepsilon\right)$
one-to-one onto an open neighborhood $U\subset U_{3}$ of $\chi$
in $G_{3}$, $P_{3}$ is invertible on $\mathcal{W}\left(\mathcal{O}_{p}\right)\cap B\left(\varphi,\varepsilon\right)$,
and the inverse $\tilde{P}_{3}^{-1}$ of the map 
\[
\mathcal{W}\left(\mathcal{O}_{p}\right)\cap B\left(\varphi,\varepsilon\right)\ni\varphi\mapsto P_{3}\varphi\in U
\]
 is $C^{1}$-smooth. As 
\[
w\left(\chi\right)=\left(\mbox{id}-P_{3}\right)\circ P_{3}^{-1}\left(\chi\right)=\left(\mbox{id}-P_{3}\right)\circ\tilde{P}_{3}^{-1}\left(\chi\right)
\]
 for all $\chi\in U$, the restriction of $w$ to $U$ is $C^{1}$-smooth.

\emph{3. The characterization of $C_{j}^{p}$, $j\in\left\{ -2,0,2\right\} $.
}Since the basin of attraction of a stable equilibrium is open in
$C$, the connecting set $C_{j}^{p}$, $j\in\left\{ -2,0,2\right\} $,
is an open subset of $\mathcal{W}^{u}\left(\mathcal{O}_{p}\right)$.
It follows immediately that $C_{j}^{p}$, $j\in\left\{ -2,0,2\right\} $,
is a three-dimensional $C^{1}$-submanifold of $C$ and 
\[
C_{j}^{p}=\left\{ \chi+w\left(\chi\right):\,\chi\in P_{3}C_{j}^{p}\right\} 
\]
for all $j\in\left\{ -2,0,2\right\} $.\end{proof}

As $\mathcal{W}^{u}\left(\mathcal{O}_{p}\right)$ is a $C^{1}$-submanifold
of $C$, it makes sense to investigate the differentiability of the
map 
\[
\Phi_{\mathcal{W}^{u}\left(\mathcal{O}_{p}\right)}:\mathbb{R}\times\mathcal{W}^{u}\left(\mathcal{O}_{p}\right)\ni\left(t,\varphi\right)\mapsto\Phi_{\mathcal{A}}\left(t,\varphi\right)\in\mathcal{W}^{u}\left(\mathcal{O}_{p}\right).
\]

Suppose that $\eta_{1}$ $\eta_{2}$ and $\eta_{3}$ form a basis
of the three-dimensional tangent space $T_{\varphi}\mathcal{W}^{u}\left(\mathcal{O}_{p}\right)$
of $\mathcal{W}^{u}\left(\mathcal{O}_{p}\right)$ at some $\varphi\in\mathcal{W}^{u}\left(\mathcal{O}_{p}\right)$.
Then for all $t\in\mathbb{R}$, the tangent space $T_{\left(t,\varphi\right)}\left(\mathbb{R}\times\mathcal{W}^{u}\left(\mathcal{O}_{p}\right)\right)$
of $\mathbb{R}\times\mathcal{W}^{u}\left(\mathcal{O}_{p}\right)$
at $(t,\varphi)$ is spanned by the tangent vectors of the following
curves at $0$: 
\[
\left(-1,1\right)\ni s\mapsto\left(t+s,\varphi\right)\quad\mbox{and}\quad\left(-1,1\right)\ni s\mapsto\left(t,\gamma_{i}\left(s\right)\right),\, i\in\left\{ 1,2,3\right\} ,
\]
where $\gamma_{i}:\left(-1,1\right)\rightarrow\mathcal{W}^{u}\left(\mathcal{O}_{p}\right)$
is a $C^{1}$-curve with $\gamma_{i}\left(0\right)=\varphi$ and $D\gamma_{i}\left(0\right)=\eta_{i}$
for all $i\in\left\{ 1,2,3\right\} $.

We are going to apply the following assertion in the proof of Theorem
\ref{main_theorem_4}.(ii).
\begin{prop}
\label{the flow is continuously differentiable} The flow $\Phi_{\mathcal{W}^{u}\left(\mathcal{O}_{p}\right)}$
is $C^{1}$-smooth. For all $t\in\mathbb{R}$ and $\varphi\in\mathcal{W}^{u}\left(\mathcal{O}_{p}\right)$,
\begin{equation}
\frac{\mbox{d}}{\mbox{d}s}\Phi_{\mathcal{W}^{u}\left(\mathcal{O}_{p}\right)}\left(t+s,\varphi\right)|_{s=0}=\dot{x}_{t}^{\varphi}.\label{derivative with respect to the first variable}
\end{equation}
For all $\varphi\in\mathcal{W}^{u}\left(\mathcal{O}_{p}\right)$ and
$\eta\in T_{\varphi}\mathcal{W}^{u}\left(\mathcal{O}_{p}\right)$,
the variational equation 
\begin{align*}
\dot{v}(t) & =-v(t)+f'\left(x^{\varphi}\left(t-1\right)\right)v\left(t-1\right)\tag{2.2}
\end{align*}
has a unique solution $v^{\eta}:\mathbb{R}\rightarrow\mathbb{R}$
with $v_{0}^{\eta}=\eta$. If $t\in\mathbb{R}$ and $\gamma:\left(-1,1\right)\rightarrow\mathcal{W}^{u}\left(\mathcal{O}_{p}\right)$
is a $C^{1}$-curve with $\gamma\left(0\right)=\varphi$ and $\gamma'\left(0\right)=\eta,$
then 
\begin{equation}
\frac{\mbox{d}}{\mbox{d}s}\Phi_{\mathcal{W}^{u}\left(\mathcal{O}_{p}\right)}\left(t,\gamma\left(s\right)\right)|_{s=0}=v_{t}^{\eta}.\label{derivative with respect to the second variable}
\end{equation}
 \end{prop}
\begin{proof}
1. To prove the smoothness of $\Phi_{\mathcal{W}^{u}\left(\mathcal{O}_{p}\right)}$,
it is sufficient to show that for all $t\in\mathbb{R}$, the map 
\begin{equation}
\left(t,\infty\right)\times\mathcal{W}^{u}\left(\mathcal{O}_{p}\right)\ni\left(s,\varphi\right)\mapsto\Phi_{\mathcal{A}}\left(s,\varphi\right)\in\mathcal{W}^{u}\left(\mathcal{O}_{p}\right)\label{a  map}
\end{equation}
 is continuously differentiable. 

Let $t\in\mathbb{R}$ be given, and introduce the map 
\[
A_{t}:\mathcal{W}^{u}\left(\mathcal{O}_{p}\right)\ni\varphi\mapsto\Phi_{\mathcal{A}}\left(t,\varphi\right)\in\mathcal{W}^{u}\left(\mathcal{O}_{p}\right).
\]
For $t\geq0$, $A_{t}$ is clearly $C^{1}$-smooth as $\Phi\left(t,\cdot\right)$
is $C^{1}$-smooth and maps $\mathcal{W}^{u}\left(\mathcal{O}_{p}\right)$
into $\mathcal{W}^{u}\left(\mathcal{O}_{p}\right)$. For $t<0$, the
smoothness of $A_{t}$ follows from the smoothness of the map $\Phi\left(-t,\cdot\right)$,
the injectivity of its derivative, the inclusion $\Phi\left(-t,\mathcal{W}^{u}\left(\mathcal{O}_{p}\right)\right)\subset\mathcal{W}^{u}\left(\mathcal{O}_{p}\right)$
and the inverse mapping theorem. 

For all $\left(s,\varphi\right)\in\left(t,\infty\right)\times\mathcal{W}^{u}\left(\mathcal{O}_{p}\right)$,
\[
\Phi_{\mathcal{A}}\left(s,\varphi\right)=\Phi\left(s+1-t,\Phi_{\mathcal{A}}\left(t-1,\varphi\right)\right)=\Phi\left(s+1-t,A_{t-1}\left(\varphi\right)\right).
\]
So the $C^{1}$-smoothness of the maps $\Phi|_{\left(1,\infty\right)\times C}$
and 
\[
\left(t,\infty\right)\times\mathcal{W}^{u}\left(\mathcal{O}_{p}\right)\ni\left(s,\varphi\right)\mapsto\left(s+1-t,A_{t-1}\left(\varphi\right)\right)\in\left(1,\infty\right)\times C
\]
 guarantee that \eqref{a  map} is also continuously differentiable. 

2. Relation \eqref{derivative with respect to the first variable}
is already known for $t>1$. It can be easily obtained for $t\leq1$
from the definition of the Fréchet derivative.

3. We already now that initial value problems corresponding to the
variational equation $(2.2)$ exist and are unique in forward time,
moreover relation \eqref{derivative with respect to the second variable}
holds for $t\geq0$. 

Fix $t<0$. Note that if $\gamma:\left(-1,1\right)\rightarrow\mathcal{W}^{u}\left(\mathcal{O}_{p}\right)$
is a $C^{1}$-curve with $\gamma\left(0\right)=\varphi$ and $\gamma'\left(0\right)=\eta,$
then 
\[
\frac{\mbox{d}}{\mbox{d}s}\Phi_{\mathcal{W}^{u}\left(\mathcal{O}_{p}\right)}\left(t,\gamma\left(s\right)\right)|_{s=0}=DA_{t}\left(\varphi\right)\eta.
\]

By part $1$, the map $A_{t}$ is a $C^{1}$-diffeomorphism with the
inverse $A_{t}^{-1}=A_{-t}$. Hence for all $\eta\in T_{\varphi}\mathcal{W}^{u}\left(\mathcal{O}_{p}\right)$,
$\chi=DA_{t}\left(\varphi\right)\eta$ exists and belongs to $T_{\Phi_{\mathcal{A}}\left(t,\varphi\right)}\mathcal{W}^{u}\left(\mathcal{O}_{p}\right)$.
Then 
\[
\eta=DA_{t}^{-1}\left(\Phi_{\mathcal{A}}\left(t,\varphi\right)\right)\chi=DA_{-t}\left(\Phi_{\mathcal{A}}\left(t,\varphi\right)\right)\chi=D_{2}\Phi\left(-t,\Phi_{\mathcal{A}}\left(t,\varphi\right)\right)\chi=u_{-t}^{\chi},
\]
where $u^{\chi}:\left[-1,\infty\right)\rightarrow\mathbb{R}$ is the
solution of 
\begin{align*}
\dot{u}\left(s\right) & =-u\left(s\right)+f'\left(x^{\Phi_{\mathcal{A}}\left(t,\varphi\right)}\left(s-1\right)\right)u\left(s-1\right)\\
 & =-u\left(s\right)+f'\left(x^{\varphi}\left(t+s-1\right)\right)u\left(s-1\right)
\end{align*}
with $u_{0}^{\chi}=\chi$. With transformation $v\left(s\right)=u\left(s-t\right)$
we obtain that the problem 
\begin{equation}
\begin{cases}
\dot{v}\left(s\right)=-v\left(s\right)+f'\left(x^{\varphi}\left(s-1\right)\right)v\left(s-1\right)\\
v_{0}=\eta
\end{cases}\label{shifted var eq}
\end{equation}
has a solution $v^{\eta}$ on $\left[t-1,\infty\right)$ satisfying
$v_{t}^{\eta}=\chi=DA_{t}\left(\varphi\right)\eta$. As this reasoning
holds for any $t<0$, \eqref{shifted var eq} admits a solution $v^{\eta}:\mathbb{R}\rightarrow\mathbb{R}$
with $v_{t}^{\eta}=DA_{t}\left(\varphi\right)\eta$ for any $t<0$.
By Proposition \ref{prop:uniqueness of solutions}, $v^{\eta}$ is
unique. Relation \eqref{derivative with respect to the second variable}
follows.
\end{proof}
The uniqueness of $v^{\eta}$ and formula \eqref{derivative with respect to the second variable}
guarantee the subsequent corollary. 
\begin{cor}
\label{injective derivative}For each fixed $t\in\mathbb{R}$, the
derivative of the map 
\[
\mathcal{W}^{u}\left(\mathcal{O}_{p}\right)\ni\varphi\mapsto\Phi_{\mathcal{W}^{u}\left(\mathcal{O}_{p}\right)}\left(t,\varphi\right)\in\mathcal{W}^{u}\left(\mathcal{O}_{p}\right)
\]
 at any $\varphi\in\mathcal{W}^{u}\left(\mathcal{O}_{p}\right)$ is
injective on $T_{\varphi}\mathcal{W}^{u}\left(\mathcal{O}_{p}\right)$
.
\end{cor}

\section{The Proof of Theorem \ref{main_theorem_4}}

Fix index $k\in\left\{ -1,1\right\} $ in the rest of the paper and
consider the sets $C_{q}^{p}$, $C_{k}^{p}$ and $S_{k}=C_{k}^{p}\cup\mathcal{O}_{p}\cup C_{q}^{p}$. 

~

\noindent \begin{center}
\textit{5.1 Preliminary results on $\overline{S_{k}}$}
\par\end{center}

\smallskip{}

In this subsection we define a projection $\pi_{2}$ from $C$ into
$\mathbb{R}^{2}$ and show that $\pi_{2}$ is injective on the closure
$\overline{S_{k}}$ of $S_{k}$ in $C$, see Proposition \ref{prop:Pi2 on S_k}).
The proof of this assertion is based on the special properties of
the discrete Lyapunov functional $V$. The injectivity of $\pi_{2}|_{\overline{S_{k}}}$
enables us to give a graph representation for $\overline{S_{k}}$
(without smoothness properties): there is a linear isomorphism $J_{2}:\mathbb{R}^{2}\rightarrow C$
such that $P_{2}=J_{2}\circ\pi_{2}:C\rightarrow C$ is a projection
onto a two-dimensional subspace $G_{2}$ of $C$, and a map $w_{k}:P_{2}\overline{S_{k}}\rightarrow P_{2}^{-1}\left(0\right)$
can be defined such that 
\[
\overline{S_{k}}=\left\{ \chi+w_{k}\left(\chi\right):\,\chi\in P_{2}\overline{S_{k}}\right\} ,
\]
 see Proposition \ref{prop:graph representation for the closure of S_k}.
The differentiability of $w_{k}$ and the properties of its domain
$P_{2}\overline{S_{k}}\subset G_{2}$ are studied only in Subsections
5.3 and 5.5. We also show at the end of this subsection that $\pi_{2}|_{\overline{S_{k}}}$
is a homeomorphism onto its image (see Proposition \ref{prop:inverse of Pi_2 is Lip-cont}),
moreover $\pi_{2}$ is injective on the tangent spaces of $\overline{S_{k}}$
(see Proposition \ref{prop:Pi_2 on tangent vectors}). 

Clearly, $S_{k}$ is invariant under $\Phi_{\mathcal{A}}$. Then it
easily follows that $\overline{S_{k}}$ is invariant too. Indeed,
let $\varphi\in\overline{S_{k}}\setminus S_{k}$ be arbitrary and
choose a sequence $\left(\varphi_{n}\right)_{n=0}^{\infty}$ in $S_{k}$
converging to $\varphi$ as $n\rightarrow\infty$. As the global attractor
$\mathcal{A}$ is closed, $\varphi\in\mathcal{A}$. By the continuity
of the flow $\Phi_{\mathcal{A}}$ on $\mathbb{R}\times\mathcal{A}$,
$S_{k}\ni x_{t}^{\varphi_{n}}\rightarrow x_{t}^{\varphi}$ as $n\rightarrow\infty$
for all $t\in\mathbb{R}$, which means that $\overline{S_{k}}$ is
invariant under $\Phi_{\mathcal{A}}$.

By Theorem B, 
\[
S_{k}=\left\{ \varphi\in\mathcal{W}^{u}\left(\mathcal{O}_{p}\right):\, x^{\varphi}\mbox{ oscillates around }\xi_{k}\right\} .\tag{1.2}
\]

Note that if $x^{\varphi}$ is nonoscillatory around $\xi_{k}$ for
some $\varphi\in C$ (i.e. there exists $T\geq0$ so that $x_{T}^{\varphi}\gg\hat{\xi}_{k}$
or $x_{T}^{\varphi}\ll\hat{\xi}_{k}$ ), then $\varphi$ has an open
neighborhood $U_{\varphi}$ in $C$ such that for all $\psi\in U_{\varphi}$,
$x^{\psi}$ is nonoscillatory around $\xi_{k}$. Hence it comes immediately
from \eqref{S_k} that for all $\varphi\in\overline{S_{k}}$, $x^{\varphi}$
oscillates around $\xi_{k}$.

The next result states that the stable set of the unstable equilibrium
$\hat{\xi}_{k}$ contains only nonordered elements with respect to
the pointwise ordering. The proof follows the first part of the proof
of Proposition 3.1 in \cite{Krisztin-Walther-Wu}.
\begin{prop}
\label{prop:solutions converging to xi_k}There exist no $\varphi\in C$
and $\psi\in C$ with $\varphi\ll\psi$ such that $x_{t}^{\varphi}$
and $x_{t}^{\psi}$ both converge to $\hat{\xi}_{k}$ as $t\rightarrow\infty$.\end{prop}
\begin{proof}
Suppose  that $\varphi\in C$, $\psi\in C$, $\varphi\ll\psi$ and
both $x_{t}^{\varphi}$, $x_{t}^{\psi}$ converge to $\hat{\xi}_{k}$
as $t\rightarrow\infty$. Then $y:=x^{\psi}-x^{\varphi}$ is positive
on $\left[-1,\infty\right)$ by Proposition \ref{pro:monotone_dynamical_system},
it satisfies 
\[
\dot{y}(t)=-y(t)+b\left(t\right)y(t-1)
\]
for all $t>0,$ where 
\[
b:\left[0,\infty\right)\ni t\mapsto\int_{0}^{1}f'\left(sx^{\psi}\left(t-1\right)+\left(1-s\right)x^{\varphi}\left(t-1\right)\right)\mbox{d}s\in\left(0,\infty\right),
\]
furthermore $b\left(t\right)\rightarrow f'\left(\xi_{k}\right)$ as
$t\rightarrow\infty$. Since $f'\left(\xi_{k}\right)>1$ by hypothesis
(H1), the number $\varepsilon=\left(f'\left(\xi_{k}\right)-1\right)e^{-1}/2$
is positive. So there exists $T\geq0$ such that $b\left(t\right)\geq f'\left(\xi_{k}\right)-\varepsilon$
for all $t\geq T$. Observe that the positivity of $y$ and $b$ implies
that 
\[
\frac{\mbox{d}}{\mbox{d}t}\left(e^{t}y\left(t\right)\right)=e^{t}b\left(t\right)y\left(t-1\right)>0\quad\mbox{for all }t>0.
\]
For this reason, $e^{t-1}y\left(t-1\right)<e^{t}y\left(t\right)$
for $t\geq1$, and 
\begin{align*}
\dot{y}\left(t\right) & \geq-y\left(t\right)+\left(f'\left(\xi_{k}\right)-\varepsilon\right)y\left(t-1\right)\\
 & \geq-\left(1+\varepsilon e\right)y\left(t\right)+f'\left(\xi_{k}\right)y\left(t-1\right)
\end{align*}
for all $t\geq T+1$. The choice of $\varepsilon$ ensures that 
\[
1+\varepsilon e=\frac{1}{2}+\frac{1}{2}f'\left(\xi_{k}\right)<f'\left(\xi_{k}\right).
\]
Hence the equation
\[
\lambda+\left(1+\varepsilon e\right)=f'\left(\xi_{k}\right)e^{-\lambda}
\]
has a positive real solution $\lambda$. Choose $\delta>0$ so that
$y(t)>\delta e^{\lambda t}$ on $[T,T+1]$. Function $z(t)=\delta e^{\lambda t}$
is a solution of the equation 
\[
\dot{z}\left(t\right)=-\left(1+\varepsilon e\right)z\left(t\right)+f'\left(\xi_{k}\right)z\left(t-1\right)
\]
on $\mathbb{R}$. Set $u=y-z.$ Then $u_{T+1}\gg\hat{0}$ and 
\[
\dot{u}(t)\geq-\left(1+\varepsilon e\right)u(t)+f'\left(\xi_{k}\right)u(t-1)\textrm{ for all }t\geq T+1.
\]
If there existed $t^{*}>T+1$ so that $u\left(t^{*}\right)=0$ and
$u$ is positive on $\left[T,t^{*}\right)$, then $\dot{u}\left(t^{*}\right)$
would be nonpositive. On the other hand, the inequality for $u$ combined
with $u\left(t^{*}\right)=0$ and $u\left(t^{*}-1\right)>0$ would
yield that $\dot{u}\left(t^{*}\right)>0$. So $u(t)=y(t)-z(t)=y(t)-\delta e^{\lambda t}>0$
for all $t\geq T$, which contradicts the boundedness of $y$. 
\end{proof}
The next proposition is the analogue of Proposition 3.1 in \cite{Krisztin-Walther-Wu}.
\begin{prop}
\label{prop:nonordeing of S_k}(Nonordering of $\overline{S_{k}}$)
For all $\varphi,\psi\in C$ with $\varphi<\psi$$ $, either $\varphi\in C\backslash\overline{S_{k}}$
or $ $$\psi\in C\backslash\overline{S_{k}}$.\end{prop}
\begin{proof}
If there are $\tilde{\varphi}\in\overline{S_{k}}$ and $\tilde{\psi}\in\overline{S_{k}}$
satisfying $\tilde{\varphi}<\tilde{\psi}$, then by Proposition \ref{pro:monotone_dynamical_system}
and the invariance of $\overline{S_{k}}$, $\varphi=x_{2}^{\tilde{\varphi}}\in\overline{S_{k}}$,
$\psi=x_{2}^{\tilde{\psi}}\in\overline{S_{k}}$ and $\varphi\ll\psi$.
Theorem 4.1 in Chapter 5 of \cite{Smith} proves that  there is an
open and dense set of initial functions in $C_{-2,2}$ so that the
corresponding solutions converge to equilibria. Hence there exist
$\varphi^{*}\in C$ and $\psi^{*}\in C$ with $\varphi\ll\varphi^{*}\ll\psi^{*}\ll\psi$
such that both $x_{t}^{\varphi^{*}}$ and $x_{t}^{\psi^{*}}$ tend
to equilibria as $t\rightarrow\infty$. 

If $x_{t}^{\psi^{*}}\rightarrow\hat{\xi}$ as $t\rightarrow\infty$,
where $\hat{\xi}$ is any equilibrium with $\xi>\xi_{k}$, then there
exists $T>0$ such that $\hat{\xi}_{k}\ll x_{T}^{\psi^{*}}$. Then
$\hat{\xi}_{k}\ll x_{T}^{\psi^{*}}\ll x_{T}^{\psi}$ by Proposition
\ref{pro:monotone_dynamical_system}, which contradicts the fact that
the elements of $\overline{S_{k}}$ oscillate around $\xi_{k}$. If
$x_{t}^{\psi^{*}}\rightarrow\hat{\xi}\ll\hat{\xi}_{k}$ as $t\rightarrow\infty$,
and there exists $T>0$ with $x_{T}^{\psi^{*}}\ll\hat{\xi}_{k}$,
then $x_{T}^{\varphi}\ll x_{T}^{\psi^{*}}\ll\hat{\xi}_{k}$, which
contradicts $\varphi\in\overline{S}_{k}$. Therefore, $\omega\left(\psi^{*}\right)=\left\{ \hat{\xi}_{k}\right\} .$
Similarly, $\omega\left(\varphi^{*}\right)=\left\{ \hat{\xi}_{k}\right\} $.
This is a contradiction to Proposition \ref{prop:solutions converging to xi_k}.\end{proof}
\begin{prop}
\label{prop:V on S_k}If $\varphi\in\overline{S_{k}}$, $\psi\in\overline{S_{k}}$
and $\varphi\neq\psi$, then $V\left(\psi-\varphi\right)=2$.\end{prop}
\begin{proof}
If $\varphi,\psi\in S_{k}$ and $\varphi\neq\psi$, then $V\left(\psi-\varphi\right)\leq2$
by Proposition \ref{prop: V on unstable set}. The lower-semicontinuity
of $V$ (see Lemma \ref{lem: continuity_of_V}) hence implies that
$V\left(\psi-\varphi\right)\leq2$ for all $\varphi,\psi\in\overline{S_{k}}$
satisfying $\varphi\neq\psi$. If $V\left(\psi-\varphi\right)=0$,
then $\varphi<\psi$ or $\psi<\varphi$, which contradicts Proposition
\ref{prop:nonordeing of S_k}.
\end{proof}
The role of $\pi_{3}$ in the proof of Theorem \ref{main_theorem_3}
is now taken over by the linear map 
\[
\pi_{2}:C\ni\varphi\mapsto\left(\varphi\left(0\right),\varphi\left(-1\right)\right)\in\mathbb{R}^{2}.
\]
The next assertion is analogous to Proposition \ref{prop:Pi3 on unstable set},
and it will be used several times in the subsequent proofs. 
\begin{prop}
\label{prop:Pi2 on S_k}$\pi_{2}$ is injective on \textup{$\overline{S_{k}}$.} \end{prop}
\begin{proof}
Suppose  that there exist $\varphi\in\overline{S_{k}}$ and $\psi\in\overline{S_{k}}$
so that $\varphi\neq\psi$ and $\pi_{2}\varphi=\pi_{2}\psi$. Consider
the solutions $x^{\varphi}:\mathbb{R}\rightarrow\mathbb{R}$ and $x^{\psi}:\mathbb{R}\rightarrow\mathbb{R}$.
The invariance of $\overline{S_{k}}$ implies that $x_{t}^{\varphi}\in\overline{S_{k}}$
and $x_{t}^{\psi}\in\overline{S_{k}}$ for all $t\in\mathbb{R}$,
and the the injectivity of the semiflow guarantees that $x_{t}^{\varphi}\neq x_{t}^{\psi}$
for all $t\in\mathbb{R}$. Hence $V\left(x_{t}^{\varphi}-x_{t}^{\psi}\right)=2$
for all real $t$ by Proposition \ref{prop:V on S_k}. The initial
assumption $\varphi\left(0\right)-\psi\left(0\right)=\varphi\left(-1\right)-\psi\left(-1\right)=0$
and Lemma \ref{lem:4_properties_of_V} (ii) however yield that 
\[
V\left(\varphi-\psi\right)<V\left(x_{-2}^{\varphi}-x_{-2}^{\psi}\right),
\]
 which is a contradiction.
\end{proof}
The injectiviy of $\pi_{2}|_{\overline{S_{k}}}$ is sufficient to
give a graph representation for $\overline{S_{k}}$. 
\begin{prop}
\label{prop:graph representation for the closure of S_k}$\overline{S_{k}}$
has a global graph representation: there exist a projection $P_{2}$
from $C$ onto a two-dimensional subspace $G_{2}$ of $C$ and a map
$w_{k}:P_{2}\overline{S_{k}}\rightarrow P_{2}^{-1}\left(0\right)$
so that
\begin{equation}
\overline{S_{k}}=\left\{ \chi+w_{k}\left(\chi\right):\,\chi\in P_{2}\overline{S_{k}}\right\} .\label{repr for the closure of S_k}
\end{equation}
\end{prop}
\begin{proof}
Let $e_{1}=\left(1,0,0\right)$ and $e_{2}=\left(0,1,0\right)$.\emph{
}Let\emph{ }$\varphi_{1}$ and $\varphi_{2}$ be the linearly independent
elements of $C$ fixed in the proof of Theorem \ref{main_theorem_3}
with the property that $\pi_{3}\varphi_{j}=e_{j}$ for $j\in\left\{ 1,2\right\} $.
Define $J_{2}:\mathbb{R}^{2}\rightarrow C$ to be the injective linear
map for which $J_{2}\left(1,0\right)=\varphi_{1}$ and $J_{2}\left(0,1\right)=\varphi_{2}$,
and set $P_{2}=J_{2}\circ\pi_{2}:C\rightarrow C$. Then $P_{2}$ is
continuous, linear and $P_{2}\varphi_{j}=\varphi_{j}$ for both $j\in\left\{ 1,2\right\} $.
Hence $P_{2}\circ P_{2}=P_{2}$, and $P_{2}$ is a projection. The
$2$-dimensional image space
\[
G_{2}=P_{2}C=\left\{ c_{1}\varphi_{1}+c_{2}\varphi_{2}:\, c_{1},c_{2}\in\mathbb{R}\right\} 
\]
is a subspace of $G_{3}$ and $C=G_{2}\oplus P_{2}^{-1}\left(0\right)$.
(Note that $P_{2}$ and $G_{2}$ are both independent of $k$.) As
the restriction of $P_{2}$ to $\overline{S_{k}}$ is injective by
Proposition \ref{prop:Pi2 on S_k}, the inverse $\left(P_{2}|_{\overline{S_{k}}}\right)^{-1}$
of the map $\overline{S_{k}}\ni\varphi\mapsto P_{2}\varphi\in G_{2}$
exists. With the map 
\[
w_{k}:P_{2}\overline{S_{k}}\ni\chi\mapsto\left(\mbox{id}-P_{2}\right)\circ\left(P_{2}|_{\overline{S_{k}}}\right)^{-1}\left(\chi\right)\in P_{2}^{-1}\left(0\right)
\]
we have \eqref{repr for the closure of S_k}. 
\end{proof}
The smoothness of this representation will be verified later. Observe
that 
\[
w_{-1}|_{P_{2}\left(\overline{S_{-1}}\cap\overline{S_{1}}\right)}=w_{1}|_{P_{2}\left(\overline{S_{-1}}\cap\overline{S_{1}}\right)}.
\]
Also note that now we have a global graph representation for any subset
$W$ of $\overline{S_{k}}$: 
\[
W=\left\{ \chi+w_{k}\left(\chi\right):\,\chi\in P_{2}W\right\} .
\]

Let $\pi_{2}^{-1}:\pi_{2}\left(\overline{S_{k}}\right)\rightarrow C$
be the inverse of the injective map $\overline{S_{k}}\ni\varphi\mapsto\pi_{2}\varphi\in\mathbb{R}^{2}$. 
\begin{prop}
\label{prop:inverse of Pi_2 is Lip-cont}$\pi_{2}^{-1}$ is Lipschitz-continuous.\end{prop}
\begin{proof}
Suppose that $\pi_{2}^{-1}$ is not Lipschitz-continuous, i.e., there
are sequences of solutions $x^{n}:\mathbb{R}\rightarrow\mathbb{R}$
and $y^{n}:\mathbb{R}\rightarrow\mathbb{R}$, $n\in\mathbb{N}$, so
that $x_{0}^{n}\neq y_{0}^{n}$ for all $n\geq0$, $x_{0}^{n},y_{0}^{n}\in\overline{S_{k}}$
for all $n\geq0$, and 
\[
\frac{\left|\pi_{2}\left(x_{0}^{n}-y_{0}^{n}\right)\right|_{\mathbb{R}^{2}}}{\left\Vert x_{0}^{n}-y_{0}^{n}\right\Vert }\rightarrow0\mbox{\quad}\mbox{as }n\rightarrow\infty.
\]
By the compactness of $\overline{S_{k}}$, the solutions $x^{n}$
and $y^{n}$ are uniformly bounded, and Eq.\,\eqref{eq:eq_general}
gives a uniform bound for their derivatives. Therefore we can use
the Arzel\`{a}\textendash{}Ascoli theorem successively on the intervals
$\left[-j,j\right]$, $j\geq1$, and apply a diagonalization process
to get subsequences $\left(x^{n_{m}}\right)_{m=0}^{\infty}$, $\left(y^{n_{m}}\right)_{m=0}^{\infty}$
and continuous functions $x:\mathbb{R}\rightarrow\mathbb{R}$, $y:\mathbb{R}\rightarrow\mathbb{R}$
so that $x^{n_{m}}\rightarrow x$ and $y^{n_{m}}\rightarrow y$ as
$m\rightarrow\infty$ uniformly on compact subsets of $\mathbb{R}$. 

Set functions 
\[
z^{m}:\mathbb{R}\ni t\mapsto\frac{x^{n_{m}}\left(t\right)-y^{n_{m}}\left(t\right)}{\left\Vert x_{0}^{n_{m}}-y_{0}^{n_{m}}\right\Vert }\in\mathbb{R},\quad m\in\mathbb{N}.
\]
Then $V\left(z_{t}^{m}\right)=2$ for all $m\geq0$ and $t\in\mathbb{R}$
by Proposition \ref{prop:V on S_k}, $\left\Vert z_{0}^{m}\right\Vert =1$
for all $m\geq0$, and 
\[
\left|\pi_{2}z_{0}^{m}\right|_{\mathbb{R}^{2}}=\frac{\left|\pi_{2}\left(x_{0}^{n_{m}}-y_{0}^{n_{m}}\right)\right|_{\mathbb{R}^{2}}}{\left\Vert x_{0}^{n_{m}}-y_{0}^{n_{m}}\right\Vert }\rightarrow0\mbox{\quad}\mbox{as }m\rightarrow\infty.
\]
In addition, $\dot{z}^{m}\left(t\right)=-z^{m}\left(t\right)+a_{m}\left(t\right)z^{m}\left(t-1\right)$
for all $m\geq0$ and $t\in\mathbb{R}$, where the coefficient functions
\[
a_{m}:\mathbb{R}\ni t\mapsto\int_{0}^{1}f'\left(sx^{n_{m}}\left(t-1\right)+\left(1-s\right)y^{n_{m}}\left(t-1\right)\right)\mbox{d}s\in\mathbb{R}^{+},\quad m\geq0,
\]
converge to 
\[
a:\mathbb{R}\ni t\mapsto\int_{0}^{1}f'\left(sx\left(t-1\right)+\left(1-s\right)y\left(t-1\right)\right)\mbox{d}s\in\mathbb{R}^{+}
\]
uniformly on compact subsets of $\mathbb{R}$. It is also obvious
that there are constants $\alpha_{1}\geq\alpha_{0}>0$ such that $\alpha_{0}\leq a_{m}\left(t\right)\leq\alpha_{1}$
for all $m\geq0$ and $t\in\mathbb{R}$. 

Therefore Lemma \ref{lem:technical result} guarantees the existence
of a subsequence $\left(z^{m_{l}}\right)_{l=0}^{\infty}$ of $\left(z^{m}\right)_{m=0}^{\infty}$
and a continuously differentiable function $z:\mathbb{R}\rightarrow\mathbb{R}$
such that $z^{m_{l}}\rightarrow z$ and $\dot{z}^{m_{l}}\rightarrow\dot{z}$
as $l\rightarrow\infty$ uniformly on compact subsets of $\mathbb{R}$,
and $z$ satisfies 
\begin{align*}
\dot{z}\left(t\right) & =-z\left(t\right)+a\left(t\right)z\left(t-1\right)\quad\mbox{for all }t\in\mathbb{R}.
\end{align*}
It is clear that $\left\Vert z_{0}\right\Vert =1$, and thus $z_{0}\neq\hat{0}$.
In addition, $\pi_{2}z_{0}=\left(0,0\right)$.

By Lemma \ref{lem: continuity_of_V}, 
\[
V\left(z_{t}\right)\leq\liminf_{l\rightarrow\infty}V\left(z_{t}^{m_{l}}\right)=2\quad\mbox{for all real }t.
\]
Hence Lemma \ref{lem:4_properties_of_V} (ii) and property $\pi_{2}z_{0}=\left(0,0\right)$
together give that $V\left(z_{0}\right)=0$. As $t\mapsto V\left(z_{t}\right)$
is monotone nonincreasing, $V\left(z_{3}\right)=0$. Lemma \ref{lem:4_properties_of_V}
(iii) then implies that $z_{3}$ belongs to the function class $R$,
and $ $the second statement of Lemma \ref{lem: continuity_of_V}
gives that 
\[
0=V\left(z_{3}\right)=\lim_{l\rightarrow\infty}V\left(z_{3}^{m_{l}}\right),
\]
which contradicts $V\left(z_{3}^{m_{l}}\right)=2$. 
\end{proof}
We get the next result as a consequence, it is analogous to Proposition
\ref{prop:Pi3 on tangent vectors}.
\begin{prop}
\label{prop:Pi_2 on tangent vectors}Suppose that $\varphi\in\overline{S_{k}}$,
$\gamma:\left(-1,1\right)\rightarrow C$ is a $C^{1}$-curve with
$\gamma\left(0\right)=\varphi$, and $\left(s_{n}\right)_{0}^{\infty}$
is a sequence in $\left(-1,1\right)\backslash\left\{ 0\right\} $
so that $s_{n}\rightarrow0$ as $n\rightarrow\infty$ and $\gamma\left(s_{n}\right)\in\overline{S_{k}}$
for all $n\geq0$. If $\gamma'\left(0\right)\neq\hat{0}$, then $\pi_{2}\gamma'\left(0\right)\neq\left(0,0\right)$.\end{prop}
\begin{proof}
Let $K>0$ be a Lipschitz-constant for $\pi_{2}^{-1}$. Proposition
\ref{prop:inverse of Pi_2 is Lip-cont} guarantees that such $K$
exists. Then 
\[
\left\Vert \frac{\gamma\left(s_{n}\right)-\gamma\left(0\right)}{s_{n}}\right\Vert \leq K\left|\frac{\pi_{2}\gamma\left(s_{n}\right)-\pi_{2}\gamma\left(0\right)}{s_{n}}\right|_{\mathbb{R}^{2}}
\]
for all $n\geq0$. Letting $ $$n\rightarrow\infty$ we obtain that
$\left\Vert \gamma'\left(0\right)\right\Vert \leq K\left|\pi_{2}\gamma'\left(0\right)\right|_{\mathbb{R}^{2}}$.
Therefore if $\gamma'\left(0\right)\neq\hat{0}$, then $\pi_{2}\gamma'\left(0\right)\neq\left(0,0\right)$.
\end{proof}
\noindent \textbf{~}

\noindent \begin{center}
\textit{5.2 The structure of $\overline{S_{k}}$}
\par\end{center}

\smallskip{}

It is obvious from the definition of $S_{k}$ that $\left(\mathcal{O}_{k}\cup S_{k}\cup\mathcal{O}_{q}\right)\subset\overline{S_{k}}$.
The converse inclusion is proved in this subsection based on the property
that $\pi_{2}$ maps $\overline{S_{k}}$ injectively into $\mathbb{R}^{2}$.
Then it will follow easily that $\overline{C_{q}^{p}}=\mathcal{O}_{p}\cup C_{q}^{p}\cup\mathcal{O}_{q}$
and $\overline{C_{k}^{p}}=\mathcal{O}_{p}\cup C_{k}^{p}\cup\mathcal{O}_{k}$. 

Proposition \ref{prop:Pi2 on S_k} implies that $\pi_{2}$ maps periodic
orbits with segments in $\overline{S_{k}}$ into simple closed curves
in $\mathbb{R}^{2}$, and the images of different periodic orbits
are disjoint curves in $\mathbb{R}^{2}$. So 
\[
\mathbb{R}\ni t\mapsto\pi_{2}p_{t}\in\mathbb{R}^{2},\ \mathbb{R}\ni t\mapsto\pi_{2}q_{t}\in\mathbb{R}^{2}
\]
and 
\[
\mathbb{R}\ni t\mapsto\pi_{2}x_{t}^{k}\in\mathbb{R}^{2}
\]
are pairwise disjoint simple closed curves. From $p\left(\mathbb{R}\right)\subsetneq q\left(\mathbb{R}\right)\subset\left(\xi_{-2},\xi_{2}\right)$
it follows that $\pi_{2}\mathcal{O}_{q}\subset\mbox{ext}\left(\pi_{2}\mathcal{O}_{p}\right)$
and $\pi_{2}\hat{\xi}_{-2}$, $\pi_{2}\hat{\xi}_{2}$ belong to $\mbox{ext}\left(\pi_{2}\mathcal{O}_{q}\right)$.
It is also obvious that $\pi_{2}\hat{0}\in\mbox{int}\left(\pi_{2}\mathcal{O}_{p}\right)$
and $\pi_{2}\mathcal{O}_{k}\in\mbox{int}\left(\pi_{2}\mathcal{O}_{p}\right)$.
For the image of the unstable equilibrium $\hat{\xi}_{k}$, we have
$\pi_{2}\hat{\xi}_{k}\in\mbox{int}\left(\pi_{2}\mathcal{O}_{k}\right)$.
See Fig.~6.

Let 
\[
A_{k}^{p}=\mbox{ext}\left(\pi_{2}\mathcal{O}_{k}\right)\cap\mbox{int}\left(\pi_{2}\mathcal{O}_{p}\right),\quad A_{q}^{p}=\mbox{ext}\left(\pi_{2}\mathcal{O}_{p}\right)\cap\mbox{int}\left(\pi_{2}\mathcal{O}_{q}\right)
\]
and 
\[
A_{k,q}=\mbox{ext}\left(\pi_{2}\mathcal{O}_{k}\right)\cap\mbox{int}\left(\pi_{2}\mathcal{O}_{q}\right),
\]
see Fig.~6. Then by the Sch\"onflies theorem \cite{Schonflies},
$A_{k}^{p}$, $A_{q}^{p}$ and $A_{k,q}$ are homeomorphic to the
open annulus $A^{\left(1,2\right)}=\left\{ u\in\mathbb{R}^{2}:\,1<\left|u\right|<2\right\} $.
For the closures $\overline{A_{k}^{p}},$ $\overline{A_{q}^{p}}$
and $\overline{A_{k,q}}$ of $A_{k}^{p}$, $A_{q}^{p}$ and $A_{k,q}$
in $\mathbb{R}^{2}$, respectively, we have $ $ 
\[
\overline{A_{k}^{p}}=A_{k}^{p}\cup\pi_{2}\mathcal{O}_{k}\cup\pi_{2}\mathcal{O}_{p},\quad\overline{A_{q}^{p}}=A_{q}^{p}\cup\pi_{2}\mathcal{O}_{p}\cup\pi_{2}\mathcal{O}_{q}
\]
and 
\[
\overline{A_{k,q}}=A_{k,q}\cup\pi_{2}\mathcal{O}_{k}\cup\pi_{2}\mathcal{O}_{q}.
\]

Observe that for all $\varphi\in C_{q}^{p}$, $\pi_{2}\varphi\in A_{q}^{p}$
because $t\mapsto\pi_{2}x_{t}^{\varphi}$ is continuous, $\pi_{2}x_{t}^{\varphi}\rightarrow\pi_{2}\mathcal{O}_{p}$
as $t\rightarrow-\infty$, $\pi_{2}x_{t}^{\varphi}\rightarrow\pi_{2}\mathcal{O}_{q}$
as $t\rightarrow\infty$, $\mathcal{O}_{p}\cup C_{q}^{p}\cup\mathcal{O}_{q}\subset\overline{S_{k}}$,
and $\pi_{2}$ is injective on $\overline{S_{k}}$. For the same reason,
$\pi_{2}C_{k}^{p}\subseteq A_{k}^{p}$. Then it is clear that $\pi_{2}\overline{C_{q}^{p}}=\overline{\pi_{2}C_{q}^{p}}\subseteq\overline{A_{q}^{p}}$
and $\pi_{2}\overline{C_{k}^{p}}=\overline{\pi_{2}C_{k}^{p}}\subseteq\overline{A_{k}^{p}}$.
As $\mathcal{O}_{p}\subseteq\overline{C_{q}^{p}}\cap\overline{C_{k}^{p}}$,
we conclude that 
\[
\pi_{2}\mathcal{O}_{p}\subseteq\pi_{2}\left(\overline{C_{q}^{p}}\cap\overline{C_{k}^{p}}\right)\subseteq\pi_{2}\overline{C_{q}^{p}}\cap\pi_{2}\overline{C_{k}^{p}}\subseteq\overline{A_{q}^{p}}\cap\overline{A_{k}^{p}}=\pi_{2}\mathcal{O}_{p},
\]
that is, $ $$\pi_{2}\mathcal{O}_{p}=\pi_{2}\left(\overline{C_{q}^{p}}\cap\overline{C_{k}^{p}}\right)$.
The injectivity of $\pi_{2}$ on $\overline{S_{k}}$ then implies
that 
\begin{equation}
\mathcal{O}_{p}=\overline{C_{q}^{p}}\cap\overline{C_{k}^{p}}.\label{the intersection of the closures of two connecting orbits}
\end{equation}

We also obtain from $\pi_{2}C_{q}^{p}\subseteq A_{q}^{p}$ and $\pi_{2}C_{k}^{p}\subseteq A_{k}^{p}$
that 
\[
\pi_{2}S_{k}=\pi_{2}C_{k}^{p}\cup\pi_{2}\mathcal{O}_{p}\cup\pi_{2}C_{q}^{p}\subseteq A_{k}^{p}\cup\pi_{2}\mathcal{O}_{p}\cup A_{q}^{p}=A_{k,q},
\]
and hence $\pi_{2}\overline{S_{k}}=\overline{\pi_{2}S_{k}}\subseteq\overline{A_{k,q}}$.
Note that this means that $\hat{\xi}_{k}\notin\overline{S_{k}}$.

\begin{center}
\begin{figure}[h]
\begin{centering}
\includegraphics[scale=0.6]{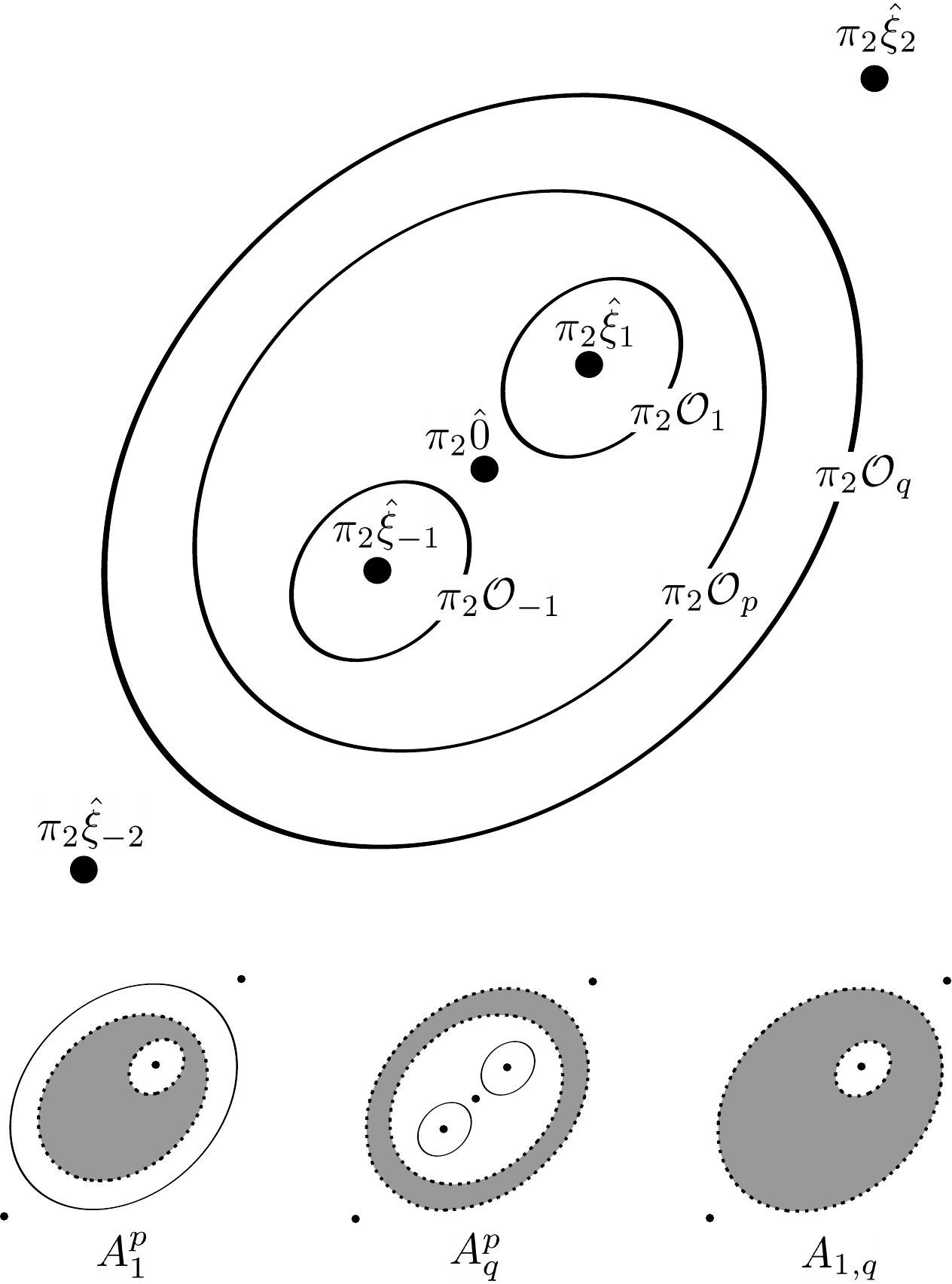} 
\par\end{centering}

\caption{The images of the equilibria and the periodic orbits under $\pi_{2}$,
and the definitions of the open sets $A_{1}^{p}$, $A_{q}^{p}$ and
$A_{1,q}$. }
\end{figure}

\par\end{center}

It has been already verified that for all $\varphi\in\overline{S_{k}}$,
$x^{\varphi}$ oscillates around $\xi_{k}$. We claim that this oscillation
is slow.
\begin{prop}
$V\left(\varphi-\hat{\xi}_{k}\right)=2$ for all $\varphi\in\overline{S_{k}}$. \end{prop}
\begin{proof}
1. First we prove the assertion for the elements of $S_{k}$. Choose
an arbitrary element $\varphi\in S_{k}$ and a sequence $\left(t_{n}\right)_{n=0}^{\infty}$
with $t_{n}\rightarrow-\infty$ as $n\rightarrow\infty$ such that
$x_{t_{n}}^{\varphi}\rightarrow p_{0}$ as $n\rightarrow\infty$.
As the $C$-norm and $C^{1}$-norm are equivalent on the global attractor,
$x_{t_{n}}^{\varphi}\rightarrow p_{0}$ as $n\rightarrow\infty$ also
in $C^{1}$-norm. Note that $p$ is slowly oscillatory around $\xi_{k}$
(see Proposition 8.2 in \cite{Krisztin-Vas}), i.e., $V\left(p_{t}-\hat{\xi}_{k}\right)=2$
for all real $t.$ Hence Lemma \ref{lem:4_properties_of_V}.(iii)
gives that $p_{0}-\hat{\xi}_{k}\in R$, and Lemma \ref{lem: continuity_of_V}
implies that 
\[
2=V\left(p_{0}-\hat{\xi}_{k}\right)=\lim_{n\rightarrow\infty}V\left(x_{t_{n}}^{\varphi}-\hat{\xi}_{k}\right).
\]
Then by the monotonicity of $V$ (see Lemma \ref{lem:4_properties_of_V}.(i)),
$V\left(x_{t}^{\varphi}-\hat{\xi}_{k}\right)\leq2$ for all $t\in\mathbb{R}$.
If $V\left(\varphi-\hat{\xi}_{k}\right)=0$ and $\varphi<\hat{\xi}_{k}$
or $\varphi>\hat{\xi}_{k}$, then $x_{2}^{\varphi}\ll\hat{\xi}_{k}$
or $x_{2}^{\varphi}\gg\hat{\xi}_{k}$ by Proposition \ref{pro:monotone_dynamical_system},
which contradicts the fact that $x^{\varphi}$ oscillates around $\xi_{k}$.

2. Now choose any $\varphi\in\overline{S_{k}}$ and fix a sequence
$\left(\varphi_{n}\right)_{n=0}^{\infty}$ in $S_{k}$ with $\varphi_{n}\rightarrow\varphi$
as $n\rightarrow\infty$. Since $\hat{\xi}_{k}\notin\overline{S_{k}}$,
$V\left(\varphi-\hat{\xi}_{k}\right)$ is defined. The lower semi-continuity
of $V$ (see Lemma \ref{lem: continuity_of_V}) and part 1 yield that
$V\left(\varphi-\hat{\xi}_{k}\right)\leq\liminf_{n\rightarrow\infty}V\left(\varphi_{n}-\hat{\xi}_{k}\right)=2$.
Observe that assumption $V\left(\varphi-\hat{\xi}_{k}\right)=0$ would
lead to a contradiction just as in the previous step. So $V\left(\varphi-\hat{\xi}_{k}\right)=2$
for all $\varphi\in\overline{S_{k}}$.
\end{proof}
Now we are ready to confirm the equalities regarding $\overline{C_{k}^{p}}$,
$\overline{C_{q}^{p}}$ and $\overline{S_{k}}$ in Theorem \ref{main_theorem_4}.(ii). 
\begin{prop}
\label{prop:the structure of the closure of S_k} $\overline{S_{k}}=\mathcal{O}_{k}\cup S_{k}\cup\mathcal{O}_{q}=\mathcal{O}_{k}\cup C_{k}^{p}\cup\mathcal{O}_{p}\cup C_{q}^{p}\cup\mathcal{O}_{q}$. \end{prop}
\begin{proof}
Let us fix $k=1$. It is clear from the definition of $S_{1}$ that
$\left(\mathcal{O}_{1}\cup\mathcal{O}_{q}\right)\subset\overline{S_{1}}$,
and thus we only need to verify the inclusion $\overline{S_{1}}\setminus S_{1}\subseteq\left(\mathcal{O}_{1}\cup\mathcal{O}_{q}\right)$.
Let $\varphi\in\overline{S_{1}}\setminus S_{1}$ be arbitrary.

It is an immediate consequence of the oscillation of $x^{\varphi}$
around $\xi_{1}$ that $\varphi\notin\mathcal{W}^{u}\left(\mathcal{O}_{p}\right)$,
otherwise $\varphi$ would also belong to $S_{1}$ by \eqref{S_k}.
It is also obvious that $\varphi\notin\mathcal{A}_{-2,0}$. There
are two possibilities by Theorem B: either $\varphi\in\mathcal{W}^{u}\left(\mathcal{O}_{q}\right)$
or $\varphi\in\mathcal{A}_{0,2}$. If $\varphi\in\mathcal{W}^{u}\left(\mathcal{O}_{q}\right)$,
then necessarily $\varphi\in\mathcal{O}_{q}$, otherwise $x^{\varphi}$
would converge to one of the equilibria $\hat{\xi}_{-2}$, $\hat{\xi}_{2}$
by Theorem B. So it remains to show that the relation $\varphi\in\mathcal{A}_{0,2}$
implies that $\varphi\in\mathcal{O}_{1}$. 

$\mathcal{A}_{0,2}$ is a compact and invariant subset of $C$, hence
$\varphi\in\mathcal{A}_{0,2}$ implies that $x_{t}^{\varphi}\in\mathcal{A}_{0,2}$
for all real $t$, moreover $\alpha\left(x^{\varphi}\right)$ and
$\omega\left(\varphi\right)$ are also subsets of $ $$\mathcal{A}_{0,2}$.
On the other hand, $\overline{S_{1}}$ is also compact and invariant,
so $\alpha\left(x^{\varphi}\right)\cup\omega\left(\varphi\right)\subset\overline{S_{1}}$,
and $V\left(\psi-\hat{\xi}_{1}\right)=2$ for all $\psi\in\alpha\left(x^{\varphi}\right)\cup\omega\left(\varphi\right)$
by the previous proposition. The Poincar\'e--Bendixson Theorem (see
Section \ref{sec:Prelimimaries}) then implies that $\omega\left(\varphi\right)$
is either a periodic orbit in $\mathcal{A}_{0,2}$ oscillating slowly
around $\xi_{1}$, or for each $\psi\in\omega\left(\varphi\right)$,
$\alpha\left(x^{\psi}\right)=\omega\left(\psi\right)=\left\{ \hat{\xi}_{1}\right\} .$
As there are no homoclinic orbits to $\hat{\xi}_{1}$ (see Proposition
3.1 in \cite{Krisztin-3}), $\omega\left(\varphi\right)=\left\{ \hat{\xi}_{1}\right\} $
in the latter case. Similarly, $\alpha\left(x^{\varphi}\right)$ is
either $\left\{ \hat{\xi}_{1}\right\} $ or a periodic orbit in $\mathcal{A}_{0,2}$
oscillating slowly around $\xi_{1}$. 

Recall that $x^{1}$ is defined so that the range $x^{1}(\mathbb{R})$
is maximal in the sense that $x^{1}(\mathbb{R})\supset r(\mathbb{R})$
for all periodic solutions $r$ oscillating slowly around $\xi_{1}$
with range in $(0,\xi_{2})$. So if $r:\mathbb{R}\rightarrow\mathbb{R}$
is a periodic solution with segments in $\alpha\left(x^{\varphi}\right)\cup\omega\left(\varphi\right)$,
then either $r$ is the time translation of $x^{1}$, or $\pi_{2}r_{t}\in\mbox{int}\left(\pi_{2}\mathcal{O}_{1}\right)$
for all $t\in\mathbb{R}$. Recall that $\pi_{2}\xi_{1}$ also belongs
to $\mbox{int}\left(\pi_{2}\mathcal{O}_{1}\right)$. On the other
hand, 
\[
\pi_{2}\left(\alpha\left(x^{\varphi}\right)\cup\omega\left(\varphi\right)\right)\subset\pi_{2}\overline{S_{k}}\subseteq\overline{A_{k,q}}\subset\mathbb{R}^{2}\setminus\mbox{int}\left(\pi_{2}\mathcal{O}_{1}\right).
\]
It follows that $\pi_{2}\left(\alpha\left(x^{\varphi}\right)\cup\omega\left(\varphi\right)\right)\subseteq\pi_{2}\mathcal{O}_{1}$
and thus$ $ $\alpha\left(x^{\varphi}\right)=\omega\left(\varphi\right)=\mathcal{O}_{1}$.
If $x^{\varphi}$ is not the time translation of $x^{1}$, then this
is only possible if the curve $t\rightarrow\pi_{2}x_{t}^{\varphi}$
is self-intersecting, which contradicts the injectivity of $\pi_{2}$
on $\overline{S_{1}}.$ Hence relation $\varphi\in\mathcal{A}_{0,2}$
implies that $\varphi\in\mathcal{O}_{1}$.

We have verified that each $\varphi\in\overline{S_{1}}\setminus S_{1}$
belongs to $\mathcal{O}_{1}\cup\mathcal{O}_{q}$, that is 
\[
\overline{S_{1}}=\mathcal{O}_{1}\cup C_{1}^{p}\cup\mathcal{O}_{p}\cup C_{q}^{p}\cup\mathcal{O}_{q}.
\]
 Handling the case $k=-1$ is completely analogous. \end{proof}
\begin{cor}
\label{cor: closures of connecting sets} $\overline{S_{-1}}\cap\overline{S_{1}}=\mathcal{O}_{p}\cup C_{q}^{p}\cup\mathcal{O}_{q}$,
$\overline{C_{q}^{p}}=\mathcal{O}_{p}\cup C_{q}^{p}\cup\mathcal{O}_{q}$
and $\overline{C_{k}^{p}}=\mathcal{O}_{p}\cup C_{k}^{p}\cup\mathcal{O}_{k}$. \end{cor}
\begin{proof}
The first equality follows immediately from Proposition \ref{prop:the structure of the closure of S_k}.
The second and third equalities come from 
\[
\mathcal{O}_{p}\cup C_{q}^{p}\cup\mathcal{O}_{q}\subseteq\overline{C_{q}^{p}}\subseteq\overline{S_{k}}=\mathcal{O}_{k}\cup C_{k}^{p}\cup\mathcal{O}_{p}\cup C_{q}^{p}\cup\mathcal{O}_{q},
\]
\[
\mathcal{O}_{k}\cup C_{k}^{p}\cup\mathcal{O}_{p}\subseteq\overline{C_{k}^{p}}\subseteq\overline{S_{k}}=\mathcal{O}_{k}\cup C_{k}^{p}\cup\mathcal{O}_{p}\cup C_{q}^{p}\cup\mathcal{O}_{q}
\]
 and \eqref{the intersection of the closures of two connecting orbits}.
\end{proof}
\noindent \begin{center}
\textit{5.3 The smoothness of $ $$C_{q}^{p}$ and $C_{k}^{p}$} \smallskip{}

\par\end{center}

Suppose $r$ is one of the periodic solutions $q$ or $x^{k}$ with
minimal period $\omega>1$, and let $C_{r}^{p}$ be the heteroclinic
connection from $\mathcal{O}_{p}$ to $O_{r}=\left\{ r_{t}:\, t\in\mathbb{R}\right\} $. 

Next we confirm that $C_{r}^{p}$ is a $C^{1}$-submanifold of $\mathcal{W}^{u}\left(\mathcal{O}_{p}\right)$.
First we verify that $\mathcal{W}^{u}\left(\mathcal{O}_{p}\right)$
intersects transversally a local stable or a local center-stable manifold
of a Poincar\'e map at a point of $\mathcal{O}_{r}$. It follows
that the intersection is a one-dimensional $C^{1}$-submanifold of
$\mathcal{W}^{u}\left(\mathcal{O}_{p}\right)$. Then we apply the
injectivity of the derivative of the flow induced by the solution
operator on $\mathcal{W}^{u}\left(\mathcal{O}_{p}\right)$ (see Proposition
\ref{the flow is continuously differentiable} and Corollary \ref{injective derivative})
to confirm that each point $\varphi$ in $C_{r}^{p}$ belongs to a
``small'' subset $W_{\varphi}$ of $C_{r}^{p}$ that is a two-dimensional
$C^{1}$-submanifold of $\mathcal{W}^{u}\left(\mathcal{O}_{p}\right)$.
This means that $C_{r}^{p}$ is an immersed $C^{1}$-submanifold of
$\mathcal{W}^{u}\left(\mathcal{O}_{p}\right)$. In order to prove
that $C_{r}^{p}$ is embedded in $\mathcal{W}^{u}\left(\mathcal{O}_{p}\right)$,
we have to show that for any $\varphi$ in $C_{r}^{p}$, there is
no sequence in $C_{r}^{p}\backslash W_{\varphi}$ converging to $\varphi$.
According to results of Subsection 5.1, $ $$\pi_{2}$ is injective
on $C_{r}^{p}$ and on the tangent spaces of $C_{r}^{p}$, which implies
that $\pi_{2}W_{\varphi}$ is open in $\mathbb{R}^{2}$. If a sequence
$\left(\varphi^{n}\right)_{n=0}^{\infty}$ from the rest of the connecting
set converges to $\varphi$ as $n\rightarrow\infty$, then $\pi_{2}\varphi^{n}\rightarrow\pi_{2}\varphi$
as $n\rightarrow\infty$, and $\pi_{2}\varphi^{n}\in\pi_{2}W_{\varphi}$
for all $n$ large enough. The injectivity of $\pi_{2}$ on $\overline{S_{k}}$
then implies that $\varphi^{n}\in W_{\varphi}$, which is a contradiction.
So $C_{r}^{p}$ is embedded in $\mathcal{W}^{u}\left(\mathcal{O}_{p}\right)$.
With the projection $P_{2}$ and the map $w_{k}$ from Proposition
\ref{prop:graph representation for the closure of S_k}, 
\[
C_{r}^{p}=\left\{ \chi+w_{k}\left(\chi\right):\,\chi\in P_{2}C_{r}^{p}\right\} .
\]
Using the previously obtained result that $C_{r}^{p}$ is a $C^{1}$-submanifold
of $\mathcal{W}^{u}\left(\mathcal{O}_{p}\right)$, we prove at the
end of this subsection that $w_{k}$ is continuously differentiable
on the open set $P_{2}C_{r}^{p}$, i.e., this representation for $C_{r}^{p}$
is smooth. 

Section \ref{sec:Floquet-multipliers and invariant manifolds} has
introduced a hyperplane $Y$, a convex bounded open neighborhood $N$
of $r_{0}$ in $C$, $\varepsilon\in\left(0,\omega\right)$ and a
$C^{1}$-map $\gamma:N\rightarrow\left(\omega-\varepsilon,\omega+\varepsilon\right)$
with $\gamma\left(r_{0}\right)=\omega$ so that for each $\left(t,\varphi\right)\in\left(\omega-\varepsilon,\omega+\varepsilon\right)\times N$,
the segment $x_{t}^{\varphi}$ belongs to $r_{0}+Y$ if and only if
$t=\gamma(\varphi)$. A Poincar\'e return map $P_{Y}$ has been defined
as
\[
P_{Y}:N\cap\left(r_{0}+Y\right)\ni\varphi\mapsto\Phi\left(\gamma(\varphi),\varphi\right)\in r_{0}+Y.
\]

Let $\mathcal{W}$ denote a local stable manifold $\mathcal{W}_{loc}^{s}\left(P_{Y},r_{0}\right)$
of $P_{Y}$ at $r_{0}$ if $\mathcal{O}_{r}$ is hyperbolic, and let
$\mathcal{W}$ be a local center-stable manifold $\mathcal{W}_{loc}^{sc}\left(P_{Y},r_{0}\right)$
of $P_{Y}$ at $r_{0}$ otherwise. By Section \ref{sec:Floquet-multipliers and invariant manifolds},
$\mathcal{W}$ is a $C^{1}$-submanifold of $r_{0}+Y$ with codimension
$1$, and it is a $C^{1}$-submanifold of $C$ with codimension $2$. 

The subsequent proposition is an important step toward the proof of
the assertion that $C_{q}^{p}$ and $C_{k}^{p}$ are two-dimensional
$C^{1}$-submanifolds of $\mathcal{W}^{u}\left(\mathcal{O}_{p}\right)$. 
\begin{prop}
$\mathcal{W}^{u}\left(\mathcal{O}_{p}\right)\cap\mathcal{W}$ is a
one-dimensional $C^{1}$-submanifold of $\mathcal{W}^{u}\left(\mathcal{O}_{p}\right)$.\end{prop}
\begin{proof}
1. Theorem B and Proposition \ref{prop:trajectory of P_Y} imply that
$\mathcal{W}^{u}\left(\mathcal{O}_{p}\right)\cap\mathcal{W}$ is nonempty.
It suffices to verify that the inclusion map $i:\mathcal{W}^{u}\left(\mathcal{O}_{p}\right)\ni\varphi\mapsto\varphi\in C$
and $\mathcal{W}$ are transversal. Then it follows that $i^{-1}\left(\mathcal{W}\right)=\mathcal{W}^{u}\left(\mathcal{O}_{p}\right)\cap\mathcal{W}$
is a $C^{1}$-submanifold of $\mathcal{W}^{u}\left(\mathcal{O}_{p}\right)$,
furthermore it has the same codimension in $\mathcal{W}^{u}\left(\mathcal{O}_{p}\right)$
as $\mathcal{W}$ in $C$ (see e.g. Corollary 17.2 in \cite{Abraham-Robbin}).
Accordingly we show that the inclusion map $i:\mathcal{W}^{u}\left(\mathcal{O}_{p}\right)\ni\varphi\mapsto\varphi\in C$
and $\mathcal{W}$ are transversal. This means that for all $\varphi\in\mathcal{W}^{u}\left(\mathcal{O}_{p}\right)$
with $\varphi=i\left(\varphi\right)\in\mathcal{W}$,

(i) the inverse image $\left(Di\left(\varphi\right)\right)^{-1}T_{i\left(\varphi\right)}\mathcal{W}=T_{\varphi}\mathcal{W}^{u}\left(\mathcal{O}_{p}\right)\cap T_{\varphi}\mathcal{W}$
splits in $T_{\varphi}\mathcal{W}^{u}\left(\mathcal{O}_{p}\right)$
(it has a closed complementary subspace in $T_{\varphi}\mathcal{W}^{u}\left(\mathcal{O}_{p}\right)$),
and

(ii) the space $Di\left(\varphi\right)T_{\varphi}\mathcal{W}^{u}\left(\mathcal{O}_{p}\right)=T_{\varphi}\mathcal{W}^{u}\left(\mathcal{O}_{p}\right)$
contains a closed complement to $T_{i\left(\varphi\right)}\mathcal{W}=T_{\varphi}\mathcal{W}$
in $C$.

\noindent Property (i) holds because $\mbox{dim}T_{\varphi}\mathcal{W}^{u}\left(\mathcal{O}_{p}\right)=3<\infty$.
In the following we confirm (ii). 

2. Let $\varphi\in\mathcal{W}^{u}\left(\mathcal{O}_{p}\right)\cap\mathcal{W}$.
First note that the invariance of $\mathcal{W}^{u}\left(\mathcal{O}_{p}\right)$
ensures that $\dot{\varphi}\in T_{\varphi}\mathcal{W}^{u}\left(\mathcal{O}_{p}\right)$.
On the other hand, Proposition \ref{qvarphi dot near r_0} gives that
$\dot{\varphi}\notin Y$ can be assumed. Therefore $\dot{\varphi}\in T_{\varphi}\mathcal{W}^{u}\left(\mathcal{O}_{p}\right)\setminus T_{\varphi}\mathcal{W}.$

We claim that $T_{\varphi}\mathcal{W}^{u}\left(\mathcal{O}_{p}\right)$
contains a sign-preserving element $\chi$. Let $Z$ be the hyperplane
in $C$ with $C=\mathbb{R}\dot{p}_{0}\oplus Z$ and define a Poincar\'e
map $P_{Z}$ on a neighborhood of $p_{0}$ in $p_{0}+Z$ as in Section
\ref{sec:Floquet-multipliers and invariant manifolds}. (Here we use
exceptionally the notation $Z$ and $P_{Z}$ to emphasize the difference
from the above mentioned $Y$ and $P_{Y}$.) Choose $\psi$ from a
local unstable manifold $\mathcal{W}_{loc}^{u}\left(P_{Z},p_{0}\right)$
of $P_{Z}$ such that $\varphi=\Phi\left(T,\,\psi\right)$ for some
$T\geq0$. This is possible by \eqref{unstable set is forwad extension of unstable manifold-1}.
Choose $\eta$ to be a strictly positive vector in $T_{\eta}\mathcal{W}_{loc}^{u}\left(P_{Z},p_{0}\right)$.
Proposition \ref{pro: properties of local unstable manifold at r_0}
yields that the existence of such $\eta$ may be supposed without
loss of generality. Then $D_{2}\Phi\left(T,\psi\right)\eta\in T_{\varphi}\mathcal{W}^{u}\left(\mathcal{O}_{p}\right)$,
and $D_{2}\Phi\left(T,\psi\right)\eta=v_{T}^{\eta}$, where $v^{\eta}:\left[-1,\infty\right)\rightarrow\mathbb{R}$
is the solution of the linear variational equation 
\[
\dot{v}(t)=-v(t)+f'\left(x^{\psi}\left(t-1\right)\right)v\left(t-1\right)
\]
 with $v_{0}^{\eta}=\eta$. The monotonicity of $V$ implies that
$v_{T}^{\eta}$ is also strictly positive. So set $\chi=v_{T}^{\eta}$. 

Vectors $\dot{\varphi}$ and $\chi$ are linearly independent because
$V\left(\chi\right)=0$ and we may assume by Proposition \ref{qvarphi dot near r_0}
that $V\left(\dot{\varphi}\right)\geq2$.

3. As $T_{\varphi}\mathcal{W}$ is a subspace of $C$ with codimension
$2$, it suffices to confirm that 
\[
T_{\varphi}\mathcal{W}\cap\left(\mathbb{R}\dot{\varphi}\oplus\mathbb{R}\chi\right)=\left\{ \hat{0}\right\} .
\]

Suppose  that $a\dot{\varphi}+b\chi\in T_{\varphi}\mathcal{W}\backslash\left\{ \hat{0}\right\} $
for some $a,b\in\mathbb{R}$. Then $b\neq0$ as $\dot{\varphi}\notin T_{\varphi}\mathcal{W}$
. Set $c=a/b$ and consider the vector $ $$c\dot{\varphi}+\chi\in T_{\varphi}\mathcal{W}\backslash\left\{ \hat{0}\right\} $.
Let $v:\left[-1,\infty\right)\rightarrow\mathbb{R}$ be the solution
of the linear variational equation 
\begin{align*}
\dot{v}(t) & =-v(t)+f'\left(x^{\varphi}\left(t-1\right)\right)v\left(t-1\right)\tag{2.2}
\end{align*}
 with $v_{0}=\chi$, and let $x=x^{\varphi}$. As $\varphi\in\mathcal{W}$,
$\gamma_{j}=\Sigma_{i=0}^{j-1}\gamma\left(P_{Y}^{i}\left(\varphi\right)\right)$
is defined for all $j\geq1$, and $\gamma_{j}\rightarrow\infty$ as
$j\rightarrow\infty$. Then by formula \eqref{derivativse of iteratives of P_Y},
\begin{align*}
T_{P_{Y}^{j}\left(\varphi\right)}\mathcal{W}\ni DP_{Y}^{j}\left(\varphi\right)\left(c\dot{\varphi}+\chi\right) & =c\dot{x}_{\gamma_{j}}+v_{\gamma_{j}}-\frac{e^{*}\left(c\dot{x}_{\gamma_{j}}+v_{\gamma_{j}}\right)}{e^{*}\left(\dot{x}_{\gamma_{j}}\right)}\dot{x}_{\gamma_{j}}\\
 & =v_{\gamma_{j}}-\frac{e^{*}\left(v_{\gamma_{j}}\right)}{e^{*}\left(\dot{x}_{\gamma_{j}}\right)}\dot{x}_{\gamma_{j}}.
\end{align*}

An application of Lemma \ref{lem:technical result 2} to the equation
\eqref{var eq} and its strictly positive solution $v:\left[-1,\infty\right)\rightarrow\mathbb{R}$
gives constants $K>0$ and $t\geq1$ such that 
\[
\left\Vert v_{s-1}\right\Vert \leq K\left\Vert v_{s}\right\Vert \qquad\mbox{for all }s\geq t.
\]
Equation \eqref{var eq} with this estimate then gives a uniform bound
for the derivatives $\dot{v}_{\gamma_{j}}/\left\Vert v_{\gamma_{j}}\right\Vert $,
$j\geq1$. So by the Arzel\`{a}\textendash{}Ascoli Theorem, there
exists a subsequence 
\[
\left(\frac{v_{\gamma_{j_{n}}}}{\left\Vert v_{\gamma_{j_{n}}}\right\Vert }\right)_{n=0}^{\infty}
\]
converging to a strictly positive unit vector $\rho$ as $n\rightarrow\infty$.
As the $C$-norm and the $C^{1}$-norm are equivalent on $\mathcal{A}$,
the convergence $x_{\gamma_{j}}=P_{Y}^{j}\left(\varphi\right)\rightarrow r_{0}$
implies that $\dot{x}_{\gamma_{j}}\rightarrow\dot{r}_{0}$ as $j\rightarrow\infty.$
It follows that 
\[
\frac{1}{\left\Vert v_{\gamma_{j_{n}}}\right\Vert }DP_{Y}^{j_{n}}\left(\varphi\right)\left(c\dot{\varphi}+\chi\right)\in T_{P_{Y}^{j_{n}}\left(\varphi\right)}\mathcal{W}
\]
converges to the vector 
\[
\rho-\frac{e^{*}\left(\rho\right)}{e^{*}\left(\dot{r}_{0}\right)}\dot{r}_{0}\in T_{r_{0}}\mathcal{W}=\begin{cases}
C_{s}, & \mbox{if }\mathcal{O}_{r}\mbox{ is hyperbolic,}\\
C_{s}\oplus\mathbb{R}\xi, & \mbox{if }\mathcal{O}_{r}\mbox{ is nonhyperbolic.}
\end{cases}
\]
As $T_{r_{0}}\mathcal{W}\subseteq C_{\leq1}$ and $\dot{r}_{0}\in C_{\leq1}$,
this means that $C_{\leq1}$ has a strictly positive element $\rho$.
This is a contradiction since $\mathcal{O}_{r}$ has a Floquet multiplier
$\lambda_{1}>1$ and $C_{<\lambda_{1}}\cap V^{-1}\left(0\right)=\emptyset$
by \eqref{C_<lambda_1}. 
\end{proof}
Now we can verify a part of Theorem \ref{main_theorem_4}.(i).
\begin{prop}
\label{prop: smooth connecting sets}$C_{q}^{p}$ and $C_{k}^{p}$
are both two-dimensional $C^{1}$-submanifolds of $\mathcal{W}^{u}\left(\mathcal{O}_{p}\right)$.\end{prop}
\begin{proof}
Define $r$, $\mathcal{W}$ and $C_{r}^{p}$ as at the begining of
this subsection. 

1. As a first step we confirm that to all $\varphi\in C_{r}^{p}$,
one can give a subset $W_{\varphi}$ of $C_{r}^{p}$ so that $W_{\varphi}$
is a two-dimensional $C^{1}$-submanifold of $\mathcal{W}^{u}\left(\mathcal{O}_{p}\right)$
and contains $\varphi$. Let $\varphi\in C_{r}^{p}$. Choose $T\geq0$
such that $\psi=\Phi\left(T,\varphi\right)\in\mathcal{W}^{u}\left(\mathcal{O}_{p}\right)\cap\mathcal{W}$
and $\dot{\psi}\notin Y$. Propositions \ref{qvarphi dot near r_0}
and \ref{prop:trajectory of P_Y} guarantee that this is possible.
Consider the two-dimensional $C^{1}$-submanifold $\mathbb{R}\times\left(\mathcal{W}^{u}\left(\mathcal{O}_{p}\right)\cap\mathcal{W}\right)$
of $\mathbb{R}\times\mathcal{W}^{u}\left(\mathcal{O}_{p}\right)$
and the map 
\[
\Sigma:\mathbb{R}\times\left(\mathcal{W}^{u}\left(\mathcal{O}_{p}\right)\cap\mathcal{W}\right)\ni\left(t,\eta\right)\mapsto\Phi_{\mathcal{W}^{u}\left(\mathcal{O}_{p}\right)}\left(t,\eta\right)\in\mathcal{W}^{u}\left(\mathcal{O}_{p}\right).
\]
Proposition \ref{the flow is continuously differentiable} proves
that $\Sigma$ is $C^{1}$-smooth and gives formulas for its derivatives.
Note that the derivative of the map $\mathcal{W}^{u}\left(\mathcal{O}_{p}\right)\cap\mathcal{W}\ni\eta\mapsto\Phi_{\mathcal{W}^{u}\left(\mathcal{O}_{p}\right)}\left(-T,\eta\right)\in\mathcal{W}^{u}\left(\mathcal{O}_{p}\right)$
at $\psi$ is injective on $T_{\psi}\left(\mathcal{W}^{u}\left(\mathcal{O}_{p}\right)\cap\mathcal{W}\right)$
by Corollary \ref{injective derivative}. Also observe that $\dot{\psi}\notin Y$
implies that $\dot{\psi}\notin T_{\psi}\left(\mathcal{W}^{u}\left(\mathcal{O}_{p}\right)\cap\mathcal{W}\right)$.
Using these two properties and a reasoning analogous to the one applied
in Proposition \ref{prop:certain subsets are submanifolds}, it is
straightforward to show that $D\Sigma\left(-T,\psi\right)$ is injective
on $\mathbb{R}\times T_{\psi}\left(\mathcal{W}^{u}\left(\mathcal{O}_{p}\right)\cap\mathcal{W}\right)$.
Thus there exists an $\varepsilon>0$ by Proposition \ref{prop:smoothness of submanifolds}
such that the set 
\[
W_{\varphi}=\left\{ \Phi_{\mathcal{W}^{u}\left(\mathcal{O}_{p}\right)}\left(t,\eta\right):\, t\in\left(-T-\varepsilon,-T+\,\varepsilon\right),\,\eta\in\mathcal{W}^{u}\left(\mathcal{O}_{p}\right)\cap\mathcal{W}\cap B\left(\psi,\varepsilon\right)\right\} 
\]
 is a two-dimensional $C^{1}$-submanifold of $\mathcal{W}^{u}\left(\mathcal{O}_{p}\right)$.
It is clear that $\varphi\in W_{\varphi}$. The invariance of $C_{r}^{p}$
implies that $W_{\varphi}\subseteq C_{r}^{p}$.

2. To complete the proof, it suffices to exclude for all $\varphi\in C_{r}^{p}$
the existence of a sequence $\left(\varphi^{n}\right)_{n=0}^{\infty}$
in $ $$C_{r}^{p}$ so that $\varphi^{n}\notin W_{\varphi}$ for $n\geq0$
and $\varphi^{n}\rightarrow\varphi$ as $n\rightarrow\infty$. By
Proposition \ref{prop:Pi_2 on tangent vectors}, $D\pi_{2}\left(\varphi\right)=\pi_{2}$
is injective on the two-dimensional tangent space $T_{\varphi}W_{\varphi}$,
hence it defines an isomorphism from $T_{\varphi}W_{\varphi}$ onto
$\mathbb{R}^{2}$. Therefore there exists $\tilde{\varepsilon}>0$
such that the restriction of $\pi_{2}$ to $W_{\varphi}\cap B\left(\varphi,\tilde{\varepsilon}\right)$
is a diffeomorphism from $W_{\varphi}\cap B\left(\varphi,\tilde{\varepsilon}\right)$
onto an open set $U$ in $\mathbb{R}^{2}$. If a sequence $\left(\varphi^{n}\right)_{n=0}^{\infty}$
in $ $$C_{r}^{p}$ converges to $\varphi$ as $n\rightarrow\infty$,
then $\pi_{2}\varphi^{n}\rightarrow\pi_{2}\varphi$ as $n\rightarrow\infty$,
and $\pi_{2}\varphi^{n}\in U$ for all $n$ large enough. The injectivity
of $\pi_{2}$ on $\overline{S_{k}}$ verified in Proposition \ref{prop:Pi2 on S_k}
then implies that $\varphi^{n}\in W_{\varphi}$.
\end{proof}
It is worth noting that the second part of the above proof confirms
the following assertion. 
\begin{prop}
\label{a detail of the above proof}$\pi_{2}C_{q}^{p}$ and $\pi_{2}C_{k}^{p}$
are open subsets of $\mathbb{R}^{2}$.
\end{prop}
We know from Proposition \ref{prop:graph representation for the closure of S_k}
that there exist a projection $P_{2}$ from $C$ onto a two-dimensional
subspace $G_{2}$ of $C$ and a map $w_{k}:P_{2}\overline{S_{k}}\rightarrow P_{2}^{-1}\left(0\right)$
so that
\[
\overline{S_{k}}=\left\{ \chi+w_{k}\left(\chi\right):\,\chi\in P_{2}\overline{S_{k}}\right\} .
\]
Then 
\[
C_{q}^{p}=\left\{ \chi+w_{k}\left(\chi\right):\,\chi\in P_{2}C_{q}^{p}\right\} \quad\mbox{and}\quad C_{k}^{p}=\left\{ \chi+w_{k}\left(\chi\right):\,\chi\in P_{2}C_{k}^{p}\right\} .
\]
The next result implies that these representations of $C_{q}^{p}$
and $C_{k}^{p}$ are smooth. 
\begin{prop}
\label{prop: smooth representation for connecting sets}$P_{2}C_{q}^{p}$
and $P_{2}C_{k}^{p}$ are open subsets of $G_{2}$, and $w_{k}$ is
continuously differentiable on $P_{2}C_{q}^{p}\cup P_{2}C_{k}^{p}$.\end{prop}
\begin{proof}
The proof is based on the smoothness of $C_{q}^{p}$ and $C_{k}^{p}$
and applies an argument which is analogous to the one in the proof
of Theorem \ref{main_theorem_3}. 

Let $C_{r}^{p}$ be any of the sets $C_{q}^{p}$ and $C_{k}^{p}$.
Let $\chi\in P_{2}C_{r}^{p}$ be arbitrary, and choose $\varphi\in C_{r}^{p}$
so that $\chi=P_{2}\varphi$. As the restriction of $\pi_{2}$ to
$T_{\varphi}C_{r}^{p}$ is injective, $J_{2}$ is a linear isomorhism
and $P_{2}=J_{2}\circ\pi_{2}$, $DP_{2}\left(\varphi\right)=P_{2}$
defines an isomorphism from $T_{\varphi}C_{r}^{p}$ to $G_{2}$. The
inverse mapping theorem implies that an $\varepsilon>0$ can be given
such that $P_{2}$ maps $C_{r}^{p}\cap B\left(\varphi,\varepsilon\right)$
one-to-one onto an open neighborhood $U\subset P_{2}C_{r}^{p}$ of
$\chi$ in $G_{2}$, $P_{2}$ is invertible on $C_{r}^{p}\cap B\left(\varphi,\varepsilon\right)$,
and the inverse $\tilde{P}_{2}^{-1}$ of the map 
\[
C_{r}^{p}\cap B\left(\varphi,\varepsilon\right)\ni\varphi\mapsto P_{2}\varphi\in U
\]
 is $C^{1}$-smooth. As 
\[
w_{k}\left(\chi\right)=\left(\mbox{id}-P_{2}\right)\circ\left(P_{2}|_{\overline{S_{k}}}\right)^{-1}\left(\chi\right)=\left(\mbox{id}-P_{2}\right)\circ\tilde{P}_{2}^{-1}\left(\chi\right)\in P_{2}^{-1}\left(0\right)
\]
 for all $\chi\in U$, the restriction of $w_{k}$ to $U$ is $C^{1}$-smooth.
\end{proof}
\noindent ~

\noindent \begin{center}
\textit{5.4} \textit{$C_{q}^{p}$, $C_{k}^{p}$ and $S_{k}$ are homeomorphic
to $A^{\left(1,2\right)}$, and their closures are homeomorphic to
}$A^{\left[1,2\right]}$\smallskip{}

\par\end{center}

Recall that 
\[
A_{q}^{p}=\mbox{ext}\left(\pi_{2}\mathcal{O}_{p}\right)\cap\mbox{int}\left(\pi_{2}\mathcal{O}_{q}\right),\qquad A_{k}^{p}=\mbox{ext}\left(\pi_{2}\mathcal{O}_{k}\right)\cap\mbox{int}\left(\pi_{2}\mathcal{O}_{p}\right)
\]
 and 
\[
A_{k,q}=\mbox{ext}\left(\pi_{2}\mathcal{O}_{k}\right)\cap\mbox{int}\left(\pi_{2}\mathcal{O}_{q}\right).
\]
We have already deduced that $\pi_{2}C_{q}^{p}\subseteq A_{q}^{p}$
and $\pi_{2}C_{k}^{p}\subseteq A_{k}^{p}$. As a result, $\pi_{2}S_{k}\subseteq A_{k,q}$.
\begin{prop}
\label{prop:hoemomorphic to the open annulus} The map $\pi_{2}|_{\overline{S_{k}}}$
is a homeomorphism onto $\overline{A_{k,q}}$, furthermore $\pi_{2}C_{k}^{p}=A_{k}^{p}$,
$\pi_{2}C_{q}^{p}=A_{q}^{p}$ and $\pi_{2}S_{k}=A_{k,q}$. \end{prop}
\begin{proof}
First we show that $\pi_{2}C_{q}^{p}=A_{q}^{p}.$ By Proposition \ref{a detail of the above proof},
$\pi_{2}C_{q}^{p}$ is open in $A_{q}^{p}$. We claim that $\pi_{2}C_{q}^{p}$
is also closed in $A_{q}^{p}$. So assume that $\left(z_{n}\right)_{n=0}^{\infty}$
is a sequence in $\pi_{2}C_{q}^{p}$ and $z_{n}\rightarrow z\in A_{q}^{p}$
as $n\rightarrow\infty$. Let $\varphi_{n}=\pi_{2}^{-1}\left(z_{n}\right)\in C_{q}^{p}$,
$n\geq0$. By Proposition \ref{prop:inverse of Pi_2 is Lip-cont},
$\pi_{2}^{-1}$ is Lipschitz-continuous. Thus $\left\{ \varphi_{n}\right\} _{n=0}^{\infty}$
is a Cauchy-sequence in $C_{q}^{p}$ and a $\varphi\in\overline{C_{q}^{p}}$
can be given such that $\varphi_{n}\rightarrow\varphi$ as $n\rightarrow\infty$,
moreover, $\varphi=\pi_{2}^{-1}\left(z\right)$. It is clear that
$\varphi\notin\mathcal{O}_{p}$ and $\varphi\notin\mathcal{O}_{q}$
because then $z=\pi_{2}\varphi\notin A_{q}^{p}$. Thus $\varphi\in\overline{C_{q}^{p}}\setminus\left(\mathcal{O}_{p}\cup\mathcal{O}_{p}\right)=C_{q}^{p}$
(here we use Corollary \ref{cor: closures of connecting sets}) and
necessarily $z=\pi_{2}\varphi\in\pi_{2}C_{q}^{p}$. In consequence,
$\pi_{2}C_{q}^{p}=A_{q}^{p}.$

It is analogous to verify that $\pi_{2}C_{k}^{p}=A_{k}^{p}$. It follows
immediately that 
\[
\pi_{2}S_{k}=\pi_{2}\left(C_{k}^{p}\cup\mathcal{O}_{p}\cup C_{q}^{p}\right)=A_{k}^{p}\cup\pi_{2}\mathcal{O}_{p}\cup A_{q}^{p}=A_{k,q}
\]
 and 
\[
\pi_{2}\overline{S_{k}}=\pi_{2}\left(\mathcal{O}_{k}\cup S_{k}\cup\mathcal{O}_{q}\right)=\pi_{2}\mathcal{O}_{k}\cup A_{k,q}\cup\pi_{2}\mathcal{O}_{q}=\overline{A_{k,q}}.
\]

As both $\pi_{2}|_{\overline{S_{k}}}:\overline{S_{k}}\rightarrow\mathbb{R}^{2}$
and $\pi_{2}^{-1}:\pi_{2}\overline{S_{k}}\rightarrow C$ are continuous,
we obtain that $\pi_{2}|_{\overline{S_{k}}}$ defines a homeomorphism
from $\overline{S_{k}}$ onto $\overline{A_{k,q}}$. 
\end{proof}
As a consequence we obtain that $C_{q}^{p}$, $C_{k}^{p}$, and $S_{k}$
are homeomorphic to the open annulus 
\[
A^{\left(1,2\right)}=\left\{ u\in\mathbb{R}^{2}:\,1<\left|u\right|<2\right\} .
\]
 Since the above proposition implies that $\pi_{2}\overline{C_{k}^{p}}=\overline{A_{k}^{p}}$
and $\pi_{2}\overline{C_{q}^{p}}=\overline{A_{q}^{p}}$, we also deduce
that the closures $\overline{C_{q}^{p}}$, $\overline{C_{k}^{p}}$,
and $\overline{S_{k}}$ are homeomorphic to the closed annulus
\[
A^{\left[1,2\right]}=\left\{ u\in\mathbb{R}^{2}:\,1\leq\left|u\right|\leq2\right\} .
\]
 Note that we have proven all the statements of Theorem \ref{main_theorem_3}.(i)
regarding $C_{q}^{p}$ and $C_{k}^{p}$ (see propositions \ref{prop: smooth connecting sets},
\ref{prop: smooth representation for connecting sets} and \ref{prop:hoemomorphic to the open annulus}).
The smoothness of $S_{k}$ is considered in the next subsection.

~

\noindent \begin{center}
\textit{5.5 The smoothness of $S_{k}$, $\overline{C_{q}^{p}}$, $\overline{C_{k}^{p}}$
and $\overline{S_{k}}$}\smallskip{}

\par\end{center}

Now we can round up the proofs of Theorem \ref{main_theorem_4}.(i)
and (ii). Recall that 
\[
S_{k}=\left\{ \chi+w_{k}\left(\chi\right):\,\chi\in P_{2}S_{k}\right\} ,\qquad P_{2}S_{k}=P_{2}C_{k}^{p}\cup P_{2}\mathcal{O}_{p}\cup P_{2}C_{q}^{p}
\]
and $w_{k}$ is continuously differentiable on the set $P_{2}C_{k}^{p}\cup P_{2}C_{q}^{p}$.
Hence the smoothness of this representation for $S_{k}$ is proved
by showing that $P_{2}S_{k}$ is open in $G_{2}$ and $w_{k}$ is
smooth at the points of $P_{2}\mathcal{O}_{p}$. It follows at once
that $S_{k}$ is a two-dimensional $C^{1}$-submanifold of $C$. Since
$S_{k}$ is a subset of the three-dimensional $C^{1}$-submanifold
$\mathcal{W}^{u}\left(\mathcal{O}_{p}\right)$, it is obvious that
$S_{k}$ is also a $C^{1}$-submanifold of $\mathcal{W}^{u}\left(\mathcal{O}_{p}\right)$.
We likewise verify that all points of $P_{2}\mathcal{O}_{k}\cup P_{2}\mathcal{O}_{q}$
have open neighborhoods on which $w_{k}$ can be extended to $C^{1}$-functions.
As $P_{2}\mathcal{O}_{k}\cup P_{2}\mathcal{O}_{q}$ is the boundary
of $P_{2}\overline{S_{k}}$, this result shows that $\overline{S_{k}}$
has a smooth representation with boundary, and thus $\overline{S_{k}}$
is a two-dimensional $C^{1}$-submanifold of the phase space $C$
with boundary. Similar reasonings yield the analogous results for
$\overline{C_{q}^{p}}$ and $\overline{C_{k}^{p}}$. 

Let $r:\mathbb{R}\rightarrow\mathbb{R}$ be any of the periodic solutions
$x^{k}$, $p$ or $q$ shifted in time so that $r\left(0\right)=\xi_{k}$
and $\dot{r}\left(0\right)>0$. As $\xi_{k}$ belongs to the ranges
of $x^{k}$, $p$ or $q$, and $\xi_{k}$ is not an extremum of them,
the monotonicity property of periodic solutions in Proposition \ref{pro: monotone_type}
implies that this choice of $r$ is possible. Let $\omega>0$ denote
the minimal period of $r$. By Eq.~\eqref{eq:eq_general},
\[
f\left(r\left(-1\right)\right)=\dot{r}\left(0\right)+r\left(0\right)>\xi_{k}=f\left(\xi_{k}\right).
\]
As $f$ strictly increases, this means that $r\left(-1\right)>\xi_{k}$.
Conversely, if there was $t_{*}\in\left(0,\omega\right)$ such that
$r\left(t_{*}\right)=\xi_{k}$ and $r\left(t_{*}-1\right)>\xi_{k}$,
then 
\[
\dot{r}\left(t_{*}\right)=-r\left(t_{*}\right)+f\left(r\left(t_{*}-1\right)\right)>-\xi_{k}+f\left(\xi_{k}\right)=0
\]
would follow, which would contradict Proposition \ref{pro: monotone_type}.
Therefore the half line $L_{k}=\left\{ \left(\xi_{k},x_{2}\right)\in\mathbb{R}^{2}:\, x_{2}>\xi_{k}\right\} $
and $\pi_{2}\mathcal{O}_{r}=\left\{ \pi_{2}r_{t}:\, t\in\left[0,\omega\right)\right\} $
have exactly one point in common: $\left(r\left(0\right),r\left(-1\right)\right)=\left(\xi_{k},r\left(-1\right)\right)$.
See Fig. 7.

Choose $s_{k},s_{p},s_{q}>\xi_{k}$ so that 
\[
\left\{ \left(\xi_{k},s_{k}\right)\right\} =L_{k}\cap\pi_{2}\mathcal{O}_{k},\quad\left\{ \left(\xi_{k},s_{p}\right)\right\} =L_{k}\cap\pi_{2}\mathcal{O}_{p}
\]
 and 
\[
\left\{ \left(\xi_{k},s_{q}\right)\right\} =L_{k}\cap\pi_{2}\mathcal{O}_{q}.
\]
As $s$ increases, $\left(\xi_{k},\infty\right)\ni s\mapsto\left(\xi_{k},s\right)\in\mathbb{R}^{2}$
first intersects $\pi_{2}\mathcal{O}_{k}$, then $\pi_{2}\mathcal{O}_{p}$
and finally $\pi_{2}\mathcal{O}_{q}$ because 
\[
\left(\xi_{k},s\right)\rightarrow\pi_{2}\hat{\xi}_{k}=\left(\xi_{k},\xi_{k}\right)\in\mbox{int}\left(\pi_{2}\mathcal{O}_{k}\right)\mbox{ whenever }s\rightarrow\xi_{k}+,
\]
 $\pi_{2}\mathcal{O}_{k}\in\mbox{int}\left(\pi_{2}\mathcal{O}_{p}\right)$
and $\pi_{2}\mathcal{O}_{p}\in\mbox{int}\left(\pi_{2}\mathcal{O}_{q}\right)$.
So $\xi_{k}<s_{k}<s_{p}<s_{q}$, as it is shown by Fig. 7.

\begin{center}
\begin{figure}[h]
\begin{centering}
\includegraphics[scale=0.07]{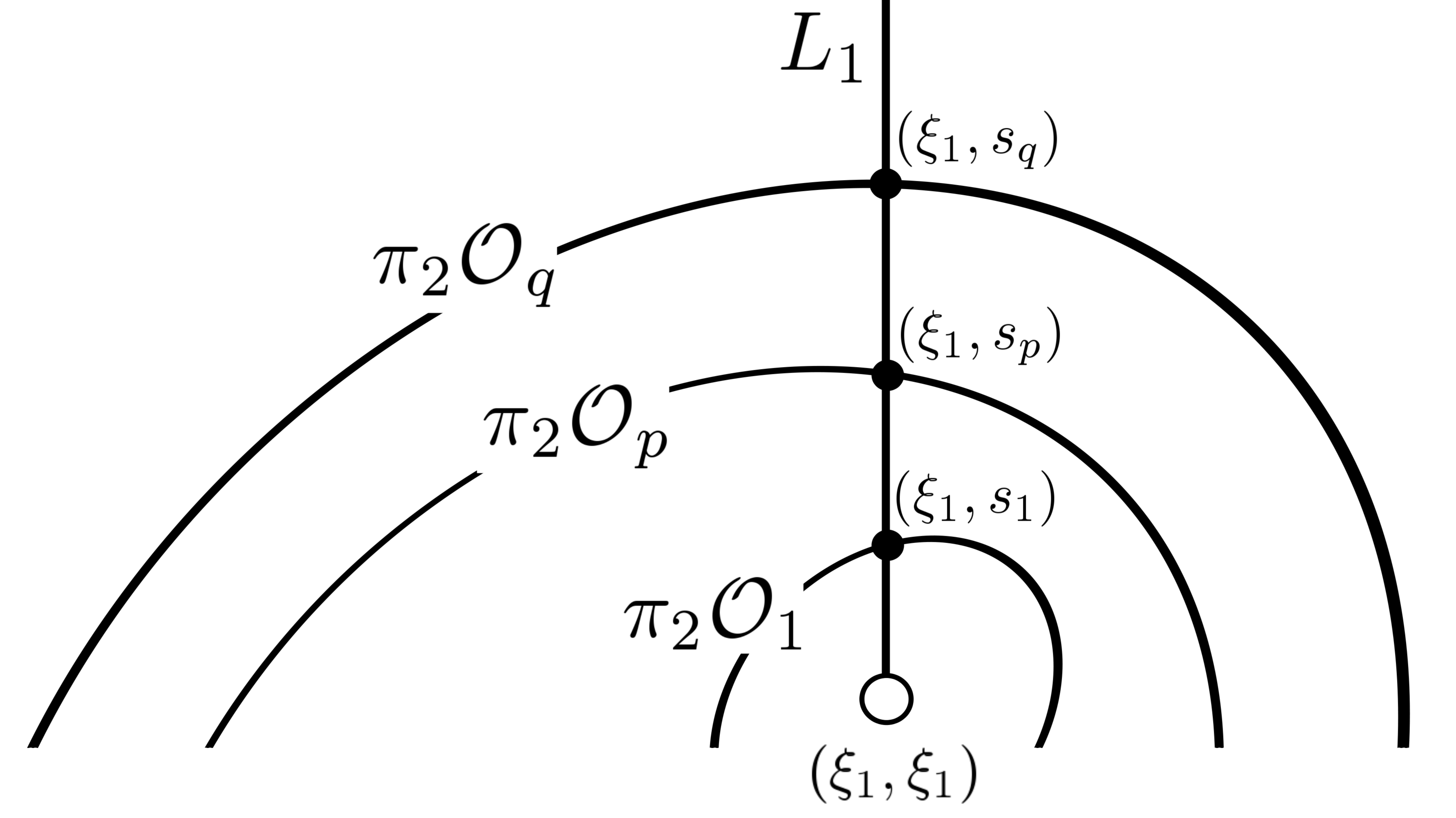} 
\par\end{centering}

\caption{The definition of $L_{1}$, $s_{1}$, $s_{p}$ and $s_{q}$ in the
case $k=1$. }
\end{figure}

\par\end{center}

Consider the curve 
\[
h:\left[s_{k},s_{q}\right]\ni s\mapsto\pi_{2}^{-1}\left(\xi_{k},s\right)\in C.
\]
 Then $h$ is Lipschitz-continuous and injective. By Proposition \ref{prop:hoemomorphic to the open annulus},
$h\left(\left[s_{k},s_{q}\right]\right)\subset\overline{S_{k}}$.
In detail, 
\[
h\left(s_{k}\right)\in\mathcal{O}_{k},\quad h\left(\left(s_{k},s_{p}\right)\right)\subset C_{k}^{p},\quad h\left(s_{p}\right)\in\mathcal{O}_{p},\quad h\left(\left(s_{p},s_{q}\right)\right)\subset C_{q}^{p},\quad\mbox{and}\quad h\left(s_{q}\right)\in\mathcal{O}_{q}.
\]
According to the next result, $h$ is $C^{1}$-smooth on $\left(s_{1},s_{q}\right)\setminus\left\{ s_{p}\right\} $.
 
\begin{prop}
$\pi_{2}^{-1}|_{\pi_{2}\left(C_{q}^{p}\cup C_{k}^{p}\right)}$ is
$C^{1}$-smooth.\end{prop}
\begin{proof}
We know from Proposition \ref{a detail of the above proof} that $\pi_{2}\left(C_{q}^{p}\cup C_{k}^{p}\right)$
is open in $\mathbb{R}^{2}$. 

For all $x\in\pi_{2}\left(C_{q}^{p}\cup C_{k}^{p}\right)$, the graph
representation of $C_{q}^{p}\cup C_{k}^{p}$ and the definition of
$P_{2}$ together give that 
\begin{align*}
C_{q}^{p}\cup C_{k}^{p}\ni\pi_{2}^{-1}\left(x\right) & =P_{2}\left(\pi_{2}^{-1}\left(x\right)\right)+w_{k}\left(P_{2}\left(\pi_{2}^{-1}\left(x\right)\right)\right)\\
 & =J_{2}\left(\pi_{2}\left(\pi_{2}^{-1}\left(x\right)\right)\right)+w_{k}\left(J_{2}\left(\pi_{2}\left(\pi_{2}^{-1}\left(x\right)\right)\right)\right)\\
 & =J_{2}\left(x\right)+w_{k}\left(J_{2}\left(x\right)\right).
\end{align*}
As $J_{2}$ defines a linear isomorphism from $\mathbb{R}^{2}$ to
$G_{2}$, it is continuously differentiable. In addition, $J_{2}\left(\pi_{2}\left(C_{q}^{p}\cup C_{k}^{p}\right)\right)=P_{2}\left(C_{q}^{p}\cup C_{k}^{p}\right)$,
and $w_{k}$ is continuously differentiable on the open subset $P_{2}\left(C_{q}^{p}\cup C_{k}^{p}\right)$
of $G_{2}$ by Proposition \ref{prop: smooth representation for connecting sets}.
Hence the statement follows.
\end{proof}
As a next step, we show the smoothness of $h$ at points $s_{k},s_{p}$
and $s_{q}$.$ $ We will need the following technical result, which
is part of Proposition 8.5 in  \cite{Krisztin-Walther-Wu}.
\begin{prop}
\label{prop:technical result 2}~

(i) Let $v:\mathbb{R}\rightarrow\mathbb{R}$ be a solution of Eq.~\eqref{eq:lin.eq.}
with $v_{0}\neq\hat{0}$. If $V\left(v_{t}\right)=2$ for all $t\in\mathbb{R}$,
then $v_{0}\in C_{r_{M}<}\cap C_{\leq1}$.

(ii) For every $\varphi\in C_{r_{M}<}\cap C_{\leq1}\setminus\left\{ \hat{0}\right\} $,
there is a solution $v:\mathbb{R}\rightarrow\mathbb{R}$ of Eq.~\eqref{eq:lin.eq.}
so that $v_{0}=\varphi$ and $V\left(v_{t}\right)=2$ for all $t\in\mathbb{R}$.
\end{prop}
\vspace{0bp}

\begin{prop}
\label{prop:smoothness of h at three points}Let $*\in\left\{ k,p,q\right\} $
and set $r:\mathbb{R}\rightarrow\mathbb{R}$ to be the periodic solution
of Eq.~\eqref{eq:eq_general} with $\pi_{2}r_{0}=\left(\xi_{k},s_{*}\right)$. 

(i) There exists a unique continuously differentiable function $z=z^{*}:\mathbb{R}\rightarrow\mathbb{R}$
satisfying 
\begin{equation}
\begin{cases}
\dot{z}\left(t\right)=-z\left(t\right)+f'\left(r\left(t-1\right)\right)z\left(t-1\right), & t\in\mathbb{R},\\
z\left(-1\right)=1,\, z\left(0\right)=0,\\
V\left(z_{t}\right)=2, & t\in\mathbb{R}.
\end{cases}\label{eq:eq for the derivative}
\end{equation}

(ii) For every $\varepsilon>0$, there exists $\delta>0$ so that
for all $\chi\in\left[s_{k},s_{q}\right]$, $\nu\in\left[s_{k},s_{q}\right]$
with $\left|\chi-s_{*}\right|<\delta$, $\left|\nu-s_{*}\right|<\delta$
and $\chi\neq\nu$, 
\[
\left\Vert \frac{h\left(\chi\right)-h\left(\nu\right)}{\chi-\nu}-z_{0}\right\Vert <\varepsilon.
\]

(iii) $z_{0}$ and $\dot{r}_{0}$ are linearly independent.\end{prop}
\begin{proof}
1. We prove that for all sequences $\left(\chi^{n}\right)_{n=0}^{\infty}$,
$\left(\nu^{n}\right)_{n=0}^{\infty}$ in $\left[s_{k},s_{q}\right]$
with $\chi^{n}\neq\nu^{n}$ for all $n\geq0$ and $\chi^{n}\rightarrow s_{*}$,
$\nu^{n}\rightarrow s_{*}$ as $n\rightarrow\infty$, there exist
a strictly increasing sequence $\left(n_{l}\right)_{l=0}^{\infty}$
and a continuously differentiable function $z=z^{*}:\mathbb{R}\rightarrow\mathbb{R}$
so that $z$ is a solution of the equation in \eqref{eq:eq for the derivative},
and 
\[
\lim_{l\rightarrow\infty}\frac{h\left(\chi^{n_{l}}\right)-h\left(\nu^{n_{l}}\right)}{\chi^{n_{l}}-\nu^{n_{l}}}=z_{0}.
\]

Consider the solutions $x^{n}:\mathbb{R}\rightarrow\mathbb{R}$ and
$y^{n}:\mathbb{R}\rightarrow\mathbb{R}$ of Eq.~\eqref{eq:eq_general}
with $x_{0}^{n}=h\left(\chi^{n}\right)$ and $y_{0}^{n}=h\left(\nu^{n}\right)$
for all indices $n\geq0$. Then $x^{n}\left(-1\right)=\chi^{n}$,
$y^{n}\left(-1\right)=\nu^{n}$ and $x^{n}\left(0\right)=y^{n}\left(0\right)=\xi_{k}$
for all $n\geq0$, moreover $x_{t}^{n}\in\overline{S_{k}}$ and $y_{t}^{n}\in\overline{S_{k}}$
for all $n\geq0$ and $t\in\mathbb{R}$. 

Introduce the functions 
\[
z^{n}=\frac{x^{n}-y^{n}}{\chi^{n}-\nu^{n}},\qquad n\geq0.
\]
It is clear that $z^{n}\left(0\right)=0$ and $z^{n}\left(-1\right)=1$
for all $n\geq0$. By Proposition \ref{prop:V on S_k}, $V\left(z_{t}^{n}\right)=2$
for all $n\geq0$ and $t\in\mathbb{R}$. In addition, $z^{n}$, $n\geq0$,
satisfies the equation 
\[
\dot{z}^{n}\left(t\right)=-z^{n}\left(t\right)+b^{n}\left(t\right)z^{n}\left(t-1\right)
\]
on $\mathbb{R}$, where the coefficient functions $b^{n}$ are defined
as 
\[
b^{n}:\mathbb{R}\ni t\mapsto\int_{0}^{1}f'\left(sx^{n}\left(t-1\right)+\left(1-s\right)y^{n}\left(t-1\right)\right)\mbox{d}s\in\left(0,\infty\right),\quad n\geq0.
\]

Since $\chi^{n}\rightarrow s_{*}$ and $\nu^{n}\rightarrow s_{*}$
as $n\rightarrow\infty$, $x_{0}^{n}\rightarrow r_{0}$ and $y_{0}^{n}\rightarrow r_{0}$
as $n\rightarrow\infty$. It follows that $b^{n}\rightarrow b$ as
$n\rightarrow\infty$ uniformly on compact subsets of $\mathbb{R}$,
where 
\[
b:\mathbb{R}\ni t\mapsto f'\left(r\left(t-1\right)\right)\in\left(0,\infty\right).
\]
As the global attractor is bounded, there are constants $b_{1}>b_{0}>0$
so that $b_{0}<b^{n}\left(t\right)<b_{1}$ for all $n\geq0$ and $t\in\mathbb{R}$.
Thus Lemma \ref{lem:technical result} ensures the existence of a
continuously differentiable function $z:\mathbb{R}\rightarrow\mathbb{R}$
and a subsequence $\left(z^{n_{l}}\right)_{l=0}^{\infty}$ of $\left(z^{n}\right)_{n=0}^{\infty}$
such that $z^{n_{l}}\rightarrow z$ and $\dot{z}^{n_{l}}\rightarrow\dot{z}$
as $l\rightarrow\infty$ uniformly on compact subsets of $\mathbb{R}$,
and $z$ is a solution of the equation in \eqref{eq:eq for the derivative}. 

It is obvious that $z\left(0\right)=0$ and $z\left(-1\right)=1$.

By the first part of Lemma \ref{lem: continuity_of_V}, $V\left(z_{t}\right)\leq2$
for all real $t$. Suppose $V\left(z_{t^{*}}\right)=0$ for some $t^{*}\in\mathbb{R}$.
Then $V\left(z_{t}\right)=0$ for all $t>t^{*}$ and $V\left(z_{t^{*}+3}\right)\in R$
by Lemma \ref{lem:4_properties_of_V}. The $C^{1}$-convergence of
$z^{n_{l}}$ to $z$ and the second part of Lemma \ref{lem: continuity_of_V}
then imply that $V\left(z_{t^{*}+3}^{n_{l}}\right)=0$ for all sufficiently
large index $l$, which is contradiction. So $V\left(z_{t}\right)=2$
for all real $t$.

2. Suppose  that $\hat{z}:\mathbb{R}\rightarrow\mathbb{R}$ is also
a continuously differentiable function satisfying \eqref{eq:eq for the derivative},
and $z\neq\hat{z}$. Then  Proposition \ref{prop:uniqueness of solutions}
yields that $z_{0}\neq\hat{z}_{0}$. Function $d=z-\hat{z}$ is a
solution of 
\[
\begin{cases}
\dot{d}\left(t\right)=-d\left(t\right)+f'\left(r\left(t-1\right)\right)d\left(t-1\right), & t\in\mathbb{R},\\
d\left(-1\right)=d\left(0\right)=0.
\end{cases}
\]
Since $z_{0},\hat{z}_{0}\in C_{r_{M}<}\cap C_{\leq1}$ by Proposition
\ref{prop:technical result 2} (i), $d_{0}\in C_{r_{M}<}\cap C_{\leq1}\setminus\left\{ \hat{0}\right\} $.
Then it follows from Proposition \ref{prop:technical result 2} (ii)
that $V\left(d_{t}\right)=2$ for all $t\in\mathbb{R}$. So $d_{0}\in R$
by Lemma \ref{lem:4_properties_of_V}.(iii), which is impossible as
$d\left(-1\right)=d\left(0\right)=0.$ 

These results imply both (i) and (ii).

3. Solution $r$ has been defined to be a time translate of $x^{k}$,
$p$ or $q$ with $r\left(0\right)=\xi_{k}$. Hence $\xi_{k}$ is
not an extremum of $r$, and thus $\dot{r}\left(0\right)\neq0$ by
Proposition \ref{pro: monotone_type}. Consequently, $z_{0}\notin\mathbb{R}\dot{r}_{0}\setminus\left\{ \hat{0}\right\} $,
and $z_{0}$ and $\dot{r}_{0}$ are linearly independent.\end{proof}
\begin{cor}
Function $h$ is $C^{1}$-smooth on $\left[s_{k},s_{q}\right]$.
\end{cor}
We extend the definition of $h$ to the half line $\left(\xi_{k},\infty\right)$:
we define $\hat{h}:\left(\xi_{k},\infty\right)\rightarrow C$ as $\hat{h}\left(s\right)=h\left(s\right)$
for $s\in\left[s_{k},s_{q}\right]$, 
\[
\hat{h}\left(s\right)=h\left(s_{k}\right)+\left(s-s_{k}\right)z_{0}^{k}\quad\mbox{for }s\in\left(\xi_{k},s_{k}\right)
\]
and 
\[
\hat{h}\left(s\right)=h\left(s_{q}\right)+\left(s-s_{q}\right)z_{0}^{q}\quad\mbox{for }s>s_{q},
\]
where $z_{0}^{k}$ and $z_{0}^{q}$ are given by Proposition \ref{prop:smoothness of h at three points}.
Then $\hat{h}$ is $C^{1}$-smooth with $\hat{h}'\left(s_{k}\right)=z_{0}^{k}$,
$\hat{h}'\left(s_{p}\right)=z_{0}^{p}$, and $\hat{h}'\left(s_{q}\right)=z_{0}^{q}$.
According to the choice of $s_{k}<s_{p}<s_{q}$ and Proposition \ref{prop:hoemomorphic to the open annulus},
\begin{equation}
\hat{h}\left(s_{k}\right)\in\mathcal{O}_{k},\quad\hat{h}\left(\left(s_{k},s_{p}\right)\right)\subset C_{k}^{p},\quad\hat{h}\left(s_{p}\right)\in\mathcal{O}_{p},\quad\hat{h}\left(\left(s_{p},s_{q}\right)\right)\subset C_{q}^{p}\mbox{ and }\hat{h}\left(s_{q}\right)\in\mathcal{O}_{q}.\label{eq:properties of hat h}
\end{equation}
Observe that $\pi_{2}\hat{h}\left(s\right)=\left(\xi_{k},s\right)$
for all $s>\xi_{k}$, hence the map $\left(\xi_{k},\infty\right)\ni s\mapsto\pi_{2}\hat{h}\left(s\right)\in\mathbb{R}^{2}$
is injective on $\left(\xi_{k},\infty\right)$ and has range in $L_{k}=\left\{ \left(\xi_{k},x_{2}\right)\in\mathbb{R}^{2}:\, x_{2}>\xi_{k}\right\} $.
So it follows from $\pi_{2}\overline{S_{k}}=\overline{A_{k,q}}$ that
\begin{equation}
\hat{h}\left(\left(\xi_{k},s_{k}\right)\cup\left(s_{q},\infty\right)\right)\cap\overline{S_{k}}=\emptyset.\label{properties of hat h 2}
\end{equation}

As $J_{2}:\mathbb{R}^{2}\rightarrow G_{2}$ is a linear isomorphism
and $P_{2}=J_{2}\circ\pi_{2}$, Proposition \ref{prop:hoemomorphic to the open annulus}
shows that 
\[
P_{2}C_{k}^{p}=\mbox{ext}\left(P_{2}\mathcal{O}_{k}\right)\cap\mbox{int}\left(P_{2}\mathcal{O}_{p}\right),\quad P_{2}C_{q}^{p}=\mbox{ext}\left(P_{2}\mathcal{O}_{p}\right)\cap\mbox{int}\left(P_{2}\mathcal{O}_{q}\right),
\]
\[
P_{2}S_{k}=\mbox{ext}\left(P_{2}\mathcal{O}_{k}\right)\cap\mbox{int}\left(P_{2}\mathcal{O}_{q}\right)
\]
 and 
\begin{equation}
P_{2}\overline{S_{k}}=P_{2}\mathcal{O}_{k}\cup\left(\mbox{ext}\left(P_{2}\mathcal{O}_{k}\right)\cap\mbox{int}\left(P_{2}\mathcal{O}_{q}\right)\right)\cup P_{2}\mathcal{O}_{q}.\label{P_2 (closure of S_k)}
\end{equation}

As $P_{2}\mathcal{O}_{k}$ and $P_{2}\mathcal{O}_{q}$ are the images
of simple closed $C^{1}$-curves, the boundary $P_{2}\mathcal{O}_{k}\cup P_{2}\mathcal{O}_{q}$
of the domain $P_{2}\overline{S_{k}}$ of $w_{k}$ is a one-dimensional
$C^{1}$-submanifold of $G_{2}$. The next result shows that $w_{k}$
is continuously differentiable at the points of $P_{2}\mathcal{O}_{p}$,
and it is smooth at the points of $P_{2}\mathcal{O}_{k}\cup P_{2}\mathcal{O}_{q}$
in the sense that $w_{k}$ can be extended to continuously differentiable
functions on open neighborhoods of the boundary points. 
\begin{prop}
\label{extension of w_k}~

(i) To each $\varphi\in\mathcal{O}_{k}\cup\mathcal{O}_{q}$ there
corresponds an open neighborhood $U$ of $P_{2}\varphi$ in $G_{2}$
and a continuously differentiable map $w_{k}^{e}:U\rightarrow P_{2}^{-1}\left(0\right)$
such that 
\begin{equation}
w_{k}^{e}|_{U\cap P_{2}\overline{S_{k}}}=w_{k}|_{U\cap P_{2}\overline{S_{k}}},\label{w_k and w_k^e coincide}
\end{equation}
and $U\setminus\left\{ P_{2}x_{t}^{\varphi}:\, t\in\mathbb{R}\right\} $
is the union of open connected disjoint subsets $U^{+}$ and $U^{-}$
with the following property:

$U^{-}\cap P_{2}\overline{S_{k}}=\emptyset$ and $U^{+}\subset P_{2}C_{k}^{p}$
if $\varphi\in\mathcal{O}_{k}$,

$U^{-}\subset P_{2}C_{q}^{p}$ and $U^{+}\cap P_{2}\overline{S_{k}}=\emptyset$
if $\varphi\in\mathcal{O}_{q}$.

(ii) The map $w_{k}$ is continuously differentiable at the points
of $P_{2}\mathcal{O}_{p}$. All $\varphi\in\mathcal{O}_{p}$ has an
open neighborhood $U$ of $P_{2}\varphi$ in $G_{2}$ such that $U\setminus P_{2}\mathcal{O}_{p}$
is the union of open connected disjoint subsets $U^{+}$ and $U^{-}$
with $U^{-}\subset P_{2}C_{k}^{p}$ and $U^{+}\subset P_{2}C_{q}^{p}$.\end{prop}
\begin{proof}
The proof below verifies assertions (i) and (ii) simultaneously.

1. Let $r:\mathbb{R}\rightarrow\mathbb{R}$ be one of the periodic
solutions $x^{k}$, $p$ or $q$ shifted in time so that $r\left(0\right)=\xi_{k}$
and $\dot{r}\left(0\right)>0$ (that is $\pi_{2}r_{0}\in L_{k}$),
and fix $*\in\left\{ k,p,q\right\} $ accordingly. Set $s_{*}=r\left(-1\right)$.
Let $\varphi\in\mathcal{O}_{r}=\left\{ r_{t}:\, t\in\mathbb{R}\right\} $
and choose $T>1$ so that $\varphi=\Phi\left(T,r_{0}\right)$. For
all $0<\varepsilon<\min\left\{ T-1,s_{k}-\xi_{k}\right\} $, the map
\[
a:\left(-\varepsilon,\varepsilon\right)\times\left(-\varepsilon,\varepsilon\right)\ni\left(t,s\right)\mapsto\Phi\left(T+t,\hat{h}\left(s_{*}+s\right)\right)\in C
\]
is $C^{1}$-smooth with 
\[
Da\left(0,0\right)\mathbb{R}^{2}=\mathbb{R}\dot{\varphi}\oplus\mathbb{R}D_{2}\Phi\left(T,r_{0}\right)z_{0}^{*},
\]
 where $z^{*}:\mathbb{R}\rightarrow\mathbb{R}$ is the solution of
\eqref{eq:eq for the derivative} given by Proposition \ref{prop:smoothness of h at three points}.
The vectors $\dot{\varphi}=D_{2}\Phi\left(T,r_{0}\right)\dot{r}_{0}$
and $D_{2}\Phi\left(T,r_{0}\right)z_{0}^{*}$ are linearly independent
because $D_{2}\Phi\left(T,r_{0}\right)$ is injective, and $\dot{r}_{0}$
and $z_{0}^{*}$ are linearly independent by Proposition \ref{prop:smoothness of h at three points}
(iii). 

Therefore Proposition \ref{prop:smoothness of submanifolds} implies
that for all small $\varepsilon>0$, the set $a\left(\left(-\varepsilon,\varepsilon\right)\times\left(-\varepsilon,\varepsilon\right)\right)$
is a two-dimensional $C^{1}$-submanifold of $C$ with 
\[
T_{\varphi}a\left(\left(-\varepsilon,\varepsilon\right)\times\left(-\varepsilon,\varepsilon\right)\right)=Da\left(0,0\right)\mathbb{R}^{2}.
\]
Then $a\left(\left(-\varepsilon,\varepsilon\right)\times\left(-\varepsilon,0\right)\right)$
and $a\left(\left(-\varepsilon,\varepsilon\right)\times\left(0,\varepsilon\right)\right)$
are also two-dimensional $C^{1}$-submanifolds of $C$. 

2. Set $E_{1}=\mathbb{R}\dot{\varphi}\oplus\mathbb{R}D_{2}\Phi\left(T,r_{0}\right)z_{0}^{*}$
and let $E_{2}$ be a closed complement of $E_{1}$ in $C$. We claim
that for small $\varepsilon>0$, there exist an open neighborhood
$N_{\varepsilon}$ of $\hat{0}$ in $E_{1}$ and a continuously differentiable
function $b:N_{\varepsilon}\rightarrow E_{2}$ so that $b\left(\hat{0}\right)=0$,
$Db\left(\hat{0}\right)=0$ and $a\left(\left(-\varepsilon,\varepsilon\right)\times\left(-\varepsilon,\varepsilon\right)\right)$
is the shifted graph of $b$: 
\[
a\left(\left(-\varepsilon,\varepsilon\right)\times\left(-\varepsilon,\varepsilon\right)\right)=\varphi+\left\{ \chi+b\left(\chi\right):\,\chi\in N_{\varepsilon}\right\} .
\]

Let $\mbox{Pr}_{E_{1}}$ denote the projection of $C$ onto $E_{1}$
along $E_{2}$, and define $j:C\rightarrow C$ by $j\left(\chi\right)=\chi-\varphi$
for all $\chi\in C$. Then 
\[
D\left(\mbox{Pr}_{E_{1}}\circ j\circ a\right)\left(0,0\right)\mathbb{R}^{2}=\mbox{Pr}_{E_{1}}\circ Da\left(0,0\right)\mathbb{R}^{2}=E_{1}.
\]
Hence the inverse function theorem guarantees that $\mbox{Pr}_{E_{1}}\circ j\circ a$
is a local $C^{1}$-diffeomorphism, i.e.~for for small $\varepsilon>0$,
$\mbox{Pr}_{E_{1}}\circ j\circ a$ maps $\left(-\varepsilon,\varepsilon\right)\times\left(-\varepsilon,\varepsilon\right)$
injectively onto an open neighborhood $N_{\varepsilon}$ of $\hat{0}$
in $E_{1}$, and the inverse $\left(\mbox{Pr}_{E_{1}}\circ j\circ a\right)^{-1}$
of $\left(-\varepsilon,\varepsilon\right)\times\left(-\varepsilon,\varepsilon\right)\ni\left(t,s\right)\mapsto\mbox{Pr}_{E_{1}}\circ j\circ a\left(t,s\right)\in N_{\varepsilon}$
is $C^{1}$-smooth. In consequence, $\mbox{Pr}_{E_{1}}$ maps $j\circ a\left(\left(-\varepsilon,\varepsilon\right)\times\left(-\varepsilon,\varepsilon\right)\right)$
onto $N_{\varepsilon}$ injectively, and there exists a map $b:N_{\varepsilon}\rightarrow E_{2}$
so that $b\left(\hat{0}\right)=0$ and 
\[
j\circ a\left(\left(-\varepsilon,\varepsilon\right)\times\left(-\varepsilon,\varepsilon\right)\right)=\left\{ \chi+b\left(\chi\right):\,\chi\in N_{\varepsilon}\right\} .
\]
The smoothness of $b$ follows from 
\[
b=\left(\mbox{id}-\mbox{Pr}_{E_{1}}\right)\circ j\circ a\circ\left(\mbox{Pr}_{E_{1}}\circ j\circ a\right)^{-1}.
\]
$Db\left(\hat{0}\right)=0$ because $Da\left(0,0\right)\mathbb{R}^{2}=E_{1}$.

3. Next we show that the continuously differentiable map 
\[
c:E_{1}\supset N_{\varepsilon}\ni\chi\mapsto P_{2}\left(\varphi+\chi+b\left(\chi\right)\right)\in G_{2}
\]
 is a local $C^{1}$-diffeomorphism. 

Note that $Dc\left(\hat{0}\right)\chi=P_{2}\chi$ for all $\chi\in E_{1}$.
So it suffices to confirm that $P_{2}|_{E_{1}}$ is injective. $E_{1}$
is spanned by the derivatives $D\gamma\left(0\right)1$ of the curves
\[
\gamma:\left(-1,1\right)\ni s\mapsto a\left(c_{1}s,c_{2}s\right)\in C,
\]
where $\left(c_{1},c_{2}\right)\in\mathbb{R}^{2}$. From \eqref{eq:properties of hat h}
and the invariance of $\overline{S_{k}}$ it follows that if $s_{*}+c_{2}s\in\left[s_{k},s_{q}\right]$,
then $\gamma\left(s\right)\in\overline{S_{k}}$. Proposition \ref{prop:Pi_2 on tangent vectors}
gives that $\pi_{2}\gamma'\left(0\right)\neq\left(0,0\right)$ if
$\gamma'\left(0\right)\neq\hat{0}$. Thus $\pi_{2}|_{E_{1}}$ is injective.
As $J_{2}$ is a linear isomorphism, $P_{2}=J_{2}\circ\pi_{2}$ is
also injective on $E_{1}$.

In consequence, a positive constant $\varepsilon_{0}$ can be given
such that $c$ is a $C^{1}$-diffeomorphism from $N_{\varepsilon_{0}}$
onto an open neighborhood $U$ of $P_{2}\varphi$ in $G_{2}$. Define
$c^{-1}$ to be the inverse of $N_{\varepsilon_{0}}\ni\chi\mapsto c\left(\chi\right)\in U$. 

Constant $\varepsilon_{0}$ can be chosen so that $\varepsilon_{0}<\min\left\{ T-1,s_{k}-\xi_{k},s_{p}-s_{k},s_{q}-s_{p}\right\} $
also holds.

4. Notice that 
\[
U=P_{2}a\left(\left(-\varepsilon_{0},\varepsilon_{0}\right)\times\left(-\varepsilon_{0},\varepsilon_{0}\right)\right),
\]
and set 
\begin{align*}
U^{-}= & P_{2}a\left(\left(-\varepsilon_{0},\varepsilon_{0}\right)\times\left(-\varepsilon_{0},0\right)\right),\\
U^{0}= & P_{2}a\left(\left(-\varepsilon_{0},\varepsilon_{0}\right)\times\left\{ 0\right\} \right),\\
U^{+}= & P_{2}a\left(\left(-\varepsilon_{0},\varepsilon_{0}\right)\times\left(0,\varepsilon_{0}\right)\right).
\end{align*}

By steps $2$ and $3$ it is clear that $P_{2}$ restricted to $a\left(\left(-\varepsilon_{0},\varepsilon_{0}\right)\times\left(-\varepsilon_{0},\varepsilon_{0}\right)\right)$
defines a $C^{1}$-diffeomorphism from $a\left(\left(-\varepsilon_{0},\varepsilon_{0}\right)\times\left(-\varepsilon_{0},\varepsilon_{0}\right)\right)$
onto $U$. As $a\left(\left(-\varepsilon_{0},\varepsilon_{0}\right)\times\left(-\varepsilon_{0},0\right)\right)$
and $a\left(\left(-\varepsilon_{0},\varepsilon_{0}\right)\times\left(0,\varepsilon_{0}\right)\right)$
are two-dimensional $C^{1}$-submanifolds of $C$, the arcwise connected
sets $U^{-}$ and $U^{+}$ are open in $G_{2}$.

5. As $\hat{h}\left(s_{*}\right)=r_{0}\in\mathcal{O}_{r}$, we have
$U^{0}\subset P_{2}\mathcal{O}_{r}$. 

Assume that $\varphi\in\mathcal{O}_{k}$, that is $r$ is the time
translate of $x^{k}$, and $s_{*}=s_{k}$. Then the relations $\varepsilon_{0}<s_{p}-s_{k}$,
$\hat{h}\left(\left(s_{k},s_{p}\right)\right)\subset C_{k}^{p}$ and
the invariance of $C_{k}^{p}$ guarantee that $U^{+}\subset P_{2}C_{k}^{p}\subset\mbox{ext}\left(P_{2}\mathcal{O}_{k}\right)$.
As $U^{0}\subset P_{2}\mathcal{O}_{k}$ belongs to the boundary of
both connected components of $G_{2}\setminus P_{2}\mathcal{O}_{k}$,
$U^{-}$ and $U^{+}$ belong to different connected components of
$G_{2}\setminus P_{2}\mathcal{O}_{k}$. So $U^{-}\subset\mbox{int}\left(P_{2}\mathcal{O}_{k}\right)$.
Then \eqref{P_2 (closure of S_k)} implies that $U^{-}\cap P_{2}\overline{S_{k}}=\emptyset$.

In cases $\varphi\in\mathcal{O}_{p}$ and $\varphi\in\mathcal{O}_{q}$,
it is similar to show that $U^{-}\subset P_{2}C_{k}^{p}$, $U^{+}\subset P_{2}C_{q}^{p}$
and $U^{-}\subset P_{2}C_{q}^{p}$, $U^{+}\cap P_{2}\overline{S_{k}}=\emptyset$,
respectively. We omit the details.

6. Introduce the $C^{1}$-map 
\[
w_{k}^{e}:U\ni\eta\mapsto\varphi+c^{-1}\left(\eta\right)+b\left(c^{-1}\left(\eta\right)\right)-\eta\in C.
\]
For all $\eta\in U$, $c^{-1}\left(\eta\right)\in N_{\varepsilon_{0}}$,
and thus 
\begin{align*}
P_{2}\left(\varphi+c^{-1}\left(\eta\right)+b\left(c^{-1}\left(\eta\right)\right)-\eta\right) & =P_{2}\left(\varphi+c^{-1}\left(\eta\right)+b\left(c^{-1}\left(\eta\right)\right)\right)-P_{2}\eta\\
 & =c\left(c^{-1}\left(\eta\right)\right)-\eta=0.
\end{align*}
So $w_{k}^{e}$ maps $U$ into $P_{2}^{-1}\left(0\right)$.

7. It remains to confirm \eqref{w_k and w_k^e coincide}. Let $\eta\in U\cap P_{2}\overline{S_{k}}$
be arbitrary. Then $\eta=P_{2}a\left(t,s\right)=P_{2}\Phi\left(T+t,\hat{h}\left(s_{*}+s\right)\right)$
for some $t\in\left(-\varepsilon_{0},\varepsilon_{0}\right)$ and
$s\in\left(-\varepsilon_{0},\varepsilon_{0}\right)$ satisfying $s_{*}+s\in\left[s_{k},s_{q}\right]$.
As $\hat{h}\left(\left[s_{k},s_{q}\right]\right)\subset\overline{S_{k}}$
and $\overline{S_{k}}$ is invariant, $a\left(t,s\right)\in\overline{S_{k}}$.
Then due to the injectivity of $P_{2}$ on $\overline{S_{k}}$, 
\[
\eta+w_{k}\left(\eta\right)=a\left(t,s\right)
\]
follows. On the other hand, relation 
\[
\eta+w_{k}^{e}\left(\eta\right)=\varphi+c^{-1}\left(\eta\right)+b\left(c^{-1}\left(\eta\right)\right)\in a\left(\left(-\varepsilon_{0},\varepsilon_{0}\right)\times\left(-\varepsilon_{0},\varepsilon_{0}\right)\right)
\]
and the injectivity of $P_{2}$ on $a\left(\left(-\varepsilon_{0},\varepsilon_{0}\right)\times\left(-\varepsilon_{0},\varepsilon_{0}\right)\right)$
implies that 
\[
\eta+w_{k}^{e}\left(\eta\right)=a\left(t,s\right).
\]
Thus $w_{k}\left(\eta\right)=w_{k}^{e}\left(\eta\right)$.
\end{proof}
Recall from Proposition \ref{prop:graph representation for the closure of S_k},that
there exist a projection $P_{2}$ from $C$ onto a two-dimensional
subspace $G_{2}$ of $C$ and a map $w_{k}:P_{2}\overline{S_{k}}\rightarrow P_{2}^{-1}\left(0\right)$
so that 
\[
\overline{S_{k}}=\left\{ \chi+w_{k}\left(\chi\right):\,\chi\in P_{2}\overline{S_{k}}\right\} .
\]
This induces a global graph representation for any subset $W$ of
$\overline{S_{k}}$: 
\[
W=\left\{ \chi+w_{k}\left(\chi\right):\,\chi\in P_{2}W\right\} .
\]

\begin{proof}[Proof of Theorem \ref{main_theorem_4}.(i)]

We already know from Propositions \ref{prop: smooth connecting sets},
\ref{prop: smooth representation for connecting sets} and \ref{prop:hoemomorphic to the open annulus}
that the connecting sets $C_{q}^{p}$ and $C_{k}^{p}$ are two-dimensional
$C^{1}$-submanifolds of $\mathcal{W}^{u}\left(\mathcal{O}_{p}\right)$
with smooth global graph representations, furthermore $C_{q}^{p}$,
$C_{k}^{p}$ and $S_{k}$ are homeomorphic to the open annulus $A^{\left(1,2\right)}$.

As $J_{2}:\mathbb{R}^{2}\rightarrow G_{2}$ is a linear isomorphism
and $P_{2}=J_{2}\circ\pi_{2}$, Proposition \ref{prop:hoemomorphic to the open annulus}
shows that $P_{2}S_{k}$ is open in $G_{2}$. In addition, Propositions
\ref{prop: smooth representation for connecting sets} and \ref{extension of w_k}.(ii)
together give that $w_{k}$ is $C^{1}$-smooth on $P_{2}S_{k}=P_{2}\left(C_{k}^{p}\cup\mathcal{O}_{p}\cup C_{q}^{p}\right)$.
So the global graph representation 
\[
S_{k}=\left\{ \chi+w_{k}\left(\chi\right):\,\chi\in P_{2}S_{k}\right\} 
\]
given for $S_{k}$ is smooth. This property with $S_{k}\subset\mathcal{W}^{u}\left(\mathcal{O}_{p}\right)$
guarantees that $S_{k}$ is a two-dimensional $C^{1}$-submanifold
of $\mathcal{W}^{u}\left(\mathcal{O}_{p}\right)$ \cite{Lang}. 

\end{proof}

\begin{proof}[Proof of Theorem \ref{main_theorem_4}.(ii)]

Recall that Propositions \ref{prop:the structure of the closure of S_k}
and \ref{cor: closures of connecting sets} have confirmed the equalities
\[
\overline{C_{q}^{p}}=\mathcal{O}_{p}\cup C_{q}^{p}\cup\mathcal{O}_{q},\qquad\overline{C_{k}^{p}}=\mathcal{O}_{p}\cup C_{k}^{p}\cup\mathcal{O}_{k}
\]
 and
\[
\overline{S_{k}}=\mathcal{O}_{k}\cup S_{k}\cup\mathcal{O}_{q}.
\]

As $J_{2}:\mathbb{R}^{2}\rightarrow G_{2}$ is a linear isomorphism
and $P_{2}=J_{2}\circ\pi_{2}$, Proposition \ref{prop:hoemomorphic to the open annulus}
yields that $P_{2}\overline{S_{k}}$ is the closure of the open set
$P_{2}S_{k}$, and its boundary is $P_{2}\left(\mathcal{O}_{k}\cup\mathcal{O}_{q}\right).$
The sets $P_{2}\mathcal{O}_{k}$ and $P_{2}\mathcal{O}_{q}$ are the
images of simple closed $C^{1}$-curves, hence the boundary is a one-dimensional
$C^{1}$-submanifold of $G_{2}$. By the proof of Theorem \ref{main_theorem_4}.(i),
$w_{k}$ is continously differentiable on $P_{2}S_{k}$. Proposition
\ref{extension of w_k}.(i) in addition verifies that all points of
$P_{2}\left(\mathcal{O}_{k}\cup\mathcal{O}_{q}\right)$ have open
neighborhoods in $G_{2}$ on which $w_{k}$ can be extended to $C^{1}$-smooth
functions. Summing up, the representation given for $\overline{S_{k}}$
is a two-dimensional smooth global graph representation with boundary.
It is analogous to show that the induced representations of $\overline{C_{q}^{p}}$
and $\overline{C_{k}^{p}}$ are two-dimensional global graph representations
with boundary, therefore we omit the details. It follows immediately
that $\overline{C_{q}^{p}}$, $\overline{C_{k}^{p}}$ and $\overline{S_{k}}$
are two-dimensional $C^{1}$-submanifolds of $C$ with boundary \cite{Lang}.

The assertion that $\overline{C_{q}^{p}}$, $\overline{C_{k}^{p}}$
and $\overline{S_{k}}$ are homeomorphic to the closed annulus $A^{\left[1,2\right]}$
follows from Proposition \ref{prop:hoemomorphic to the open annulus}.

\end{proof}

\begin{center}
\textbf{~}
\par\end{center}

\begin{center}
5.6 \textit{$S_{1}$ and $S_{-1}$ are indeed separatrices}
\par\end{center}

\smallskip{}

To complete the proof of Theorem \ref{main_theorem_4}, it remains
to show that $S_{-1}$ and $S_{1}$ are separatrices in the sense
that $C_{2}^{p}$ is above $S_{1}$, $C_{0}^{p}$ is between $S_{-1}$
and $S_{1}$, furthermore $C_{-2}^{p}$ is below $S_{-1}$. The underlying
idea of the following proof is that the assertion restricted to a
local unstable manifold $\mathcal{W}_{loc}^{u}\left(P_{Y},p_{0}\right)$
is true, and the monotonicity of the semiflow can be used to extend
the statement for $\mathcal{W}^{u}\left(\mathcal{O}_{p}\right).$

Recall that for the periodic orbit $\mathcal{O}_{p}$, the unstable
space $C_{u}$ is two-dimensional: 
\[
C_{u}=\left\{ c_{1}v_{1}+c_{2}v_{2}:\, c_{1},c_{2}\in\mathbb{R}\right\} ,
\]
where $v_{1}$ is a positive eigenfunction corresponding the leading
real Floquet multiplier $\lambda_{1}>1$, and $v_{2}$ is an eigenfunction
corresponding the Floquet multiplier $\lambda_{2}$ with $1<\lambda_{2}<\lambda_{1}$.
Also recall that a local unstable manifold $\mathcal{W}_{loc}^{u}\left(P_{Y},p_{0}\right)$
of $P_{Y}$ at $p_{0}$ is a graph of a $C^{1}$-map: there exist
convex open neighborhoods $N_{s}$, $N_{u}$ of $\hat{0}$ in $C_{s}$,
$C_{u}$, respectively, and a $C^{1}$-map $w_{u}:N_{u}\rightarrow C_{s}$
with range in $N_{s}$ so that $w_{u}\left(\hat{0}\right)=\hat{0}$,
$Dw_{u}\left(\hat{0}\right)=0$ and 
\[
\mathcal{W}_{loc}^{u}\left(P_{Y},p_{0}\right)=\left\{ p_{0}+\chi+w_{u}\left(\chi\right):\,\chi\in N_{u}\right\} .
\]

Choose $\alpha\in\left(0,1\right)$ so small that $\left(-\alpha,\alpha\right)v_{1}+\left(-\alpha,\alpha\right)v_{2}\subset N_{u}$
and 
\begin{equation}
\sup_{\chi\in\left(-\alpha,\alpha\right)v_{1}+\left(-\alpha,\alpha\right)v_{2}}\left\Vert Dw_{u}\left(\chi\right)\right\Vert <\frac{1}{2}.\label{small derivative}
\end{equation}
Introduce the sets 
\[
A_{s}=\left\{ p_{0}+\chi+w_{u}\left(\chi\right):\,\chi\in\left(-\alpha,\alpha\right)v_{1}+sv_{2}\right\} \subset\mathcal{W}_{loc}^{u}\left(P_{Y},p_{0}\right),\quad s\in\left(-\alpha,\alpha\right).
\]

The elements of $A_{s}$, $s\in\left(-\alpha,\alpha\right)$, are
ordered pointwisely. Indeed, if $s\in\left(-\alpha,\alpha\right)$
is fixed and $a,b\in\left(-\alpha,\alpha\right)$ are arbitrary with
$a<b$, then \eqref{small derivative} implies that 
\[
\left[b-a+\int_{a}^{b}Dw_{u}\left(uv_{1}+sv_{2}\right)\mbox{d}u\right]v_{1}\gg\hat{0},
\]
 and thus 
\[
p_{0}+\left(av_{1}+sv_{2}\right)+w_{u}\left(av_{1}+sv_{2}\right)\ll p_{0}+\left(bv_{1}+sv_{2}\right)+w_{u}\left(bv_{1}+sv_{2}\right).
\]

Introduce the subsets 
\[
A_{s}^{k,+}=\left\{ \varphi\in A_{s}:\, x_{t}^{\varphi}\gg\hat{\xi_{k}}\mbox{ for some }t\geq0\right\} 
\]
and 
\[
A_{s}^{k,-}=\left\{ \varphi\in A_{s}:\, x_{t}^{\varphi}\ll\hat{\xi_{k}}\mbox{ for some }t\geq0\right\} 
\]
of $A_{s}$ for all $s\in\left(-\alpha,\alpha\right)$. Then $A_{s}^{k,+}$
and $A_{s}^{k,-}$ are open and disjoint in $A_{s}$. It is also clear
from the monotonicity of the semiflow that for any $\varphi^{-}\in A_{s}^{k,-}$
and $\varphi^{+}\in A_{s}^{k,+}$, $\varphi^{-}\ll\varphi^{+}$. 

We claim that there exists $\beta\in\left(0,\alpha\right]$ so that
$A_{s}^{k,+}$ and $A_{s}^{k,-}$ are nonempty for all $s\in\left(-\beta,\beta\right)$.
Choose 
\[
\eta_{1}=p_{0}-\frac{\alpha}{2}v_{1}+w_{u}\left(-\frac{\alpha}{2}v_{1}\right)\in A_{0}\mbox{ and }\eta_{2}=p_{0}+\frac{\alpha}{2}v_{1}+w_{u}\left(\frac{\alpha}{2}v_{1}\right)\in A_{0}.
\]
Then $\eta_{1}\ll p_{0}\ll\eta_{2}$. By Theorem 4.1 in Chapter 5
of \cite{Smith}, there is an open and dense set of initial functions
in $C$ so that the corresponding solutions converge to equilibria.
In consequence, there exist $\eta_{1}^{+},\eta_{1}^{-},\eta_{2}^{+},\eta_{2}^{-}\in C$
such that 
\[
\eta_{1}^{-}\ll\eta_{1}\ll\eta_{1}^{+}\ll p_{0}\ll\eta_{2}^{-}\ll\eta_{2}\ll\eta_{2}^{+},
\]
 and for both $i=1$ and $i=2$, $x_{t}^{\eta_{i}^{-}}$ and $x_{t}^{\eta_{i}^{+}}$
converge to one of the equilibrium points as $t\rightarrow\infty$.
Since $\max_{t\in\mathbb{R}}p\left(t\right)>\xi_{1}$, $\min_{t\in\mathbb{R}}p\left(t\right)<\xi_{-1}$
and 
\[
x_{t}^{\eta_{1}^{-}}\ll x_{t}^{\eta_{1}^{+}}\ll p_{t}\ll x_{t}^{\eta_{2}^{-}}\ll x_{t}^{\eta_{2}^{+}}\quad\mbox{for all }t\geq0
\]
 by Proposition \ref{pro:monotone_dynamical_system}, we obtain that
\[
x_{t}^{\eta_{1}^{-}}\rightarrow\hat{\xi}_{-2},\ x_{t}^{\eta_{1}^{+}}\rightarrow\hat{\xi}_{-2},\ x_{t}^{\eta_{2}^{-}}\rightarrow\hat{\xi}_{2}\mbox{ and }x_{t}^{\eta_{2}^{+}}\rightarrow\hat{\xi}_{2}\ \mbox{as}\ t\rightarrow\infty.
\]
 Using again Proposition \ref{pro:monotone_dynamical_system}, we
get that $x_{t}^{\eta_{1}}\rightarrow\hat{\xi}_{-2}$ and $x_{t}^{\eta_{2}}\rightarrow\hat{\xi}_{2}$
as $t\rightarrow\infty$, therefore $x_{t_{1}}^{\eta_{1}}\ll\hat{\xi}_{k}$
and $x_{t_{2}}^{\eta_{2}}\gg\hat{\xi}_{k}$ for some $t_{1},t_{2}\geq0$.
The continuity of the semiflow $\Phi$ implies that there exist open
balls $B_{1},B_{2}$ centered at $\eta_{1},\eta_{2}$, respectively,
such that $x_{t_{1}}^{\varphi}\ll\hat{\xi}_{k}$ for all $\varphi\in B_{1}$
and $x_{t_{2}}^{\varphi}\gg\hat{\xi}_{k}$ for all $\varphi\in B_{2}$.
It follows that there exists $\beta\in\left(0,\alpha\right]$ so that
$A_{s}^{k,+}$ and $A_{s}^{k,-}$ are nonempty for all $s\in\left(-\beta,\beta\right)$. 

Consequently, the set $A_{s}\setminus\left(A_{s}^{+}\cup A_{s}^{-}\right)$
is nonempty for all $s\in\left(-\beta,\beta\right)$, i.e., $A_{s}$
has at least one element in $S_{k}$. On the other hand, the nonordering
property of $S_{k}$ stated in Proposition \ref{prop:nonordeing of S_k}
implies that $A_{s}\cap S_{k}$ contains at most one element, i.e.,
$A_{s}\cap S_{k}$ is a singleton for all $s\in\left(-\beta,\beta\right)$.

Note that for any $s\in\left(-\beta,\beta\right),$ $\varphi^{-}\in A_{s}^{k,-}$,
$\varphi^{+}\in A_{s}^{k,+}$ and $\psi\in A_{s}\cap S_{k}$, $\varphi^{-}\ll\psi\ll\varphi^{+}$. 

Also observe that if $\left(\varphi_{n}\right)_{-\infty}^{0}$ is
a trajectory of $P_{Y}$ in $\mathcal{W}_{loc}^{u}\left(P_{Y},p_{0}\right)$
with $\varphi_{n}\rightarrow p_{0}$ as $n\rightarrow-\infty$, then
for all indeces with sufficiently large absolute value, $\varphi_{n}\in A_{s}$
for some $s\in\left(-\beta,\beta\right)$.

An element $\varphi$ of $\mathcal{W}^{u}\left(\mathcal{O}_{p}\right)$
is said to be above $S_{k}$ if $\psi\in S_{k}$ can be given with
$\psi\ll\varphi$, and $\varphi\in\mathcal{W}^{u}\left(\mathcal{O}_{p}\right)$
is said to be below $S_{k}$ if there exists $\psi\in S_{k}$ with
$\varphi\ll\psi$. An element of $\mathcal{W}^{u}\left(\mathcal{O}_{p}\right)$
is between $S_{-1}$ and $S_{1}$ if it is below $S_{1}$ and above
$S_{-1}$.

A subset $W$ of $\mathcal{W}^{u}\left(\mathcal{O}_{p}\right)$ is
above (below) $S_{k}$, if all elements of $W$ are above (below)
$S_{k}$. A subset $W$ of $\mathcal{W}^{u}\left(\mathcal{O}_{p}\right)$
is between $S_{-1}$ and $S_{1}$ if it is below $S_{1}$ and above
$S_{-1}$. 
\begin{prop}
For each $\varphi\in\mathcal{W}^{u}\left(\mathcal{O}_{p}\right)$,
exactly one of the following cases holds:

(i) $\varphi\in S_{k}$,

(ii) $\varphi$ is above $S_{k}$,

(ii) $\varphi$ is below $S_{k}$.\end{prop}
\begin{proof}
It is clear that $\varphi\in\mathcal{W}^{u}\left(\mathcal{O}_{p}\right)$
cannot be below and above $S_{k}$ at the same time because then there
would exist $\psi_{1},\psi_{2}\in S_{k}$ with $\psi_{1}\ll\varphi\ll\psi_{2}$,
which would contradict Proposition \ref{prop:nonordeing of S_k}.
For the same reason, $\varphi\in S_{k}$ cannot be above (or below)
$S_{k}$. So at most one of the above cases holds for all $\varphi\in\mathcal{W}^{u}\left(\mathcal{O}_{p}\right)$.

Let $\varphi\in\mathcal{W}^{u}\left(\mathcal{O}_{p}\right)\backslash S_{k}$
be arbitrary. By \eqref{unstable set is forwad extension of unstable manifold-1}
and the characterization of $\mathcal{W}_{loc}^{u}\left(P_{Y},p_{0}\right)$,
there exists a sequence $\left(t_{n}\right)_{-\infty}^{0}$ with $t_{n}\rightarrow-\infty$
as $n\rightarrow-\infty$ so that $\left\{ x_{t_{n}}^{\varphi}\right\} _{-\infty}^{0}$
is a trajectory of $P_{Y}$ in $\mathcal{W}_{loc}^{u}\left(P_{Y},p_{0}\right)$
and $x_{t_{n}}^{\varphi}\rightarrow p_{0}$ as $n\rightarrow-\infty$.
So an index $n^{*}\leq0$ can be given with $t_{n^{*}}<0$ such that
$x_{t_{n^{*}}}^{\varphi}\in A_{s}$ for some $s\in\left(-\beta,\beta\right)$.
Let $\psi$ denote the single element of $A_{s}\cap S_{k}$. As the
elements of $A_{s}$ are ordered pointwisely, we obtain that $x_{t_{n^{*}}}^{\varphi}\ll\psi$
or $x_{t_{n^{*}}}^{\varphi}\gg\psi$ or $x_{t_{n^{*}}}^{\varphi}=\psi$.
Observe that $x_{t_{n^{*}}}^{\varphi}=\psi$ is impossible: as $\psi\in S_{k}$
and $S_{k}$ is invariant, $x_{t_{n^{*}}}^{\varphi}=\psi$ would imply
that $\varphi=x_{-t_{n^{*}}}^{\psi}\in S_{k}$, which contradicts
the choice of $\varphi$. If $x_{t_{n^{*}}}^{\varphi}\ll\psi$, then
the invariance of $S_{k}$ and the monotonicity of the semiflow imply
that $x_{-t_{n^{*}}}^{\psi}\in S_{k}$ and $\varphi\ll x_{-t_{n^{*}}}^{\psi}$,
that is, $\varphi$ is below $S_{k}$. If $x_{t_{n^{*}}}^{\varphi}\gg\psi$,
then $\varphi\gg x_{-t_{n^{*}}}^{\psi}$ and $\varphi$ is above $S_{k}$.
\end{proof}
Now we are able to complete the proof of Theorem \ref{main_theorem_4}.

\begin{proof}[Proof of Theorem \ref{main_theorem_4}.(iii)]

1. First we show that for any $\varphi\in\mathcal{W}^{u}\left(\mathcal{O}_{p}\right)\setminus\mathcal{O}_{p}$,
$\varphi\in C_{2}^{p}$ if and only if $\varphi$ is above $S_{1}$. 

Suppose that $\varphi\in C_{2}^{p}$. Then $x_{t_{1}}^{\varphi}\gg\hat{\xi}_{1}$
for some $t_{1}>0$. Choose $t_{2}>0$ in addition so that $x_{-t_{2}}^{\varphi}\in A_{s}$
for some $s\in\left(-\beta,\beta\right)$. Necessarily $x_{-t_{2}}^{\varphi}\in A_{s}^{1,+}$,
and thereby $x_{-t_{2}}^{\varphi}\gg\psi$, where $\psi$ is the single
element of $A_{s}\cap S_{1}$. Then $x_{t_{2}}^{\psi}\in S_{1}$ and
$\varphi\gg x_{t_{2}}^{\psi}$, that is, $\varphi$ is above $S_{1}$. 

Conversely, suppose that $\varphi\in\mathcal{W}^{u}\left(\mathcal{O}_{p}\right)\setminus\mathcal{O}_{p}$
is above $S_{1}$, and choose $\psi\in S_{1}$ with $\varphi\gg\psi$.
Recall that there is an open and dense set of initial functions in
$C$ so that the corresponding solutions are convergent (Theorem 4.1
in Chapter 5 of \cite{Smith}). Hence $\eta_{1}\in C$, $\eta_{2}\in C$
and $\eta_{3}\in C$ can be given such that 
\[
\psi\ll\eta_{1}\ll\eta_{2}\ll\varphi\ll\eta_{3},
\]
furthermore $x_{t}^{\eta_{1}}$, $x_{t}^{\eta_{2}}$ and $x_{t}^{\eta_{3}}$
converge to equilibria as $t\rightarrow\infty$. By the monotonicity
of the semiflow, 
\begin{equation}
x_{t}^{\psi}\ll x_{t}^{\eta_{1}}\ll x_{t}^{\eta_{2}}\ll x_{t}^{\varphi}\ll x_{t}^{\eta_{3}}\quad\mbox{for all }t\geq0,\label{***}
\end{equation}
hence the oscillation of $x^{\psi}$ around $\hat{\xi}_{1}$ implies
that $\omega\left(\eta_{i}\right)$ is either $\left\{ \hat{\xi}_{1}\right\} $
or $\left\{ \hat{\xi}_{2}\right\} $ for all $i\in\left\{ 1,2,3\right\} $.
If $\omega\left(\eta_{2}\right)=\left\{ \hat{\xi}_{1}\right\} $,
then necessarily $\omega\left(\eta_{1}\right)=\omega\left(\eta_{2}\right)=\left\{ \hat{\xi}_{1}\right\} $,
which contradicts Proposition \ref{prop:solutions converging to xi_k}.
So $\omega\left(\eta_{2}\right)=\left\{ \hat{\xi}_{2}\right\} $.
Then \eqref{***} guarantees that $x_{t}^{\eta_{3}}\rightarrow\hat{\xi}_{2}$
and thus $x_{t}^{\varphi}\rightarrow\hat{\xi}_{2}$ as $t\rightarrow\infty$.

2. It is similar to show that for any $\varphi\in\mathcal{W}^{u}\left(\mathcal{O}_{p}\right)\setminus\mathcal{O}_{p}$,
$\varphi\in C_{-2}^{p}$ if and only if $\varphi$ is below $S_{-1}$.

3. Relations $S_{k}=C_{k}^{p}\cup\mathcal{O}_{p}\cup C_{q}^{p}$,
$k\in\left\{ -1,1\right\} $, imply the equalities $C_{q}^{p}\cup\mathcal{O}_{p}=S_{-1}\cap S_{1}$
and $C_{k}^{p}=S_{k}\backslash S_{-k}$ for both $k\in\left\{ -1,1\right\} $. 

4. It remains to verify that for $\varphi\in\mathcal{W}^{u}\left(\mathcal{O}_{p}\right)\setminus\mathcal{O}_{p}$,
$\omega\left(\varphi\right)=\left\{ \hat{0}\right\} $ if and only
if  $\varphi$ is between $S_{-1}$ and $S_{1}$. Recall that for
both $k\in\left\{ -1,1\right\} $ and each $\varphi\in\mathcal{W}^{u}\left(\mathcal{O}_{p}\right)$,
$\varphi$ is either below $S_{k}$, or it is above $S_{k}$, or it
is an element of $S_{k}$. For this reason, $\varphi\in\mathcal{W}^{u}\left(\mathcal{O}_{p}\right)\setminus\mathcal{O}_{p}$
is between $S_{-1}$ and $S_{1}$ if and only if all the following
three properties hold: $\varphi\notin S_{-1}\cup S_{1}$, $\varphi$
is not above $S_{1}$ and $\varphi$ is not below $S_{-1}$. So by
the above results, $\varphi\in\mathcal{W}^{u}\left(\mathcal{O}_{p}\right)\setminus\mathcal{O}_{p}$
is between $S_{-1}$ and $S_{1}$ if and only if 
\[
\varphi\in\mathcal{W}^{u}\left(\mathcal{O}_{p}\right)\setminus\left\{ \mathcal{O}_{p}\cup C_{-2}^{p}\cup C_{-1}^{p}\cup C_{q}^{p}\cup C_{1}^{p}\cup C_{2}^{p}\right\} =C_{0}^{p}.
\]
\end{proof}


\begin{thebibliography}{References}
\bibitem{Abraham-Robbin}Abraham R., Robbin, J. (1967). Transversal
mappings and flows. An appendix by Al Kelley, W. A. Benjamin, Inc.,
New York-Amsterdam.

\bibitem{Diekmann et al.}Diekmann, O., van Gils, S. A., Verduyn Lunel,
S. M., and Walther, H.-O. (1995). Delay equations. Functional, complex,
and nonlinear analysis, Springer-Verlag, New York.

\bibitem{Fiedler-Rocha-Wolfrum}Fiedler, B., Rocha, C., Wolfrum, M.
(2012). Sturm global attractors for S1-equivariant parabolic equations.
Networks and Heterogeneous Media 7, 617--659.

\bibitem{Hale-1}Hale, J. K. (1988). Asymptotic behavior of dissipative
systems, American Mathematical Society, Providence, RI.

\bibitem{Krisztin-1} Krisztin, T. (2008). Global dynamics of delay
differential equations. Period. Math. Hungar. 56, 83--95.

\bibitem{Krisztin-2}Krisztin, T. (2001). Unstable sets of periodic
orbits and the global attractor for delayed feedback. In Faria T.,
Freitas P. (ed.), Topics in functional differential and difference
equations, Amer. Math. Soc., Providence, RI, 267--296.

\bibitem{Krisztin-3} Krisztin, T. (2000). The unstable set of zero
and the global attractor for delayed monotone positive feedback. Discrete
Contin. Dynam. Systems, Added Volume, 229--240.

\bibitem{Krisztin-Vas}Krisztin, T. and Vas, G. (2011). Large-amplitude
periodic solutions for differential equations with delayed monotone
positive feedback. J. Dynam. Differential Equations 23, no. 4, 727\textendash{}790. 

\bibitem{Krisztin-Walther} Krisztin T., and Walther H.-O. (2001).
Unique periodic orbits for delayed positive feedback and the global
attractor. J. Dynam. Differential Equations 13, 1--57.

\bibitem{Krisztin-Walther-Wu} Krisztin T., Walter H.-O., and Wu J.
(1999). Shape, smoothness and invariant stratification of an attracting
set for delayed monotone positive feedback, Amer. Math. Soc., Providence,
RI.

\bibitem{Krisztin-Wu}Krisztin T., and Wu J. The global structure
of an attracting set. In preparation.

\bibitem{Lang}Lang, S., Differential manifolds. (1972). Addison-Wesley
Publishing Co., Inc., Reading, Mass.-London-Don Mills, Ont.

\bibitem{Lani-Wayda} Lani-Wayda B. (1995). Persistence of Poincar\'e
mappings in functional differential equations (with application to
structural stability of complicated behavior).\emph{ }J. Dynam. Differential
Equations 7, 1--71.

\bibitem{Mallet-Paret}Mallet-Paret, J. (1988). Morse decompositions
for delay-differential equations. J. Differential Equations 72, 270--315.

\bibitem{mallet-paret_nussbaum} Mallet-Paret, J., Nussbaum, R. D.
(2013). Tensor products, positive linear operators, and delay-differential
equations, Journal of Dynamics and Differential Equations 25, no.
4, 843--905. 

\bibitem{Mallet-Paret_Sell}Mallet-Paret, J., and Sell G. R. (1996).
Systems of differential delay equations: Floquet multipliers and discrete
Lyapunov Functions\emph{.} J. Differential Equations 125, 385--440.

\bibitem{Mallet-Paret_Sell2}Mallet-Paret, J., and Sell, G. R. (1996).
The Poincar\'e--Bendixson theorem for monotone cyclic feedback systems
with delay\emph{.} J. Differential Equations 125, 441--489.

\bibitem{Polner} Polner, M. (2002).\emph{ }Morse decomposition for
delay differential equations with positive feedback. Nonlinear Anal.
48, 377-397.

\bibitem{Schonflies}Sch\"onflies, A. (1908). Die Entwicklung der
Lehre von den Punktmannigfaltigkeiten. Bericht erstattet der Deutschen
Mathematiker-Vereinigung. Teil II. J.-Ber. Deutsch. Math.-Verein.,
Erg\"anzungsband II.

\bibitem{Smith}Smith, H. L. (1995). Monotone Dynamical Systems: An
Introduction to the Theory of Competitive and Cooperative Systems,
Amer. Math. Soc., Providence, RI.

\bibitem{Walther-2}Walther, H.-O. (1995). The $2$-dimensional attractor
of $x'(t)=-\lyxmathsym{\textmu}x(t)+f(x(t-1))$. Mem. Amer. Math.
Soc. 113, no. 544.\end{thebibliography}
\end{document}